\documentclass[11pt,twoside]{article}
\usepackage{fullpage}

\usepackage{epsf}
\usepackage{fancyhdr}
\usepackage{graphics}
\usepackage{graphicx}
\usepackage{psfrag}

\newcommand{\quadvar}[2]{[#1]_{#2}}

\newcommand{\gausscenter}{u}
\newcommand{\gausscenterstar}{\gausscenter^*}

\newcommand{\annulus}{\ensuremath{\mathbb{A}}}

\newcommand{\loglihoodsing}{F^S}
\newcommand{\smoothloglihood}{\tilde{\loglihood}^S}

\newcommand{\holderconst}{A}

\usepackage{microtype}
\usepackage{subfigure}
\usepackage{algorithmic}
\usepackage{color,xcolor}
\usepackage[linesnumbered,ruled]{algorithm2e}
\DontPrintSemicolon
\usepackage{color}

\usepackage{amsthm}
\usepackage{amsfonts}
\usepackage{amsmath}
\usepackage{amssymb,bbm}
\usepackage{cleveref}

\usepackage{epsf}
\usepackage{fancyhdr}
\usepackage{graphics}
\usepackage{graphicx}
\usepackage{psfrag}
\usepackage{microtype}
\usepackage{subfigure}
\usepackage{color,xcolor}
\usepackage{color}

\usepackage{amsfonts}
\usepackage{amsmath}
\usepackage{amssymb,bbm}

\newcommand{\dims}{\ensuremath{d}}

\newcommand{\real}{\ensuremath{\mathbb{R}}}

\newcommand{\thetastar}{\ensuremath{\theta^*}}
\newcommand{\thetahat}{\ensuremath{\widehat{\theta}}}
\newcommand{\thetatil}{\ensuremath{\widetilde{\theta}}}

\newcommand{\NORMAL}{\ensuremath{\mathcal{N}}}

\newcommand{\Xspace}{\ensuremath{\mathcal{X}}}

\newcommand{\brackets}[1]{\left[ #1 \right]}
\newcommand{\parenth}[1]{\left( #1 \right)}

\newcommand{\abss}[1]{\left| #1 \right |}

\newcommand{\Rspace}{\ensuremath{\mathbb{R}}}

\newcommand{\ball}{\ensuremath{\mathbb{B}}}
\newcommand{\sphere}{\ensuremath{\mathbb{S}}}

\newcommand{\normden}{\ensuremath{\phi}}

\newcommand{\sd}{\ensuremath{\sigma}}

\newcommand{\mydefn}{\ensuremath{:=}}

\newcommand{\paraspace}{\ensuremath{\Theta}}
\newcommand{\loglihood}{\ensuremath{F}}
\newcommand{\prior}{\ensuremath{\pi}}
\newcommand{\posterior}{\ensuremath{\Pi}}

\newcommand{\noise}{\ensuremath{\varepsilon}}
\newcommand{\loglihoodgaus}{\ensuremath{F^{G}}}
\newcommand{\loglihoodind}{\ensuremath{F^{I}}}
\newcommand{\loglihoodlogit}{\ensuremath{F^{R}}}
\newcommand{\noiseone}{\ensuremath{\varepsilon_{1}}}
\newcommand{\noisetwo}{\ensuremath{\varepsilon_{2}}}

\newcommand{\weakcon}{\ensuremath{\psi}}
\newcommand{\perturb}{\ensuremath{\zeta}}
\newcommand{\inver}{\ensuremath{\xi}}

\newcommand{\unicon}{c}

\newcommand{\smoothness}{L} 

\newcommand{\defn}{:=}

\newcommand{\matsnorm}[2]{|\!|\!| #1 | \! | \!|_{{#2}}}
\newcommand{\vecnorm}[2]{\left\| #1\right\|_{#2}}
\newcommand{\enorm}[1]{\vecnorm{#1}{2}} 

\newcommand{\opnorm}[1]{\ensuremath{\matsnorm{#1}{\tiny{\mbox{op}}}}}

\newcommand{\inprod}[2]{\ensuremath{\langle #1 , \, #2 \rangle}}

\newcommand{\kull}[2]{\ensuremath{D_{\text{KL}}(#1\; \| \; #2)}}

\newcommand{\Exs}{\ensuremath{{\mathbb{E}}}}
\newcommand{\Prob}{\ensuremath{{\mathbb{P}}}}

\newtheoremstyle{named}{}{}{\itshape}{}{\bfseries}{.}{.5em}{\thmnote{#3's }#1}
\theoremstyle{named}

\theoremstyle{plain}

\newtheorem{theorem}{Theorem}
\newtheorem{proposition}{Proposition}
\newtheorem{lemma}{Lemma}

\newtheorem{corollary}{Corollary}

\newlength{\widebarargwidth}
\newlength{\widebarargheight}
\newlength{\widebarargdepth}
\DeclareRobustCommand{\widebar}[1]{%
  \settowidth{\widebarargwidth}{\ensuremath{#1}}%
  \settoheight{\widebarargheight}{\ensuremath{#1}}%
  \settodepth{\widebarargdepth}{\ensuremath{#1}}%
  \addtolength{\widebarargwidth}{-0.3\widebarargheight}%
  \addtolength{\widebarargwidth}{-0.3\widebarargdepth}%
  \makebox[0pt][l]{\hspace{0.3\widebarargheight}%
    \hspace{0.3\widebarargdepth}%
    \addtolength{\widebarargheight}{0.3ex}%
    \rule[\widebarargheight]{0.95\widebarargwidth}{0.1ex}}%
  {#1}}

\makeatletter
\long\def\@makecaption#1#2{
        \vskip 0.8ex
        \setbox\@tempboxa\hbox{\small {\bf #1:} #2}
        \parindent 1.5em  
        \dimen0=\hsize
        \advance\dimen0 by -3em
        \ifdim \wd\@tempboxa >\dimen0
                \hbox to \hsize{
                        \parindent 0em
                        \hfil
                        \parbox{\dimen0}{\def\baselinestretch{0.96}\small
                                {\bf #1.} #2
                                }
                        \hfil}
        \else \hbox to \hsize{\hfil \box\@tempboxa \hfil}
        \fi
        }
\makeatother

\long\def\comment#1{}
\definecolor{battleshipgrey}{rgb}{0.52, 0.52, 0.51}
\definecolor{darkgray}{rgb}{0.66, 0.66, 0.66}
\definecolor{darkgreen}{rgb}{0.0, 0.2, 0.13}
\definecolor{darkspringgreen}{rgb}{0.09, 0.45, 0.27}
\definecolor{dukeblue}{rgb}{0.0, 0.0, 0.61}
\definecolor{olivedrab7}{rgb}{0.24, 0.2, 0.12}
\definecolor{darkblue}{rgb}{0.0, 0.0, 0.55}
\definecolor{darkscarlet}{rgb}{0.34, 0.01, 0.1}
\definecolor{candyapplered}{rgb}{1.0, 0.03, 0.0}
\definecolor{ao(english)}{rgb}{0.0, 0.5, 0.0}
\definecolor{applegreen}{rgb}{0.55, 0.71, 0.0}

\usepackage{url}

\usepackage{nicefrac}
\usepackage{comment}
\newcommand{\totalvariation}{d_{\mathrm{TV}}}
\newcommand{\Wass}{\mathcal{W}}
\usepackage{chngpage}

 \usepackage{tabularx}%

\usepackage{enumitem}
\usepackage{booktabs}
\usepackage{caption}

\usepackage{bm,bbm}
\usepackage{mathtools}
\newtheorem{assumption}{Assumption}

\newcommand{\numobs}{\ensuremath{n}}

\newcommand{\uniconprime}{\unicon'}

\newcommand{\DataX}{\ensuremath{X_1^{\numobs}}}
\newcommand{\DataZ}{\ensuremath{Z_1^{\numobs}}}
\newcommand{\usedim}{\ensuremath{d}}

\newcommand{\HessianStar}{\ensuremath{H^*}}
\newcommand{\stepsize}{\eta}

\newcommand{\offpar}{\omega}
\newcommand{\diffusionforposteriorTilde}{\widetilde{\Theta}^{(n)}}
\newcommand{\lyap}{\Phi}
\newcommand{\smalloffset}{\varsigma}

\newcommand{\globalmaxima}{\mathcal{M}^*}
\newcommand{\gap}{\Delta_0}

\newcommand{\targetdensity}{\mu}
\newcommand{\thetamap}{\widehat{\theta}^{(n)}}

\newcommand{\boundconsprior}{\ensuremath{B}}

\newcommand{\rade}{\ensuremath{\varepsilon}}

\newcommand{\Ball}{\ensuremath{\mathbb{B}}}
\newcommand{\radius}{\ensuremath{r}}
\newcommand{\RBALL}{\ensuremath{\Ball(\thetastar; \radius)}}

\newcommand{\mprob}{\ensuremath{\mathbb{P}}}

\newcommand{\Event}{\mathcal{E}}
\newcommand{\Term}{T}
\newcommand{\Ind}{\mathbb{I}}
\newcommand{\constrong}{\alpha}
\newcommand{\localradius}{r_0}
\newcommand{\lowerweakconcavemu}{\mu}
\long\def\comment#1{}

\newenvironment{carlist}
 {\begin{list}{$\bullet$}
 {\setlength{\topsep}{0in} \setlength{\partopsep}{0in}
  \setlength{\parsep}{0in} \setlength{\itemsep}{\parskip}
  \setlength{\leftmargin}{0.07in} \setlength{\rightmargin}{0.08in}
  \setlength{\listparindent}{0in} \setlength{\labelwidth}{0.08in}
  \setlength{\labelsep}{0.1in} \setlength{\itemindent}{0in}}}
 {\end{list}}

\newcommand{\bcar}{\begin{carlist}}
\newcommand{\ecar}{\end{carlist}}

\newcommand{\strongconvex}{\mu}
\newcommand{\smooth}{L_1}
\newcommand{\smoothprior}{L_2}

\newcommand{\singconstone}{c_1}

\newcommand{\simiid}{\stackrel{\mathrm{i.i.d.}}{\sim}}

\newcommand{\ellipse}{\mathfrak{E}}

\newcommand{\hackpar}{\ensuremath{\nu}}

\newcommand{\Gind}{\ensuremath{V}}

\begin{document}

\begin{center}
{\bf{\LARGE{A Diffusion Process Perspective on Posterior Contraction
      Rates for Parameters}}}

\vspace*{.2in}
 {\large{
 \begin{tabular}{ccc}
  Wenlong Mou$^{\diamond}$ & Nhat Ho$^{\star}$ &  Martin J. Wainwright$^{\diamond, \dagger, \ddagger}$ \\
 \end{tabular}
 \begin{tabular}
 {cc}
  Peter Bartlett$^{\diamond, \dagger}$ & Michael I. Jordan$^{\diamond, \dagger}$
 \end{tabular}

}}

\vspace*{.2in}

 \begin{tabular}{c}
 Department of EECS$^\diamond$,
 Department of Statistics$^\dagger$,
 UC Berkeley\\
 \end{tabular}

 \vspace*{.1in}
 \begin{tabular}{c}
 Department of Statistics and Data Science, UT Austin$^\star$\\
 \end{tabular}

 \vspace*{.1in}
 \begin{tabular}{c}
 Department of EECS, MIT$^\ddagger$
 \end{tabular}

\vspace*{.2in}

\today

\vspace*{.2in}

\begin{abstract}
We analyze the posterior contraction rates of parameters in Bayesian
models via the Langevin diffusion process, in particular by
controlling moments of the stochastic process and taking
limits. Analogous to the non-asymptotic analysis of statistical
M-estimators and stochastic optimization algorithms, our contraction
rates depend on the structure of the population log-likelihood
function, and stochastic perturbation bounds between the population
and sample log-likelihood functions. Convergence rates are determined
by a non-linear equation that relates the population-level structure
to stochastic perturbation terms, along with a term characterizing the
diffusive behavior.  Based on this technique, we also prove
non-asymptotic versions of a Bernstein-von-Mises guarantee for the
posterior. We illustrate this general theory by deriving posterior
convergence rates for various concrete examples, as well as
approximate posterior distributions computed using Langevin sampling
procedures.
\end{abstract}
\end{center}

\section{Introduction} 
\label{sec:introduction}

Bayesian inference is one of the central pillars of statistics.  In
Bayesian analysis, we first endow the parameter space with a prior
distribution chosen by modeling considerations, and then apply Bayes'
rule, combining the prior with the likelihood, so as to form the
posterior distribution.  From a statistical perspective, this
posterior is of fundamental interest, and there are various questions
associated with its behavior, including its consistency as the sample
size goes to infinity, and from a more refined point of view, its
contraction rate in various metrics.

The earliest work on posterior consistency dates back to the seminal
work of Doob~\cite{Doob-49}, who demonstrated that the posterior
distribution is consistent for all parameters apart from a set of zero
measure.  Subsequent work by Freedman~\cite{Freedman-63_first,
  Freedman-65_second} provided examples showing that this null set can
be problematic for Bayesian consistency in non-parametric settings. In
order to address this issue, Schwartz~\cite{Schwartz-65} proposed a
general framework for establishing posterior consistency for both
semiparametric and nonparametric models. Since then, a number of
researchers have isolated conditions that are useful for studying
posterior distributions~\cite{Barron-Shervish-Wasserman-99,
  Walker-2003, Walker-2004}.

Moving beyond posterior consistency, convergence rates for the
posterior density function, along with associated parameters of
models, remains an active area of research. For posterior densities,
Ghosal et al.~\cite{Ghosal-2000} gave a general testing framework for
proving convergence rates for both finite and infinite dimensional
models; it has been used by various researchers to analyzer posterior
densities for Dirichlet and nonparametric Beta
mixtures~\cite{Ghosal-2001, Ghosal-2007, Judith-2010, Shen-2013}.
Other work~\cite{Bhattacharya-2014, Yang-2015, Yang-2016} established
minimax optimal rates for regression functions in nonparametric
regression models.  Related problems include adaptive rates for the
density in nonparametric Bayesian inference~\cite{Jonge-2010,
  Gao-2016}, and posterior contraction rates of density under
misspecified models~\cite{Kleijn-2006}.  Other popular general
frameworks for analyzing the density functions of posterior
distributions include those of Shen and Wasserman~\cite{Shen-2001},
and Walker et al.~\cite{Walker-2007}.

\subsection{From frequentist to Bayesian analysis}

The focus of this paper is on posterior convergence rates for
parameters---namely, how for parametric Bayesian models, the posterior
distribution assigns mass to certain regions of the parameter space.
Our contributions can be put into perspective by considering known
results for $M$-estimators.  In the world of frequentist statistics,
estimators based on maximizing empirically-defined objective
functions---known as $M$-estimators---play a central role.  In the
parametric setting, a generic $M$-estimator takes the form
\begin{align}
\label{eq:m-estimator-in-intro-problem-setup}  
\thetahat_{\numobs} \mydefn \mathop {\arg \max}_{\theta \in \Theta}
F_\numobs(\theta) \quad \mbox{where $F_\numobs(\theta) \mydefn
  \frac{1}{\numobs} \sum_{i = 1}^\numobs f (\theta; X_i)$, with $X_i
  \simiid \Prob$ for $i = 1, \ldots, \numobs$,}
\end{align}
while the parameters $\theta$ range over some constraint set $\Theta$,
and the real-valued function $f$ has domain $\Theta \times \Xspace$.
Maximum-likelihood is the archetypal example, obtained when $f$ is the
log likelihood.

There is now a rich and well-developed theory---one which exploits
ideas from both optimization theory and empirical process theory---for
deriving sharp non-asymptotic bounds on the difference between the
estimate $\thetahat_{\numobs}$ and the maximizer $\thetastar$ of the
population-level objective (e.g., see the
books~\cite{vanderVaart-Wellner-96,Vandegeer-2000,Wainwright_nonasymptotic}).
This theory leverages properties of the population-level objective
$F(\theta) \defn \Exs [f (\theta, X)]$ where the expectation is taken
with respect to $X \sim \Prob$.  At a high level, there are two key
steps in the analysis of an $M$-estimator: exploiting the structure of
$F$, and linking the behavior of the empirical objective $F_\numobs$
to the population objective $F$. In the simplest setting, the
population objective is strongly concave around its unique maximum
$\thetastar$.  More generally, when $F$ is differentiable, one can
consider a condition of the following type
\begin{subequations}
\label{EqnIntroConditions}  
  \begin{align}
\label{EqnLocalGrowth}    
  - \inprod{\nabla F(\theta)}{\theta - \thetastar} & \geq \psi
  (\vecnorm{\theta - \thetastar}{2}),
\end{align}
assumed to hold uniformly for all $\theta$ in a local neighborhood of
$\thetastar$.  Here $\psi$ is an increasing function on the positive
real-line, with $\psi(t) = \tfrac{\mu}{2} t^2$ being the one obtained
for a $\mu$-strongly concave function.  The second step is to relate
the empirical and population objective, for instance by establishing a
uniform bound on their gradients---say
\begin{align}
\label{EqnStochPert}
  \vecnorm{\nabla F_\numobs (\theta) - \nabla F(\theta)}{2} \leq
  \zeta(\vecnorm{\theta - \thetastar}{2}) \varepsilon_\numobs,
\end{align}
\end{subequations}
where the function $\zeta$ is again defined on the positive real line,
and $\varepsilon_{\numobs}$ measures the magnitude of the noise.

When the functions $F$ and $F_\numobs$ satisfy bounds of the
form~\eqref{EqnLocalGrowth} and~\eqref{EqnStochPert}, it can be shown
that the estimate $\thetahat_\numobs$ satisfies a bound of the form
$\|\thetahat_\numobs - \thetastar\|_2 \precsim r_\numobs$, where
$r_\numobs > 0$ is the largest positive solution to the
inequality\footnote{This solution exists and is unique under mild
regularity conditions on the pair $(\psi, \zeta)$.}
\begin{align}
\label{eq:m-estimator-in-intro-rate}
\psi(r) & \leq \varepsilon_{\numobs} \, \zeta(r).
\end{align}
This framework is very convenient to use, since optimization theory
and empirical process theory give us various tools for establishing
the local growth condition~\eqref{EqnLocalGrowth} and the stochastic
perturbation bound~\eqref{EqnStochPert}.

By using this framework with care, one can often obtain sharp results
in terms of \emph{problem dimension} $d$, in both the rate itself and
sample size lower bound needed to achieve such rates. Moreover, the
local growth condition~\eqref{EqnLocalGrowth} is relatively flexible;
for instance, it allows for models in which the Fisher information
matrix is singular (so that the function $\psi$ is \emph{not}
quadratic).  There are many different instantiations of this general
approach in past work, including various methods or establishing
growth conditions and empirical process
bounds~\cite{spokoiny2012parametric,ostrovskii2021finite}, analysis of
iterative optimization
algorithm~\cite{Siva_2017,Raaz_Ho_Koulik_2018,ma2020implicit,ho2020instability},
as well as regularized and constrained
$M$-estimators~\cite{loh2013regularized,chretien2021learning}.


\subsection{Our contributions}

Moving back to the Bayesian setup, it is natural to seek to a
similarly flexible and user-friendly method for establishing
finite-sample results for posterior contraction.  The main
contribution of this paper is do so by using the Langevin diffusion
process---a stochastic differential equation that can encode the
posterior distribution---as a lens of analysis.

There are natural parallels between our mode of analysis, and
deterministic analyses of optimization algorithms via differential
equations~\cite{su2016differential,shi2021understanding}.  To provide such
intuition, recall the $M$-estimator defined by the objective
function~\eqref{eq:m-estimator-in-intro-problem-setup}.  Under the
given conditions, its optimum $\thetastar$ can be characterized as the
limiting point of an \emph{ordinary differential equation} known as
the gradient flow, and the rate~\eqref{eq:m-estimator-in-intro-rate}
via the gradient flow dynamics for population and empirical loss
functions, respectively.  Now consider the analogous approach for
studying \emph{not} the $M$-estimator, but rather (in the Bayesian
set-up) the posterior distribution.  It is
well-known~\cite{risken1996fokker} that under mild regularity
conditions, the posterior distribution can be represented as the
stationary distribution of a \emph{stochastic differential equation}
known as the Langevin diffusion.  Consequently, just as information
about the $M$-estimator can be recovered by studying the gradient
flow, we can recover information about the posterior distribution by
studying the Langevin diffusion.  In particular, we do so by
leveraging stochastic calculus so as to control the moments of this
diffusion process.  At a high-level, our main results involving
showing that, under assumptions of the
form~\eqref{EqnIntroConditions}, the posterior convergence rate is
governed by the inequality $\psi(r) \leq \varepsilon_{\numobs}
\zeta(r) + \frac{d}{\numobs}$.  By comparison to
inequality~\eqref{eq:m-estimator-in-intro-rate}, relevant for
$M$-estimation, we see that this inequality includes an additional
$\tfrac{d}{\numobs}$ term: it characterizes the diffusive behavior
(with dimension $d$ and sample size $\numobs$) induced from sampling
from the Gibbs measure $e^{- F_\numobs}$ as opposed to taking its
maximum.

With this overview in place, we now summarize the different classes of
contributions that are made in this paper:
\paragraph{Globally concave problems:} We begin with the simplest
setting, in which the population log-likelihood function is strongly
concave in a global sense.  Under certain regularity
conditions,\footnote{Briefly, we require the prior distribution to be
sufficiently smooth and the perturbation error between the population
and empirical log-likelihood function to be well-controlled.}  we
prove that the posterior contraction rate around the true parameter is
$(d/ n)^{1/ 2}$.  Our technique allows us to specify precise
non-asymptotic conditions on the sample size and other model
properties under which a guarantee of this type holds.  We then relax
our assumption from strongly concave to (weakly) concave, and prove
related guarantees.  We illustrate these general results for three
concrete classes of models: Bayesian non-linear regression models,
over-specified Bayesian location Gaussian mixture models, and Bayesian
logistic regression models.  Our theory reveals the influence of
different modeling assumptions on the behavior of the posterior.

\paragraph{From global to local concavity:}
In order to extend the scope of our theory, we next relax the global
nature of our conditions.  We study posterior contraction when the
population log-likelihood function $F$ is only locally concave in a
ball around $\thetastar$, thereby allowing for multi-modality.  In
this setting, we find key properties that govern the convergence rate:
the rate of growth of the population log-likelihood, and the
deviations between the gradients of the sample and population
log-likehoods.  In particular, consider a log-likelihood function such
that
  \begin{align*}
 -\inprod{\nabla F (\theta)}{\theta - \thetastar} \gtrsim
 \vecnorm{\theta - \thetastar}{2}^{\alpha + 1}, \quad \mbox{and} \quad
 \vecnorm{\nabla F_\numobs (\theta) - \nabla F (\theta) }{2} \lesssim
 \vecnorm{\theta - \thetastar}{2}^\beta \cdot \sqrt{d / n}
  \end{align*}
  for some positive values of $\alpha$ and $\beta$ with $\alpha >
  \beta$.  Our theory guarantees that the posterior convergence rate
  of parameters is given by $O \big( (d / \numobs)^{\min \{\frac{1}{1
      + \alpha} , \frac{1}{2 (\alpha - \beta)}\}} \big)$. This result
  not only recovers the classical results when the Fisher information
  is non-singular---i.e., when $\alpha = 1$ and $\beta = 0$---in a
  non-asymptotic way for a suitable range of $\numobs$, but also
  applies to a broad class of models with singular Fisher
  information---i.e., for which $\alpha > 1$ and $\beta \geq 0$. The
  proof relies on the similar diffusion process considered in the
  globally concave settings, with a modified version of the potential
  function that exhibits the same local behavior as the empirical
  log-likelihood function.

\paragraph{Guarantees for approximate posteriors computed via Langevin algorithms:}

By adapting the continuous-time arguments to a discrete-time setting,
we show contraction rate bounds for the output of the unadjusted
Langevin algorithm. Working with the local strongly convex setting, we
show that the output of Langevin algorithm satisfies contraction
bounds that (up to logarithmic factors) match the optimal posterior
contraction behavior. Compared to existing works, our result does not
put stringent assumptions on the stepsize, allowing for faster
convergence of the algorithm.
   
\paragraph{Non-asymptotic Bernstein-von-Mises (BvM) results:} Our
  final contribution is to establish two non-asymptotic BvM results
  for models with non-degenerate Fisher information. For the first
  result, we derive a non-asymptotic upper bound on the
  Kullback-Leibler (KL) divergence between the posterior distribution
  and the limiting Gaussian distribution with mean given by maximum a
  posteriori (MAP) estimate, and covariance matrix by the inverse of
  Hessian matrix of the population log-likelihood function. This bound
  scales at the order $\mathcal{O}(1/n)$ in terms of the sample size
  $\numobs$.  Second, we prove non-asymptotic tail bounds that are
  satisfied by the posterior distribution; those bounds almost match
  the tail bounds that are satisfied by the limiting Gaussian law, up
  to high-order terms. In particular, we show that the posterior mass
  concentrates within an ellipsoid whose shape is determined by the
  Hessian matrix of the population log-likelihood at $\thetastar$. We
  note that the diffusion process approach plays a central role in
  this proof: in particular, a key technical ingredient is an error
  estimate between the underlying diffusion process and an
  Ornstein-Uhlenbeck (OU) process, whose stationary distribution is
  the limiting Gaussian law in BvM theorems.

The remainder of the paper is organized as follows.  In
\Cref{sec:setup}, we set up the basic framework for Bayesian models
and introduce a diffusion process that admits posterior distribution
as its stationary distribution.  \Cref{sec:posterior_contrac} is
devoted to establishing the general results for posterior convergence
rates of parameters under various assumptions on the global concavity
of the population log-likelihood. We then study these convergence
rates under the locally concave settings of the population
log-likelihood function in \Cref{subsec:contrac_local_assumption}.
\Cref{subsec:nonasymp_BVM} is devoted to non-asymptotic BvM results
for models with non-degenerate Fisher information.  We discuss an
application of these general theories to Bayesian logistic regression
and Gaussian mixture models in \Cref{sec:examples} and other
statistical models in \Cref{sec:app-additional-examples}. We conclude
our work with a discussion in \Cref{sec:discussion} while proofs of
results in the paper are in the supplementary
material~\cite{mou2019supplementary}.


\paragraph{Notation.} In the paper, the expression $a_{n} \succsim
b_{n}$ will be used to denote $a_{n} \geq c b_{n}$ for some positive
universal constant $c$ that does not change with $n$. Additionally, we
write $a_{n} \asymp b_{n}$ if both $a_{n} \succsim b_{n}$ and $a_{n}
\precsim b_{n}$ hold. For any $n \in \mathbb{N}$, we denote $[n] =
\{1, 2, \ldots, n \}$. The notation $\sphere^{d - 1}$ stands for the
unit sphere, namely, the set of vectors $u \in \Rspace^{d}$ such that
$\enorm{ u} = 1$. For any subset $\Theta$ of $\Rspace^{d}$, $r \geq
1$, and $\varepsilon > 0$, we denote $\mathcal{N}(\varepsilon, \Theta,
\|.\|_{r})$ the covering number of $\Theta$ under $\|.\|_{r}$ norm,
namely, the minimum number of $\varepsilon$-balls under $\|.\|_{r}$
norm to cover the entire set $\Theta$. Given a positive-definite
matrix $M \succ 0$, we use $\lambda_{\max} (M)$ and $\lambda_{\min}
(M)$ to denote its largest and smallest eigenvalue, respectively, and
we use $\kappa (M) \mydefn \lambda_{\max} (M) / \lambda_{\min} (M)$ to
denote its condition number. Finally, for any $x , y \in \Rspace$, we
denote $x \vee y = \max \{x, y \}$ and $x \wedge y = \min \{x, y\}$.


\section{Background and problem formulation}
\label{sec:setup}

This section is devoted to background material along with formulation
of the problems studied in this paper.  We first set up the problem of
studying convergence rates for posterior distributions over parameters
in \Cref{subsec:posterior_parameter}, and provide background on its
representation as the stationary distribution of a Langevin diffusion
process in in \Cref{subsec:diffusion_posterior}. Finally, we define
the population likelihood function, and introduce various smoothness
conditions in \Cref{subsec:empirical_population}.


\subsection{Posterior contraction rates for parameters}
\label{subsec:posterior_parameter}

Consider a parametric family of distributions $\{P_\theta \mid \theta
\in \paraspace\}$.  Throughout the paper, we assume that each
distribution $P_\theta$ has density $p_\theta$ with respect to the
Lebesgue measure.  Let $\DataX \mydefn (X_1, \ldots, X_\numobs)$ be a
sequence of random variables drawn $\mathrm{i.i.d.}$ from
$P_{\thetastar}$, where $\thetastar \in \paraspace$ is the true
parameter, albeit unknown.  Given a prior $\prior$ over the parameter
space, we define the the log-likelihood
\begin{align}
\label{EqnEmpiricalLike}  
 \loglihood_{n}(\theta) \mydefn \frac{1}{n} \sum_{i = 1}^n \log
 p_\theta (X_i), \quad \mbox{along with the posterior} \quad
 \posterior \left(\theta \mid \DataX \right) \mydefn \tfrac{ e^{n
     F_{n}( \theta)} \prior( \theta)} {\int_{\paraspace} e^{n F_{n}(
     u)} \prior( u) du}.
\end{align}
As the sample size $\numobs$ increases, we expect that the posterior
distribution will concentrate more of its mass over increasingly
smaller neighborhoods of the true parameter $\thetastar$.  Posterior
contraction rates allow us to study how quickly this concentration of
mass takes place.  In particular, for a given norm, we study the
posterior mass of a ball of the form $\|\theta - \thetastar\| \leq
\rho$ for a suitably chosen radius $\rho > 0$.  For a given $\delta
\in (0,1)$, our goal is to prove statements of the form $\posterior
\big( \|\theta - \thetastar\| \geq \rho(\numobs, \usedim, \delta) \mid
\DataX \big) \leq \delta$, with probability at least $1 - \delta$ over
the randomly drawn data $\DataX$.  Our interest is in the scaling of
the radius $\rho(\numobs, \usedim, \delta)$ as a function of sample
size $\numobs$, problem dimension $\usedim$, and the error tolerance
$\delta$, as well as other problem-specific parameters.


\subsection{From diffusion processes to the posterior distribution}
\label{subsec:diffusion_posterior}

The analysis of this paper relies on a well-known connection between
the posterior distribution and a particular stochastic differential
equation (SDE) known as the Langevin diffusion.  For a parameter
$\beta > 0$, the Langevin diffusion can be written as
\begin{align}
  \label{EqnLangevinSDE}
  d \theta_t = - \nabla U(\theta_t) dt + \sqrt{\tfrac{2}{\beta}} \; d
  B_t,
\end{align}
where $(B_t, t \geq 0)$ is a standard $\usedim$-dimensional Brownian
motion~\cite{MR1725357}, and $U: \real^\usedim \rightarrow \real$ is
known as the potential function.  Suppose that we impose the following
regularity conditions on the potential: (a) its gradient $\nabla U$ is
locally Lipschitz, and (b) its gradient satisfies the inequality
$\inprod{\nabla U (\theta)}{\theta} \geq c_1 \vecnorm{\theta}{2} -
c_2$ \mbox{for any $\theta \in \real^d$,} for some strictly positive
constants $c_1, c_2$.  Under these conditions, by known results on
general Langevin diffusions~\cite{bakry2008simple}, the solution to
the Langevin diffusion~\eqref{EqnLangevinSDE} exists and is unique in
the strong sense.  Furthermore, the density of $\theta_t$ converges in
$\mathbb{L}^2$ to the stationary distribution with density
proportional to $e^{- \beta U}$.

In the context of Bayesian inference, we can apply this argument to
the potential function $U_n(\theta) \mydefn - n \loglihood_n(\theta) -
\log \prior(\theta)$.  Doing so will require us to verify that $U_n$
satisfies the requisite regularity conditions.  Assuming this
validity, we are guaranteed that the posterior distribution
$\posterior (\theta \mid \DataX)$ is the stationary distribution of
the SDE
\begin{align}
 \label{eq-diffusion-main}
  d \theta_t = \tfrac{1}{2} \nabla \loglihood_{n} ( \theta_t) dt +
  \tfrac{1}{2n} \nabla \log \prior (\theta_t ) dt +
  \tfrac{1}{\sqrt{n}} dB_t,
\end{align}
with initial condition $\theta_0 = \thetastar$. Moreover, the density
of $\theta_t$ converges in $\mathbb{L}^2$ to the posterior density.

It should be noted that this SDE-based representation of the posterior
underlies various algorithms for drawing samples from the posterior
distribution; we refer the reader to the
papers~\cite{dalalyan2017theoretical,durmus2017nonasymptotic,durmus2019high}
for some recent state-of-the-art results in this direction.  In this
paper, we exploit this SDE-based representation for statistical
analysis (as opposed to efficient computation).  In particular, by
characterizing the behavior of the process $(\theta_t, t \geq 0)$ as a
function of time, we can obtain bounds on the posterior distribution
by taking limits. The following proposition guarantees the convergence
of the moments based on a uniform-in-time moment upper bound and a
convergence in total variation distance.
\begin{proposition}
\label{prop:keybound}
Consider a sequence of distributions $(\pi_t)_{t \geq 0}$ on $\real^d$
such that $\totalvariation (\pi_t, \pi^*) \rightarrow 0$, and suppose
that $\sup_{t \geq 0} \Exs_{\pi_t} \left[ \vecnorm{X}{2}^p \right] < +
\infty$ and $ \Exs_{\pi^*} \left[ \vecnorm{X}{2}^p \right] < + \infty$
for any even integer $p \geq 2$.  We then have $\lim \limits_{t
  \rightarrow + \infty} \Exs_{\pi_t} \left[ \vecnorm{X}{2}^p \right] =
\Exs_{\pi^*} \left[ \vecnorm{X}{2}^p \right]$.
\end{proposition}
\noindent See \Cref{subsec:proof-prop-keybound} in our
supplementary material~\cite{mou2019supplementary} for the proof of
this proposition.

Given this limiting behavior, we can establish posterior contraction
rates for the parameters by controlling the moments of the diffusion
process $\{\theta_t\}_{t \geq 0}$.  The main theoretical results of
this paper are obtained by following this general roadmap.


\subsection{From empirical to population likelihood}
\label{subsec:empirical_population}

Before proceeding to our main results, let us introduce some
additional definitions and conditions.  A useful notion for our
analysis is the population log-likelihood $\loglihood$.  It
corresponds to the limit of log-likelihood function $\loglihood_{n}$,
as previously defined in equation~\eqref{EqnEmpiricalLike}, as the
sample size $n$ goes to infinity---viz.
\begin{align}
\label{EqnPopulationLike}
  \loglihood (\theta) \mydefn \mathbb{E} \brackets{ \log p_\theta
    (X)},
\end{align}
where the expectation is taken with respect to $X \sim
P_{\thetastar}$.  Throughout the paper, we impose the following
smoothness conditions on the population log-likelihood $\loglihood$
and the log prior density $\log \prior$:
\begin{enumerate}[label={\bf{(\Alph*)}}]
\item\label{item:smooth_population_condition} There exist positive
  constants $\smooth$ and $\smoothprior$ such that for any $\theta_1,
  \theta_2 \in \real^\dims$, we have
  \begin{align*}
    \vecnorm{ \nabla \loglihood(\theta_1) - \nabla
      \loglihood(\theta_2) }{2} \leq \smooth \vecnorm{\theta_1 -
      \theta_2}{2}, \quad \mbox{and} \quad \vecnorm{ \nabla \log
      \prior (\theta_1) - \nabla \log \prior (\theta_2) }{2} \leq
    \smoothprior \vecnorm{\theta_1 - \theta_2}{2}.
  \end{align*}

\item\label{item:smooth_log_prior} There exists a non-negative
  constant $ \boundconsprior \geq 0$ such that
  \begin{align*}
    \inprod{\nabla \log \prior (\theta)}{\theta - \thetastar} \leq
    \boundconsprior \vecnorm{\theta - \thetastar}{2} \qquad \mbox{for
      all $\theta \in \Rspace^{d}$.}
    \end{align*}
\end{enumerate}

Although the constant $\boundconsprior$ in
Assumption~\ref{item:smooth_log_prior} can depend on $\thetastar$, we
suppress this dependence so as to keep the notation streamlined.  When
the function $\log \prior$ is globally Lipschitz (so that $\|\nabla
\log \prior(\theta)\|_2$ is uniformly bounded),
Assumption~\ref{item:smooth_log_prior} is automatically satisfied, but
it only requires a one-sided control, allowing for important examples
such as Gaussian prior.

The above conditions are relatively mild, and we provide a number of
examples in the sequel for which they are satisfied.


\section{Results under global conditions}
\label{sec:posterior_contrac}

We now turn to our first set of results, which provide bounds on
posterior contraction rates under global concavity conditions on the
population log-likelihood function.  Results under milder local
conditions are given in \Cref{sec:without_global_concavity} to follow.

In \Cref{subsec:strongly_convex}, we present a result
(~\cref{thm-one-point-SC}) that establishes the posterior convergence
under strong concavity.  \Cref{subsec:weakly_convex} answers the same
question when the population log-likelihood is only weakly concave;
see the statement of \cref{theorem-main-weakly-convex}.


\subsection{Posterior contraction under strong concavity}
\label{subsec:strongly_convex}

We begin with results under strong concavity conditions.  For this
part, the following assumptions underlie our analysis:
\begin{enumerate}[label={\bf{(S.\arabic*)}}]
\item
\label{item:strong_concavity_population}
There exists a scalar $\strongconvex > 0$ such that
  \begin{align*}
   - \inprod{ \nabla \loglihood(\theta)}{ \thetastar - \theta } \geq
   \strongconvex \vecnorm{ \theta - \thetastar }{2}^2 \quad \mbox{for
     any $\theta \in \real^\dims$.}
  \end{align*}
\item \label{item:strong_concavity_deviation}
There exist non-negative functions $\noiseone$ and $\noisetwo$ that
map from $\mathbb{N} \times (0,1]$ to $\real_{+}$ such that for any
  radius $r > 0$ and any $\delta \in (0,1)$, we have
\begin{align*}
  \sup_{\theta \in \ball (\thetastar, r)} \vecnorm{\nabla
    \loglihood_{n} (\theta) - \nabla \loglihood(\theta) }{2} \leq
  \noiseone(n, \delta) r + \noisetwo(n, \delta) \qquad \mbox{with
    prob. at least $1 - \delta$.}
\end{align*}

\end{enumerate}

Assumption~\ref{item:strong_concavity_population} is a standard strong
concavity condition of function $\loglihood$ around $\thetastar$,
whereas Assumption~\ref{item:strong_concavity_deviation} provides
uniform control on the gradients of the population and sample
log-likelihoods.  It is important to note that these assumptions,
along with other assumptions to follow, \emph{do not} require the
data-generating distribution $P$ to belong to the specified
parameteric class.  Indeed, the results throughout this paper apply to
both well-specified and mis-specified models. In the latter case, the
parameter $\thetastar$ is typically the KL-projection of the true
model, i.e., $\thetastar \in \arg \min_{\theta \in \Theta}
\kull{\Prob}{\Prob_{\theta}}$.

Given the above assumptions, we are ready to state our first result
regarding the posterior convergence rate of parameters for a strongly
concave population log likelihood:
\begin{theorem}
\label{thm-one-point-SC}
Suppose that
Assumptions~\ref{item:smooth_population_condition},~\ref{item:smooth_log_prior},~\ref{item:strong_concavity_population},
and~\ref{item:strong_concavity_deviation} hold.  Then there is a
universal constant $\unicon$ such that for any $\delta \in (0,1)$ and
any sample size $n$ for which \mbox{$\noiseone(n, \delta) \leq
  \frac{\mu}{6}$,} we have
\begin{align*}
  \posterior \Big( \vecnorm{\theta - \thetastar}{2} \geq \unicon
  \sqrt{ \tfrac{d}{\numobs \strongconvex}} + \tfrac{\boundconsprior}{n
    \strongconvex} + \tfrac{ \noisetwo(n, \delta)}{\strongconvex} +
  \unicon \sqrt{\tfrac{\log(1/ \delta)}{\numobs \strongconvex}} \;
  \Big| \; \DataX \Big) & \leq \delta
\end{align*}
with probability $1 - \delta$, taken with respect to the random
observations $\DataX$.
\end{theorem}
\noindent
See \cref{subsec:strongly_convex_proof} for the proof of
\cref{thm-one-point-SC}.

This result guarantees posterior convergence at the rate $(d/ n)^{1/
  2}$ when the log likelihood is strongly concave.  To be clear, such
rate of posterior contraction for the parameters can be derived from
the asymptotic behavior of the posterior distribution via the
classical Bernstein-von-Mises theorem.  However, the guarantee in
\cref{thm-one-point-SC} is non-asymptotic, and provides
explicit dependence of the rate on other model parameters, including
$B$ and $\mu$, both of which might vary as a function of
$\thetastar$. At the moment, we do not know whether the dependence of
these parameters is optimal.  This guarantee is valid as long as the
error term $\noiseone (\numobs, \delta)$ is less than an absolute
constant; such a bound typically holds as long as $\numobs \gtrsim
\usedim$.  In \cref{thm:non-asymp-credible-set} to follow, we
also provide near-optimal non-asymptotic contraction bounds on the
posterior distribution that nearly match the exact shape of the
posterior distribution.

Although our set-up is focused on simple sampling models, it should be
noted that our method is sufficiently flexible so as to accommodate
certain non-$\mathrm{i.i.d.}$ forms of sampling, along with
mis-specified models. After the first version was posted, Mazumdar et
al.~\cite{mazumdar2020approximate} used a variant of this result to
study the posterior contraction rates for Thompson sampling in
contextual bandits. In their problem, the data are adaptively
collected instead of being $\mathrm{i.i.d.}$, and the empirical
process bound~\ref{item:strong_concavity_deviation} can be verified
using martingale concentration inequalities.

\vspace{0.5 em}
\noindent
\textit{Proof overview:} As described in our motivating introduction,
the proof of \cref{thm-one-point-SC} is based on analyzing the
Langevin diffusion $(\theta_t)_{t \geq 0}$ from
equation~\eqref{eq-diffusion-main}. The key idea---one which plays a
key role in the proofs throughout the entire paper---is the use of a
Lyapunov function $\lyap_t$. In particular, we use It\^{o} calculus to
track the growth of $\lyap_t$ over time $t$.  By taking $t \rightarrow
+ \infty$, the bounds on the Lyapunov function carry over to the
stationary distribution.

In more detail, we prove \cref{thm-one-point-SC} using the Lyapunov
function $\lyap_t \mydefn \tfrac{1}{2} e^{\frac{\strongconvex t}{2}}
\vecnorm{\theta_t - \thetastar}{2}^2$ and bounding the moments of this
stochastic process. Some calculation leads to the upper bound
\begin{align*}
 e^{\frac{\strongconvex t}{2}} \vecnorm{\theta_t - \thetastar}{2}^2
 \leq \tfrac{1}{\sqrt{\numobs}} \int_0^t e^{\strongconvex t / 2}
 \inprod{\theta_s - \thetastar}{d B_s} + c\left(
 \tfrac{\usedim}{\numobs} + \frac{\noisetwo^2 (n,
   \delta)}{\strongconvex} + \tfrac{\boundconsprior}{\numobs^2}
 \right) \cdot \tfrac{e^{\strongconvex t / 2}}{\strongconvex}.
\end{align*}
The last term is deterministic, and gives rise to the terms
$\sqrt{\frac{d}{\mu n}} + \frac{\noisetwo (\numobs, \delta)}{\mu} +
\frac{\boundconsprior}{\numobs \mu}$ in
\cref{thm-one-point-SC}.  Taking expectations on both sides of
the bound yields a non-asymptotic bound on the second moment of the
posterior distribution.  In order to provide a high probability bound,
as stated in the claim, we control the martingale term by invoking the
Burkholder-Davis-Gundy (BDG) inequality for continuous-time
martingales; doing so produces the term
$\sqrt{\tfrac{\log(1/\delta)}{\numobs}}$ in our bound.  The full proof
is given in \cref{subsec:strongly_convex_proof}.

\subsection{Posterior contraction under weak concavity}
\label{subsec:weakly_convex}

~\cref{thm-one-point-SC} requires global strong concavity, which
is relatively strong. In this section, we relax this assumption in two
ways: we relax the growth condition locally around $\thetastar$ so as
to allow for weak concavity, and the global behavior need not coincide
with this local behavior.  Weakly concave log-likelihoods arise for
singular problems, for which the Fisher information matrix at the true
parameter $\thetastar$ is rank-degenerate. Examples of such singular
problems include Bayesian non-linear regression models with certain
choices of link functions~\cite{McCullagh_generalized}, as well as
over-specified mixture models~\cite{Rousseau-2011}, in which the
fitted mixture model has more components than the true mixture
distribution. The mismatch between local and global concavity
conditions exists not only in such models, but also in non-singular
problems such as Bayesian logistic regression. We discuss implications
of these examples in the supplementary
material~\cite{mou2019supplementary}. Note that the results in this
section still require the global maximum $\thetastar$ to be unique, so
that the posterior is unimodal. This requirement is removed in the
analysis of the next section. \\

Our analysis in the weakly concave setting is based on the following
assumptions: \\

\begin{enumerate}[label={\bf{(W.1)}}]
\item\label{item:weak_concavity_population} There exists a convex,
  non-decreasing function $\weakcon: [0, +\infty) \rightarrow \real$
    such that 
  \begin{align*}
    - \inprod{ \nabla \loglihood(\theta)}{ \theta - \thetastar } \geq
    \weakcon( \vecnorm{ \theta - \thetastar }{2} ) \qquad \mbox{for
      any $\theta \in \real^\dims$.}
  \end{align*}
\end{enumerate}
Assumption~\ref{item:weak_concavity_population} characterizes the weak
concavity of the function $\loglihood$ around the global maxima
$\thetastar$.  This condition can hold when the log likelihood is
locally strongly concave around $\thetastar$ but only weakly concave
in a global sense, or it can hold when the log likelihood is weakly
concave but nowhere strongly concave.  An example of the former type
is the logistic regression model analyzed
in~\Cref{sec:logistic_regress}, whereas an example of the latter type
is given by certain kinds of non-linear regression models, as analyzed
in~\Cref{sec:single_index}.

Our next assumption controls the deviation between the gradients of
the population and sample likelihoods, and involves a failure
probability $\delta \in (0,1)$:
\begin{enumerate}[label={\bf{(W.2)}}]
\item\label{item:weak_concavity_deviation} There exist a function
  $\noise: \mathbb{N} \times (0,1] \mapsto \real_{+}$ and a
  non-decreasing function $\perturb: \real \rightarrow \real$ with
  that $\perturb (0) \geq 0$ such that for any radius $r > 0$, we
  \begin{align*}
    \sup_{\theta \in \ball (\thetastar, r)} \vecnorm{\nabla
      \loglihood_{n} (\theta) - \nabla \loglihood ( \theta) }{2} \leq
    \noise(n, \delta) \perturb(r)  \qquad
  \mbox{with prob. at least $1 - \delta$.}    
  \end{align*}
\end{enumerate}
Note that the function $\perturb$ can depend on the sample size $n$
and other model parameters; such dependence arises in our analysis of
over-specified Bayesian mixture model given in
\cref{sec:over_Gaussian_mixture}.  In this main text, we suppress this
dependence so as to keep the notation streamlined.

The previous conditions involved two functions, namely $\weakcon$ and
$\perturb$. We let $\inver : \real_+ \rightarrow \real$ denote the
inverse function of the strictly increasing function $r \mapsto r
\perturb (r)$. Our third assumption imposes certain inequalities on
these functions and their derivatives:
\begin{enumerate}[label={\bf{(W.3)}}]
\item
  \label{item:weak_concavity_growth_conditions}
The function $r \mapsto \weakcon( \inver (r))$ is convex, and
$\weakcon$ and $\perturb$ satisfy the differential inequalities
\begin{align*}
r \weakcon'(r) \perturb (r) & \stackrel{(i)}{\geq} r \weakcon (r)
\perturb' (r) + \weakcon (r) \perturb (r), \quad \mbox{and} \\
r^2 \weakcon'' (r) \perturb (r) + r \weakcon'(r) \perturb (r) &
\stackrel{(ii)}{\geq} 3 \weakcon(r) \perturb(r) + r^2 \weakcon(r)
\perturb''(r) \quad \mbox{for all $r > 0$.}
\end{align*}

\end{enumerate}
\noindent These differential inequalities are needed controlling the
moments of the diffusion process $\{\theta_{t}\}_{t > 0}$ in
equation~\eqref{eq-diffusion-main}.  In our discussion of concrete
examples, we provide instances for which they are satisfied.

Our result involves a certain fixed point equation that depends on the
parameters and functions in our assumptions.  In particular, for any
tolerance parameter \mbox{$\delta \in (0,1)$} and \mbox{sample size
  $n$,} consider the following fixed point equation in the variable $z
> 0$:
\begin{align}
    \label{eq:key_equation}
    \psi (z) = \noise(n, \delta) \zeta (z) z +
    \tfrac{\boundconsprior}{\numobs} z + \tfrac{d}{\numobs} +
    \tfrac{\log(1/\delta)}{\numobs}.
\end{align}
In order to ensure that this equation has a unique positive solution,
our final assumption imposes certain condition on the growth of the
functions $\weakcon$ and $\perturb$:\\

\begin{enumerate}[label={\bf{(W.4)}}] 
\item\label{item:weak_concavity_tail} The limit $\lim \inf \limits_{z
  \rightarrow + \infty} \frac{\psi (z)}{ z \zeta (z)}$ is strictly
  positive, and the sample size $n$ and tolerance parameter $\delta
  \in (0, 1)$ are such that \mbox{$\varepsilon (n, \delta) < \lim \inf
    \limits_{z \rightarrow + \infty} \frac{\psi (z)}{ z \zeta (z)}$.}
\end{enumerate}

\medskip 
\noindent With this set-up, we are now ready to state our second main
result:
\begin{theorem}
\label{theorem-main-weakly-convex}
Suppose that Assumptions~\ref{item:smooth_population_condition},
~\ref{item:smooth_log_prior},
and~\ref{item:weak_concavity_population}---~\ref{item:weak_concavity_growth_conditions}
hold.  Then for any given sample size $n$ and $\delta \in (0,1)$ such
that Assumption~\ref{item:weak_concavity_tail} holds,
equation~\eqref{eq:key_equation} has a unique positive solution
$z^*(n, \delta)$ such that
  \begin{align}
    \label{EqnWeaklyConvexPosterior}
    \posterior \Big( \vecnorm{\theta - \thetastar}{2} \geq z^*(n,
    \delta) \; \mid \; \DataX \Big) \leq \delta \quad \mbox{with
      probability $1 - \delta$ w.r.t. $\DataX$.}
\end{align}
\end{theorem}
\noindent See \cref{subsection:weakly_convex_proof} for the proof of
\cref{theorem-main-weakly-convex}.

A few comments are in order. First, the convergence
guarantee~\eqref{EqnWeaklyConvexPosterior} depends on the weak
convexity function $\weakcon$ and the perturbation function $\perturb$
through the non-linear equation~\eqref{eq:key_equation}.  See the
proof sketch below for the origins of this equation. Second, at least
in general, it is not possible to compute an explicit form for the
positive solution $z^{*}(n, \delta)$ to the non-linear
equation~\eqref{eq:key_equation}.  However, for certain forms of the
function $\weakcon$ and $\perturb$, we can derive a relatively simple
upper bound.  For instance, given some positive parameters $(\alpha,
\beta)$ such that $\alpha > \beta$, suppose that these functions
are defined locally, in a interval above zero, as follows:
\begin{subequations}  
\begin{align}
\label{eq:special_funcs}
\weakcon(r) = r^{\alpha + 1}, \quad \text{and} \quad \perturb(r) =
r^{\beta} \qquad \mbox{for all $r$ in some interval $[0,\bar{r})$.}
\end{align}
Moreover, suppose that the perturbation function takes the form
\begin{align}
\label{eq:special_per_func}
\noise(n, \delta) = \sqrt{\parenth{ d + \log( \tfrac{1}{\delta})}/ n}.
\end{align}
\end{subequations}
As shown in in \cref{sec:examples}, these particular forms
arise in several statistical models, including Bayesian logistic
regression and over specified Bayesian Gaussian mixture models.
Under these conditions, we have the following simple upper bound:
\begin{corollary}
\label{corollary:spec_case_weak_concave}
Assume that the functions $\weakcon$, $\perturb$ have the local
behavior~\eqref{eq:special_funcs}, and the perturbation term
$\noise(n, \delta)$ has the form~\eqref{eq:special_per_func}.  If, in
addition, the global forms of $\weakcon$ and $\perturb$ satisfy
Assumption~\ref{item:weak_concavity_growth_conditions}, then the
scalar $z^{*}(n, \delta)$ from \cref{theorem-main-weakly-convex}
satisfies the bound $z^{*}(n, \delta) \leq \unicon \; \parenth{
  \tfrac{d + \log(1/ \delta) }{n}} ^{\frac{1}{2 \parenth{ \alpha -
      \beta}}} \vee \parenth{ \tfrac{d + \log(1/
    \delta)}{n}}^{\frac{1}{\alpha + 1}} +
\parenth{\tfrac{\boundconsprior}{\numobs}}^{\frac{1}{\alpha}}$.
\end{corollary}

Note that Corollary~\ref{corollary:spec_case_weak_concave} ensures
that the posterior has the following contraction property
\begin{align}
\label{eq:poster_rate_special_funcs}
\posterior \Big( \|\theta - \thetastar\|_2 \geq \unicon \parenth{
  \tfrac{d + \log(1/ \delta)}{n}}^{\frac{1}{2 \parenth{ \alpha -
      \beta}} \wedge \frac{1}{\alpha + 1}} +
\parenth{\tfrac{\boundconsprior}{\numobs}}^{\frac{1}{\alpha}} \;
\biggr| \; \DataX \Big) \leq \delta \quad
\mbox{with prob. $1 - \delta$}
\end{align}
with respect to the training data. The posterior convergence rate
scales as $(d/ n)^{\frac{1}{2 (\alpha - \beta)}}$ when $\alpha \geq 2
\beta + 1$. On the other hand, this rate becomes $(d/
n)^{\frac{1}{\alpha + 1}}$ when $\alpha < 2 \beta + 1$.

\vspace{0.5 em}
\noindent

\textit{Proof overview:} Similar to the proof
of~\cref{thm-one-point-SC}, the proof
of~\cref{theorem-main-weakly-convex} is based on tracking the behavior
of a Lyapunov function along the trajectory of diffusion
process~\eqref{eq-diffusion-main}. In doing so, we study the moments
$\Exs \left[ \vecnorm{\theta_t - \thetastar}{2}^p \right]$ for $p \geq
2$. Unlike the strongly concave case, however, the negative term in
the expression is no longer the $p$-th moment itself, but rather a
quantity depending on the local geometry of the population
log-likelihood $\loglihood$.  More precisely, we adopt the Lyapunov
function $\lyap_t \mydefn \Exs \left[ \vecnorm{\theta_t -
    \thetastar}{2}^{p - 2} \psi (\vecnorm{\theta_t - \thetastar}{2})
  \right]$, where $\psi$ is the function from
Assumption~\ref{item:weak_concavity_population}. Under the conditions
on the functions $\psi$ and $\zeta$ given in
Assumption~\ref{item:weak_concavity_growth_conditions}, the time
derivative of the $p^{th}$ moment $\Exs \left[ \vecnorm{\theta_t -
    \thetastar}{2}^p \right]$ can then be controlled as a function of
$\lyap_t$. Since the moment converges to a finite quantity when $t
\rightarrow + \infty$, its time derivative cannot converge to a
positive number. Using the convexity of $\psi$, the bound on the
Lyapunov function leads to the inequality
\begin{align*}
\lim_{t \rightarrow +\infty} \left( \mathbb{E} \left ( \vecnorm{
  \theta_t - \thetastar }{2}^p\right) \right)^{ \frac{1}{p} } \leq
z_p^*,
\end{align*}
where $z_p^*$ is the unique positive solution to the equation $\psi
(z) = \noise(n, \delta) \zeta (z) z + \frac{\boundconsprior}{\numobs}
z + \frac{p + d }{ n}$.  In light of the above result and
\cref{prop:keybound}, when $p$ is of the order $\log(1/ \delta)$, we
obtain the posterior convergence
rate~\eqref{EqnWeaklyConvexPosterior}.  The full proof is given in
\cref{subsection:weakly_convex_proof}.

\section{Results under local conditions}
\label{sec:without_global_concavity}

In this section, we present results without the global conditions on
the population log-likelihood function in
\cref{subsec:strongly_convex} and
\cref{subsec:weakly_convex}.  Our set-up allows the posterior
distribution to be multi-modal in nature; only local growth conditions and empirical process bounds around $\thetastar$ are needed in our
analysis. In \cref{subsec:contrac_local_assumption}, we
establish the posterior convergence rate of parameters under mild
local conditions on the population and empirical log-likelihood
functions, and also extend the results to Langevin algorithms. Finally, we provide non-asymptotic Bernstein-von-Mises
results in \cref{subsec:nonasymp_BVM}.

\subsection{Non-asymptotic contraction rates under local assumptions}
\label{subsec:contrac_local_assumption}
We begin with posterior concentration results. When the log-likelihood function satisfies suitable growth conditions and perturbation bounds in a local neighborhood of $\thetastar$, we show posterior convergence rates conditionally on such a local ball. We further extend our results to contraction bounds of the last iterate of Langevin algorithm, again under such local conditions.

\subsubsection{Conditional posterior contraction}
\label{subsec:unique_global}

For some local radius $\localradius > 0$, we make the following
assumptions with the population and sample log-likelihood functions
within the local region $\ball (\thetastar, \localradius)$:
\begin{enumerate}[label={\bf{(LWC.1)}}]
\item
  \label{item:without_global_geometry}
  There exist $\alpha \geq 0$, $\lowerweakconcavemu > 0$ and
  $\smalloffset \geq 0$ such that for $\theta \in \ball (\thetastar,
  \localradius)$, we have
\begin{align*}
  \inprod{\nabla \loglihood (\theta)}{\theta - \thetastar} \leq -
  \lowerweakconcavemu \vecnorm{\theta - \thetastar}{2}^{\alpha + 1} +
  \smalloffset.
    \end{align*}
\end{enumerate}
Assumption~\ref{item:without_global_geometry} characterizes the local growth of the function
$\loglihood$ around the global maximum $\thetastar$. We note that in
either the well-specified case ($\Prob = \Prob_{\thetastar}$), or the
mis-specified case when $\Prob_\thetastar$ is the KL-projection of
$\Prob$, it follows from the optimality condition that this assumption
is satisfied with $\smalloffset = 0$.  Relaxing to values
$\smalloffset > 0$ allows us to accommodate mis-specified cases in
which $\thetastar$ is not the exact projection, or
situations in which variants of the 
log-likelihood are used. See \cref{sec:ibra-kham} in the
supplementary material for an application of this result with
$\smalloffset > 0$ to a Bayesian location model with singular
densities on the density function in the Ibragimov-Khasminskii
sense~\cite{ibragimov1979asymptotic}.

The parameters $(\alpha, \lowerweakconcavemu)$ control the rate of
local growth of the log-likelihood. When $\alpha = 1$, the function
$\loglihood$ is locally strongly concave around $\thetastar$, so that
one should expect posterior convergence at the rate given in
~\cref{thm-one-point-SC}. On the other hand, when $\alpha > 1$, the
log likelihood is only weakly concave in a local neighborhood; such
behavior arises when the Fisher information matrix at $\thetastar$ is
degenerate.  Concrete instances of such degenerate models include
over-specified mixture distributions, and certain types of non-linear
regression models. See \cref{sec:over_Gaussian_mixture} in the
supplementary material for discussion of these specific examples.

Our next assumption concerns the deviation between the gradients of
population and sample log-likelihood functions within the ball $\ball
(\thetastar, \localradius)$. \\
\begin{enumerate}[label={\bf{(LWC.2)}}]
    \item\label{item:without_global_deviation} There exists $\beta \in
      (-1, \alpha)$ and $\varepsilon (n, \delta) > 0$ such that with
      probability $1 - \delta$, we have
    \begin{align*}
        \sup_{\theta \in \ball (\thetastar, \localradius)}
        \frac{\vecnorm{\nabla \loglihood_n (\theta) - \nabla
            \loglihood (\theta)}{2}}{\vecnorm{\theta -
            \thetastar}{2}^\beta} \leq \varepsilon(n, \delta) .
    \end{align*}
\end{enumerate}
Note that the assumption~\ref{item:without_global_deviation} requires that $\beta < \alpha$, which means that the variance of score functions cannot decay too quickly around a neighborhood of $\thetastar$. This condition is needed to make the presentation simpler. On the other hand, when $\beta \geq \alpha$, exact recovery of $\thetastar$ is possible, and the Bayesian approach may lead to sub-optimal results. A detailed development for this setting is left for the future work.

Under Assumptions~\ref{item:without_global_geometry}
and~\ref{item:without_global_deviation}, we have the following result
on posterior convergence for the parameters.  It involves the radius
$r_n$ given by
\begin{align*}
r_n \mydefn \left(\frac{\log(1/\vartheta) + d}{n \lowerweakconcavemu}
+ \frac{\smalloffset}{\lowerweakconcavemu} \right)^{\frac{1}{\alpha +
    1}} + \left( {\frac{2 \varepsilon (n,
    \delta)}{\lowerweakconcavemu}}\right)^{\frac{1}{\alpha - \beta}} + \left(\frac{\boundconsprior}{n
  \lowerweakconcavemu} \right)^{\frac{1}{\alpha}},
    \end{align*}
where $\vartheta \in (0,1)$ is a pre-specified tolerance parameter.
\begin{theorem}
\label{thm:local-weak-convex}
Assume that
Assumptions~\ref{item:without_global_geometry},~\ref{item:without_global_deviation}
and~\ref{item:smooth_log_prior} hold. For any given $\hackpar \in
(0,1)$ and pair ($n, \delta)$ such that $\varepsilon(n, \delta) \leq
\frac{\lowerweakconcavemu}{2} (\frac{\localradius}{2})^{\alpha -
  \beta}$, we have
\begin{align}
\label{EqnWeirdLower}  
  \posterior \left( \vecnorm{\theta - \thetastar}{2} \leq r_n \mid
  X_1^n \right) \geq \big( 1 - \hackpar \big) \; \posterior \left(
  \ball (\thetastar, \localradius / 2) \mid X_1^n \right).
    \end{align}
\end{theorem}
\noindent
See~\cref{subsec:proof:thm:local-weak-convex} for the proof. \\

\vspace{0.5 em}
\noindent
\textit{Remark on $\posterior \left( \ball (\thetastar, \localradius /
  2) \mid X_1^n \right)$:} As shown in the
bound~\eqref{EqnWeirdLower}, the posterior convergence rate depends on
the non-asymptotic behavior of the probability mass $\posterior \left(
\ball (\thetastar, \localradius / 2) \mid X_1^n \right)$ that the
posterior assigns to a local ball around $\thetastar$. In particular,
given a (potentially non-sharp) non-asymptotic posterior contraction
that ensures concentration within a \emph{constant-radius ball} $\ball
(\thetastar, \localradius / 2)$, ~\cref{thm:local-weak-convex}
automatically improves it to a concentration result with the
\emph{optimal} radius. Moreover, in the non-identifiable case where
the global maxima of the population log-likelihood function
$\loglihood$ is not unique, one can still apply
~\cref{thm:local-weak-convex} to obtain concentration around the
(finite) set of global maxima (see Corollary~\ref{cor:final-nonconvex}
in \cref{subsec:multiple_global_maxima} for more details).

\vspace{0.5 em}
\noindent
\textit{Remarks on $r_\numobs$:} Let us consider the three different
terms in $r_\numobs$.  First, consider idealized situation in which
the empirical log-likelihood function replaced by the population one,
and we ignore the contribution from the prior $\prior$.  The
``posterior'' in this case takes the form $e^{- \numobs \loglihood}$;
note that it satisfies the contraction bounds with radius
$\left(\frac{\log (1/\hackpar) + d}{n \lowerweakconcavemu} +
\frac{\smalloffset}{\lowerweakconcavemu} \right)^{\frac{1}{\alpha +
    1}}$. The second term $ \left( {\frac{2 \varepsilon (n,
    \delta)}{\lowerweakconcavemu}}\right)^{\frac{1}{\alpha - \beta}}$
characterizes the effect of using empirical data instead of
population-level functions. This term coincides with the
non-asymptotic rates for the maximal likelihood estimator in a local
neighborhood of $\thetastar$. Finally, the last term
$\left(\frac{\boundconsprior}{n \lowerweakconcavemu}
\right)^{\frac{1}{\alpha}}$ characterizes the effect of the prior
density $\prior$. Under mild regularity conditions on $\prior$, this
term is of higher order compared to the first term $\left(\frac{\log
  (1/\hackpar) + d}{n \lowerweakconcavemu} \right)^{\frac{1}{\alpha +
    1}}$, as in the case of \cref{thm-one-point-SC}.

\vspace{0.5 em}

\noindent \textit{Remarks on the proof:} Let us provide some
high-level comments on the proof.  The argument involves constructing
a Lyapunov function similar to that used
\cref{theorem-main-weakly-convex}. However, since the
condition~\ref{item:without_global_geometry} holds only in a small
ball $\ball (\thetastar, \localradius)$, the leading term
$\inprod{\nabla \loglihood (\theta)}{\theta - \thetastar}$ in
It\^{o}'s formula \emph{cannot} be uniformly upper bounded by a
negative function of the distance $\vecnorm{\theta -
  \thetastar}{2}$. In order to overcome this issue, we first study a
modified version of the posterior distribution, and then transform the
result back to the posterior distribution itself. In particular, we construct a
probability density function $\widetilde{\posterior}$ over $\real^d$
such that:
\begin{itemize}
    \item Within the local ball $\ball (\thetastar, \localradius /
      2)$, the shape of the function $\widetilde{\posterior}$ exactly
      matches that of the true posterior $\posterior (\cdot \mid
      \DataX)$, up to a multiplicative constant.
    \item Outside the larger ball $\ball (\thetastar, \localradius)$,
      the function $\widetilde{\posterior}$ behaves as a Gaussian
      density---in particular, we have $\widetilde{\posterior}
      (\theta) \propto \exp \left( - \frac{\numobs \smooth}{2}
      \vecnorm{\theta - \thetastar}{2}^2 \right)$.
    \item In the annulus between the two balls, we interpolate between
      the two regimes so as to ensure that $\log
      \widetilde{\posterior}$ is smooth.
\end{itemize}

By applying the analysis in \cref{theorem-main-weakly-convex}
to the modified density $\widetilde{\posterior}$, one can show that
this modified density concentrates within $\ball (\thetastar,
r_\numobs)$ with high probability. Since the shapes of
$\widetilde{\posterior}$ and $\posterior (\cdot \mid \DataX)$ are
exactly the same inside $\ball (\thetastar, \localradius / 2)$, we can
prove that conditionally in the ball $\ball (\thetastar, \localradius
/ 2)$, the posterior $\posterior (\cdot \mid \DataX)$ also contracts
around $\thetastar$ with the correct radius.

\subsubsection{Contraction of approximate posterior via Langevin algorithm}
We have analyzed posterior contraction properties using the Langevin diffusion process~\eqref{EqnLangevinSDE}, upon which most posterior sampling algorithms are built. It is therefore natural to extend our techniques to the discretized Langevin process, and obtain contraction rates for the approximate posterior distribution computed via Langevin algorithm. In this section, we analyze the following forward Euler discretization of Langevin diffusion, a widely-used algorithm for computation of posterior~\cite{durmus2017nonasymptotic,durmus2017nonasymptotic,dalalyan2017theoretical}.
\begin{align}
  \theta_{k + 1} = \theta_k +  \stepsize \nabla \loglihood_\numobs (\theta_k) + \sqrt{\frac{2 \stepsize}{\numobs}} W_k, \quad \mbox{for $k = 0,1,\cdots$}
\end{align}
where $(W_k)_{k = 0,1, \cdots}$ are $\mathrm{i.i.d.}$ standard Gaussian random vectors.

 As Euler discretization can be unstable when applied to functions with growth at infinity faster than quadratic (see~\cite{roberts1996exponential}), we focus on the case where $\alpha = 1$ and $\beta = 0$. The general case, for which a more stable discretization scheme may be employed, is an important direction of future research. We also restrict our attention to algorithms with \emph{local initialization}, satisfying $\vecnorm{\theta_0 - \thetastar} \leq \localradius / 2$. Finally, we require the stepsize $\stepsize$ and the sample size $\numobs$ to satisfy the following conditions:
 \begin{align}
    \stepsize \leq \frac{\lowerweakconcavemu}{3\smoothness^2}\quad \mbox{and} \quad
   \frac{3 \varepsilon (\numobs, \delta)}{\strongconvex} + \frac{3 \boundconsprior}{\numobs \strongconvex} + \sqrt{\frac{3c d}{\numobs \strongconvex} \log^3 \frac{T}{\delta}} \leq \frac{\localradius}{2}.\label{eq:sample-size-req-in-langevin-alg}
 \end{align}

 Under such setup, we have the following theorem:

\begin{theorem}
\label{thm:contraction-langevin-alg}
Under Assumptions~\ref{item:without_global_geometry} and
\ref{item:without_global_deviation} with $\alpha = 1$ and $\beta = 0$,
for sample size $\numobs$ and stepsize satisfying Eq~\eqref{eq:sample-size-req-in-langevin-alg},
given a local initialization satisfying, we have the following with
probability $1 - \delta$ with respect to both the data and the
  randomness in the algorithm:
  \begin{align}
   \vecnorm{\theta_T - \thetastar}{2} \leq  e^{- \frac{T \strongconvex \stepsize}{12 \log (1 / \delta)} } \vecnorm{\theta_0 - \thetastar}{2} + c \Big\{ \frac{ \varepsilon (\numobs, \delta)}{\strongconvex}  + \frac{\boundconsprior}{\strongconvex \numobs} + \log (1 / \delta) \cdot \sqrt{\frac{\usedim + \log (1 / \delta)}{\strongconvex \numobs}} \Big\}. \label{eq:bound-in-langevin-alg-thm-statement}
  \end{align}
\end{theorem}
\noindent See \cref{subsec:proof-contraction-langevin} for the proof of this theorem. A few remarks are in order.

The contraction rate for Langevin algorithm consists of four terms: the first term depends on the initial distance $\vecnorm{\theta_0 - \thetastar}{2}$, and is exponentially decaying with the number of iterations $T$.
By taking number of iterations $T \geq \frac{c}{\strongconvex \stepsize} \log^2 \big( \frac{\localradius^2 \strongconvex \numobs}{\delta} \big)$, the first term in the bound~\eqref{eq:bound-in-langevin-alg-thm-statement} becomes dominated by other terms. The rest three terms in Eq~\eqref{eq:bound-in-langevin-alg-thm-statement} matches the optimal posterior contraction rates in \cref{thm-one-point-SC}, up to extra logarithmic factors in $1/\delta$. The sample size requirement in Eq~\eqref{eq:sample-size-req-in-langevin-alg} is essentially the sample size needed for the bound to be smaller than a constant $\localradius$. Notably, unlike existing literature on contraction analysis for Langevin algorithm~\cite{mazumdar2020approximate,gadat2020cost} the stepsize requirement in Eq~\eqref{eq:sample-size-req-in-langevin-alg} does not depend on the sample size $\numobs$ or the problem dimension $\usedim$. Indeed, we only require it to be smaller than a stability threshold $\frac{\lowerweakconcavemu}{3\smoothness^2}$. This makes it possible for Langevin algorithms to use larger stepsize and achieve faster convergence, while still preserving good posterior contraction properties. Such distinction is due to the proof technique: instead of bounding the error between the distribution of Langevin algorithm iterates and the true posterior, we directly analyze the dynamics itself following the same approach as we analyze posterior contraction.


\subsection{Non-asymptotic Bernstein-von-Mises results}
\label{subsec:nonasymp_BVM}
In this section, we develop non-asymptotic Bernstein-von-Mises results
using the diffusion process~\eqref{eq-diffusion-main}. Under mild
assumptions on the population-level and empirical-level landscapes, we
establish the KL divergence between the posterior distribution and the
limiting Gaussian distribution based on the posterior convergence
rates of the parameters.

In order to obtain the non-asymptotic Bernstein-von-Mises results, we first need the following assumptions on the second order derivatives with respect to the parameters (or equivalently Hessian matrices) of the empirical and population log-likelihoods:
\begin{enumerate}[label={\bf{(BvM.\arabic*)}}]
\item\label{item:Bvm_population}
    There exists $\holderconst > 0$ such that the population log-likelihood function $\loglihood$ satisfies the one-point Lipschitz condition:
\begin{align*}
  \forall \theta \in \real^d, \quad \opnorm{\nabla^2 \loglihood
    (\theta) - \nabla^2 \loglihood (\thetastar)} \leq \holderconst
  \vecnorm{\theta - \thetastar}{2}.
    \end{align*}
\item \label{item:Bvm_deviation} For any $\delta > 0$, there exist
  non-negative functions $\noiseone^{(2)}$ and $\noisetwo^{(2)}$ with
  domain \mbox{$\mathbb{N} \times (0,1]$} such that
\begin{align*}
\sup_{\theta \in \ball (\thetastar, r)} \opnorm{\nabla^2
  \loglihood_{n} (\theta) - \nabla^2 \loglihood(\theta) } \leq
\noiseone^{(2)} (n, \delta) r + \noisetwo^{(2)} (n, \delta),
\end{align*}
for any radius $r > 0$ with probability at least $1 - \delta$.
\end{enumerate}

\medskip

The first condition~\ref{item:Bvm_population} is a standard smoothness
condition needed to prove quantitative results about asymptotic
normality (e.g., the paper~\cite{polyak1992acceleration}), and
satisfied by many models such as exponential family models, location
density models, as well as their mixtures and hierarchical
composition. The second condition~\ref{item:Bvm_deviation} is an
empirical process condition on the Hessian matrix $\nabla^2
\loglihood_\numobs$. This condition can usually be verified using
suitable concentration bounds for each $\theta$, as well as smoothness
conditions on $\nabla^2 \loglihood_\numobs$ used in controlling metric
entropies. Both assumptions are naturally needed: the limiting
Gaussian law $\mathcal{N} \big( \thetamap, (\numobs\HessianStar)^{-1}
\big)$, which depends on the population-level Hessian at the point
$\thetastar$. The shape of posterior distribution, on the other hand,
depends on the sample-level Hessian $\nabla^2 \loglihood_\numobs$ in a
local neighborhood of $\thetastar$. These two conditions are needed to
relate the shape of the sample-level posterior with the matrix
$\HessianStar$. As before, we note that these assumptions do not
require the model to be well-specified, and our non-asymptotic
Bernstein-von-Mises theorems applies to the mis-specified case, where
$\thetastar$ is the KL-projection of the model to this parametric
class.

Consider the MAP estimate $\thetamap \mydefn \arg\max_{\theta \in
  \real^d} \left( \loglihood_n (\theta) + \frac{1}{n} \log \prior
(\theta) \right)$.  Then, we have the following upper bound on the
difference between the posterior distribution of the parameters and
the Gaussian distribution with mean $\thetamap$ and covariance matrix
$(n\HessianStar)^{-1}$, where $\HessianStar \mydefn - \nabla^2 F
(\thetastar)$.
\begin{proposition}
\label{thm-non-asymp-bvm}
Under Assumptions~\ref{item:Bvm_population},~\ref{item:Bvm_deviation}
and~\ref{item:smooth_log_prior}, suppose that $\HessianStar \succ 0$,
and that $\|\thetamap - \thetastar\|_2 \leq \sigma
\sqrt{\tfrac{d}{\numobs}}$ and $\Exs_\posterior (\|\theta -
\thetastar\|_2^4 \mid X_1^n)^{1/4} \leq \sigma \sqrt{\tfrac{d}{n}}$
with prob. $1 - \delta$.  Then there exists a constant $\unicon$ such
that the KL divergence $\kull{\posterior (\cdot \mid
  X_1^\numobs)}{\mathcal{N} (\thetamap, (n\HessianStar)^{-1})}$ is at
most
\begin{align*}
\unicon \cdot \tfrac{1}{\lambda_{\min} (\HessianStar)} \left(
\tfrac{\holderconst^2 d^2 \sigma^4 }{n} +
\tfrac{\noiseone^{(2)}(\numobs, \delta)^2 d^2 \sigma^4}{n} + \sigma^2
\left( \noisetwo^{(2)}(\numobs, \delta)^2 + \tfrac{
  \smoothprior^2}{n^2} \right) d \right) \quad \mbox{with prob. at
  least $1 - 2 \delta$.}
\end{align*}
\end{proposition}
\noindent
See~\cref{subsec:proof:thm-non-asymp-bvm} for the proof of this claim.

A few remarks are in order. First, assuming that the problem-dependent
constants $(\holderconst, \sigma, \smoothprior)$ are of constant
order, and that the deviation bound scales as $\noisetwo^{(2)}
(\numobs, \delta) = O (1 / \sqrt{\numobs})$,
\cref{thm-non-asymp-bvm} shows that the KL divergence
between the posterior distribution and the Gaussian limit is of order
$O (1 / \numobs)$; second, the non-asymptotic behavior of posterior
distribution depends on the Hessian matrix $\HessianStar = - \nabla^2
\loglihood (\thetastar)$. In the well-specified case where the data
points $\DataX$ are $\mathrm{i.i.d.}$ samples from the distribution
$\Prob_{\thetastar}$, the standard Fisher-information identity
$\HessianStar = \Exs_{\thetastar} \left[ \nabla \log p_{\thetastar}
  (X) \nabla \log p_{\thetastar} (X)^\top \right]$ holds true, and the
Bayesian credible set is asymptotically the same as the confidence set
in the frequentist sense. On the other hand, in the mis-specified
models where $\thetastar = \arg\min_{\theta \in \Theta}
\kull{\Prob}{\Prob_\theta}$, the limiting Gaussian law is $\mathcal{N}
(\thetamap, (\numobs \HessianStar)^{-1})$, depending on the Hessian
matrix but not the covariance of the log-likelihood. This result
coincides with the asymptotic Bernstein-von-Mises theorem for
mis-specified parametric models~\cite{kleijn2012bernstein}, providing
a non-asymptotic characterization. Using Pinsker's inequality and Talagrand's
$T_2$-inequality~\cite{talagrand1996transportation}, the KL
divergence bound can also be transformed into bounds in term of total
variation and Wasserstein-$2$ distances, yielding a non-asymptotic $O (1 / \sqrt{\numobs})$ rate of convergence.

We can also use the diffusion process approach to derive more
fine-grained concentration bounds for the posterior distribution, with
behavior mathching the limiting Gaussian law.  Doing so requires the
following stronger version of the posterior contraction condition:
\begin{align}
\label{eq:contraction-rate-condition-for-bvm}  
\left(\Exs_\posterior \left[ \vecnorm{\theta - \thetastar}{2}^{2p}
  \mid \DataX \right] \right)^{1/p} \leq \frac{\sigma^2 p d }{n},
\quad \mbox{for all $p > 0$ with probability at least $1 - \delta$.}
\end{align}

In addition, we define the function
\begin{align*}
  \mathcal{H}_\numobs( t, \delta) \mydefn (\holderconst +
  \noiseone^{(2)}(\numobs, \delta))^2 \cdot \frac{\sigma^4 d^2 t^2
  }{n^2} + \frac{\sigma d}{n} \left( \noisetwo^{(2)} (\numobs,
  \delta)^2 + \frac{\smoothprior^2}{\numobs^2} + (\holderconst +
  \noiseone^{(2)}(\numobs, \delta))^2 \frac{\sigma d}{\numobs}
  \right),
\end{align*}
which plays the role of a higher-order term.  Equipped with this
notation, we have:
\begin{theorem}
\label{thm:non-asymp-credible-set}
Suppose that conditions~\ref{item:Bvm_population}
and~\ref{item:Bvm_deviation} are in force, the Hessian $\HessianStar$
is strictly positive definite, and the high-probability posterior
contraction condition~\eqref{eq:contraction-rate-condition-for-bvm}
holds. Then for any $\delta \in (0,1)$, uniformly over all $\offpar
\in (0, 1)$ and $t > 0$, we have
\begin{align}
  \posterior \left( \vecnorm{\theta - \thetamap}{\HessianStar}^2 \geq
  (1 + \offpar)\frac{\usedim}{\numobs} + c \frac{1 + \log \kappa
    (\HessianStar)}{\offpar} \left( \frac{t}{n} + \mathcal{H}_n (t,
  \delta) \right) \; \bigg| \; X_1^\numobs \right) \leq e^{- t},\label{eq:non-asymtotic-bvm-credible-set-statement}
\end{align}
 with probability at least $1 - \delta$.
\end{theorem}
\noindent See \cref{subsec:proof:thm:non-asymp-credible-set} for the
proof of the theorem.

\vspace{0.5em}

A few remarks are in order. Note that the limiting Gaussian density
$\gamma_n = \mathcal{N} \big(0, (n \HessianStar)^{-1} \big)$ satisfies
a tail bound of the form $\gamma_\numobs \Big( \|\theta -
\thetamap\|_{\HessianStar}^{2} \geq \tfrac{d}{\numobs} +
\tfrac{t}{\numobs} \Big) \leq e^{-t/2}$ for any $t > 0$.  Unless the
posterior is actually Gaussian in finite samples, it cannot satisfy
this bound exactly.  However, \cref{thm:non-asymp-credible-set}
provides a bound with near-matching behavior: note that the
leading-order term scales $\frac{d}{n}$, matching the asymptotics with
a pre-factor $1 + \omega$ that can be made arbitrarily close to $1$
(at the expense of the other term). The $\frac{t}{n}$ dependency on
the tail probability comes with a mild $\log \kappa (\HessianStar)$
factor due to technical reasons. The bound also contains a high-order
term $\mathcal{H}_n (t, \delta)$, which scales as $O(n^{-2})$. It is
also worth noticing that the terms in
\cref{thm:non-asymp-credible-set} depend on the tail probability
$\hackpar = e^{-t}$ only logarithmically, allowing for very small
value of $\hackpar$. We can therefore use
equation~\eqref{eq:non-asymtotic-bvm-credible-set-statement} to
construct non-asymptotic credible sets of ellipsoid shape, adapted to
the geometry of local Hessian matrix $\HessianStar$.

\vspace{0.5em}
\noindent\textit{Proof outline:} The proofs of both
\cref{thm-non-asymp-bvm} and \cref{thm:non-asymp-credible-set} rely on
a first-order approximation of the gradient $\nabla
\loglihood_\numobs$. In particular, the diffusion
process~\eqref{eq-diffusion-main} can be written in the form $d
\theta_t = - \frac{1}{2} \HessianStar (\theta_t - \thetamap) dt +
\tfrac{1}{2} e_\numobs (\theta_t) dt + \tfrac{1}{2 \numobs} \log
\prior (\theta_t) dt + \tfrac{1}{\sqrt{\numobs}} d B_t$, where we have
defined the linearization error $e_\numobs(\theta) \mydefn \nabla
\loglihood_\numobs (\theta) + \HessianStar (\theta -
\thetastar)$. Under the smoothness
assumption~\ref{item:Bvm_population} and the empirical process
bound~\ref{item:Bvm_deviation}, one can show that $\vecnorm{e_\numobs
  (\theta)}{2} \leq \vecnorm{\theta - \thetastar}{2} \cdot
O(\sqrt{d/n})$ with high probability.  When this error term is
ignored, the diffusion process is an Ornstein-Uhlenbeck process whose
stationary distribution is $\mathcal{N} \big( \thetamap, (\numobs
\HessianStar)^{-1} \big)$. Therefore, given the non-asymptotic bounds
on the error $e_\numobs (\theta)$ stated above, we can provide a
non-asymptotic characterization of the distance between the stationary
distribution and the limiting Gaussian law.  In order to prove
\cref{thm-non-asymp-bvm}, we use the Gaussian log-Sobolev
inequality~\cite{gross1975logarithmic} to control the KL divergence,
whereas proving \cref{thm:non-asymp-credible-set} is based on using
It\^{o} calculus to study the growth of a Lyapunov function defined
using the metric induced by $\HessianStar$.  Full proofs for the two
results are given in~\cref{subsec:proof:thm-non-asymp-bvm}
and~\cref{subsec:proof:thm:non-asymp-credible-set}, respectively.


\section{Some illustrative examples}
\label{sec:examples}

Having developed some general theory, we now use it to derive some
concrete results for two examples of interest in statistical analysis:
Bayesian logistic regression and Gaussian mixture models.  Due to
space constraints, we defer the treatment of additional examples
to~\cref{sec:app-additional-examples}.

\subsection{Bayesian logistic regression}
\label{sec:logistic_regress}

Logistic regression is a classical way of modeling the relationship
between a binary response variable $Y \in \{-1, +1 \}$ and a vector $X
\in \real^\usedim$ of explanatory variables (e.g., see the
book~\cite{McCullagh_generalized}).  In the logistic regression model,
the pair $(X, Y)$ are related by the conditional distribution
\begin{align}
\label{eq:Bayes_logistic_regress}
\Prob \parenth{ Y = 1 \mid X, \theta } = \tfrac{e^{
    \inprod{X}{\theta}}}{1 + e^{ \inprod{X}{\theta}}}, \qquad
\mbox{where $\theta \in \real^\usedim$ is a parameter vector.}
\end{align}

Suppose that we observe a collection $\DataZ = \{Z_i \}_{i=1}^\numobs$
of $\numobs$ i.i.d paired samples \mbox{$Z_i = (X_i, Y_i)$,} each
generated in the following way.  First, the covariate vector $X_i$ is
drawn from a standard Gaussian distribution $N(0, I_\usedim)$, and
then the binary response $Y_i$ is drawn according to the conditional
distribution $\Prob \parenth{ \cdot \mid X_i, \thetastar }$ from
equation~\eqref{eq:Bayes_logistic_regress}, where $\thetastar \in
\real^\usedim$ is a fixed but unknown value of the parameter vector.
Given these assumptions, the sample log-likelihood function of the
samples $\DataZ$ takes the form $\loglihoodlogit_\numobs (\theta)
\mydefn \frac{1}{\numobs} \sum_{i=1}^\numobs \left \{ \log \Prob
\parenth{ Y_{i} \mid X_{i}, \theta} + \log \phi(X_i) \right \}$, where
$\phi$ denotes the density of a standard normal vector.  Combining
this log likelihood with a given prior $\prior$ over $\theta$ yields
the posterior distribution in the usual way. We assume that the prior
function $\prior$ satisfies
Assumptions~\ref{item:smooth_population_condition}
and~\ref{item:smooth_log_prior}, and recall the constant $B$ defined
in the latter assumption.

With this set-up, the following result establishes the posterior
convergence rate of $\theta$ around $\thetastar$, conditionally on the
observations $\DataZ$.
\begin{corollary}
\label{cor:logit_regres}
For any $\delta \in (0,1)$, given $\frac{\numobs}{\log \numobs} \geq
\uniconprime d \log(\frac{1}{\delta})$ i.i.d. samples from the
Bayesian logistic regression model~\eqref{eq:Bayes_logistic_regress},
we have $\posterior \Big( \|\theta - \thetastar\|_2 \geq \unicon \big
\{ \sqrt{\tfrac{\usedim}{\numobs}} + \sqrt{\tfrac{\log(1/
    \delta)}{\numobs}} + \tfrac{\boundconsprior}{\numobs} \big \} \;
\mid \; \DataZ \Big) \leq \delta$ with probability $1 - \delta$ over
the data $\DataZ$.
\end{corollary}
\noindent See~\cref{subsec:cor:logit_regres} for the proof of this
claim.

A few comments are in order. First, the result of
Corollary~\ref{cor:logit_regres} shows that for Bayesian logistic
regression model~\eqref{eq:Bayes_logistic_regress}, the posterior
convergence rate for the parameter is of the order
$(d/\numobs)^{1/2}$.  Furthermore, this result also gives a concrete
dependence of the rate on $B$ characterizing the degree to which the
prior is concentrated away from the true parameter.  Second, by taking
the sample size in the function $\loglihoodlogit_\numobs$ to infinity,
we find that the population log-likelihood is given by
$\loglihoodlogit(\theta) \mydefn \Exs_{(X, Y)} \brackets{- \log
  \parenth{1 + e^{- Y \inprod{X} {\theta}}} + \log \phi(X)}$. Here
$\phi$ denotes the standard normal density in $\real^d$, and the outer
expectation in the above display is taken with respect to $X$ and $Y
\mid X$ from the logistic model~\eqref{eq:Bayes_logistic_regress}.

Let us sketch how \cref{theorem-main-weakly-convex} can be applied so
as to prove this corollary. The first step in our proof, as given in
\cref{subsec:cor:logit_regres}, is to show that there are universal
constants $\unicon, \unicon_{1}, \unicon_{2}$ such that
\begin{subequations}
\begin{align}
\label{eq:weak_conv_logit}  
  -\inprod{ \nabla \loglihoodlogit(\theta)}{\theta - \thetastar} &
  \geq \unicon_1 \begin{cases} \enorm{ \theta - \thetastar}^{2}, \quad
    \text{for all} \ \enorm{\theta - \thetastar} \leq 1 \\ \enorm{
      \theta - \thetastar}, \quad \text{otherwise} \end{cases}, \quad \mbox{and} \\
\label{eq:empi_process_logit}
  \sup_{\theta \in \real^d } \vecnorm{\nabla \loglihoodlogit_{n}
    (\theta) - \nabla \loglihoodlogit ( \theta) }{2} & \leq \unicon_2
  \parenth{\sqrt{\tfrac{\usedim}{\numobs}} +
    \sqrt{\tfrac{\log(1/\delta)}{\numobs}} +
    \tfrac{\log(1/\delta)}{\numobs}},
\end{align}
\end{subequations}
for any $r > 0$ with probability $1 - \delta$ as long as
$\frac{n}{\log n} \geq \unicon d \log(1/ \delta)$.  Using these
results, we show that Assumptions~\ref{item:weak_concavity_population}
and~\ref{item:weak_concavity_deviation} hold with
\begin{align}
\label{eq:value_func_logit}
\weakcon(r) & = \unicon_1 \begin{cases} r^2 \quad & \mbox{
    for all $r \in (0,1)$, and} \\
  r  & \mbox{otherwise} \end{cases}, \quad
\text{and} \quad \perturb(r) = \unicon_2 \quad \mbox{for all $r > 0$.}
\end{align}
We can check that the functions $\weakcon$ and $\perturb$ satisfy the
conditions in Assumptions~\ref{item:weak_concavity_growth_conditions}
and~\ref{item:weak_concavity_tail}.  Therefore,
applying~\cref{theorem-main-weakly-convex} to these functions yields
the posterior contraction rate claimed in~\cref{cor:logit_regres}.
See~\cref{subsec:cor:logit_regres} for the details.


\subsection{Over-specified Bayesian Gaussian mixture models}
\label{sec:over_Gaussian_mixture}

Gaussian mixtures are widely used for modeling heterogenous datasets;
clusters in the data are naturally associated with different mixture
components~\cite{Lindsay-1995}.  In fitting such models, the true
number of components is generally unknown, and several approaches have
been proposed to deal with this challenge.  One of the most popular
methods is to deliberately include a large number of conmponents,
leading to what are known as overspecified Gaussian mixture
models~\cite{Rousseau-2011}.  While the behavior of posterior
densities in such mixture models is relatively
well-understood~\cite{Ghosal-2001}, the behavior of the posterior in
terms of its parametric components is not as well understand.  When
the covariance matrices are known and the parameter space is bounded,
the location parameters have been shown to have posterior convergence
rates of the order $\numobs^{- 1/ 4}$ in the Wasserstein-$2$
metric~\cite{Nguyen-13}. However, neither the dependence on dimension
$d$ nor on the true number of components have been established.

In this section, we consider the behavior of overspecified Gaussian
mixture models in a particular setting, and provide convergence rates
for the parameters with precise dependence on the dimension $d$, and
without requiring any boundedness assumption.  In order to model the
simplest form of over-specification, suppose that we fit a Bayesian
location mixture model to a collection of i.i.d.  samples
$X_{1}^\numobs = (X_{1}, \ldots, X_\numobs)$ drawn from a Gaussian
Gaussian distribution $\NORMAL(\thetastar, I_{d})$.  (For
concreteness, we set $\thetastar = 0$.)  We study the behavior of the
Bayesian Gaussian mixture model
\begin{align} 
\label{eq:Gaussian_mixture} 
    \theta \sim \prior(\cdot ), \qquad \Gind_{i} \in \{- 1, 1\}
    \overset{\text{i.i.d.}}{\sim} \text{Cat}(1/ 2, 1/ 2), \qquad X_{i}
    \mid \Gind_{i}, \theta \overset{\text{i.i.d.}}{\sim} \NORMAL(
    \Gind_i \theta, I_{d}),
\end{align}
where $\text{Cat}(1/2, 1/2)$ stands for the categorical distribution
with parameters $(1/ 2, 1/ 2)$. We assume that the prior $\prior$
satisfies the smoothness
Assumptions~\ref{item:smooth_population_condition}
and~\ref{item:smooth_log_prior}; one example is a Gaussian
distribution (over the location parameter $\theta$.  Our goal in this
section is to characterize the posterior contraction rate of the
location parameter $\theta$ around $\thetastar$.

In order to do so, we first define the sample log-likelihood function
$\loglihoodgaus_{n}$ given data $X_{1}^{n}$.  It has the form
$\loglihoodgaus_{n}(\theta) \mydefn \frac{1}{n} \sum_{i = 1}^{n} \log
\parenth{\frac{1}{2} \normden(X_{i}; - \theta, I_{d}) + \frac{1}{2}
  \normden(X_{i}; \theta, I_{d})}$, where $x \mapsto \normden(x;
\theta, I_{d}) = (2 \pi)^{- d/ 2}e^{- \enorm{ x - \theta}^2/2}$
denotes the density of multivariate Gaussian distribution
$\NORMAL(\theta, \sd^2 I_{d})$. Similarly, the population
log-likelihood function is given by $\loglihoodgaus( \theta) \mydefn
\Exs_{X} \brackets{ \log \parenth{\frac{1}{2} \normden(X; - \theta,
    I_{d}) + \frac{1} {2} \normden(X; \theta, I_{d})}}$, where the
outer expectation in the above display is taken with respect to $X
\sim \NORMAL(\thetastar, I_{d})$.

In \cref{subsec:proof:cor:Gaussian_mixture}, we prove that there is a
universal constant $\unicon_1 > 0$ such that
\begin{subequations}
\begin{align}
    -\inprod{ \nabla \loglihoodgaus(\theta)}{\theta - \thetastar} &
    \geq \begin{cases} \unicon_{1} \enorm{ \theta - \thetastar}^{4},
      \quad & \text{for all} \ \enorm{\theta - \thetastar} \leq
      \sqrt{2} \\ 4 \unicon_{1} \parenth{ \enorm{ \theta -
          \thetastar}^2 - 1}, \quad &
      \text{otherwise} \end{cases}, \label{eq:weak_conv_gaus}
\end{align}
and moreover, there are universal constants $(\unicon, \unicon_{2})$
such that for any $\delta \in (0,1)$, given a sample size $n \geq
\unicon d \log(1/ \delta)$, we have
\begin{align}
\label{eq:empi_process_gaus}    
\sup_{\theta \in \ball (\thetastar, r)} \|\nabla \loglihoodgaus_{n}
(\theta) - \nabla \loglihoodgaus(\theta)\|_2 & \leq \unicon_{2}
\parenth{ r + \tfrac{1}{ \sqrt{n}} } \big( \sqrt{\tfrac{d}{\numobs}} +
\sqrt{\tfrac{ \log(\log(\numobs/\delta))}{\numobs}} \big) \quad
\mbox{with prob. $1 - \delta$.}
\end{align}
\end{subequations}

Given the above results, the functions $\weakcon$ and $\perturb$ in
Assumptions~\ref{item:weak_concavity_population}
and~\ref{item:weak_concavity_deviation} take the form
\begin{align}
\label{eq:value_func_gauss}
\weakcon(r) = \begin{cases}
  c_{1} r^4, \quad & \text{for all} \ 0< r \leq \sqrt{2} \\
    4 c_{1} \parenth{ r^2 - 1}, \quad & \text{otherwise}
\end{cases},
\quad \text{and} \quad \perturb ( r) = r + \frac{1}{\sqrt{n}} \quad
\mbox{for all $r > 0$.}
\end{align}
These functions satisfy the conditions of Assumptions~\ref{item:weak_concavity_growth_conditions} and~\ref{item:weak_concavity_tail}.  Therefore,
it leads to the following result regarding the posterior contraction
rate of parameters under overspecified Bayesian location Gaussian
mixtures~\eqref{eq:Gaussian_mixture}:
\begin{corollary}
\label{cor:Gaussian_mixture}
Given the overspecified Bayesian location Gaussian mixture
model~\eqref{eq:Gaussian_mixture}, there are universal constants
$\unicon, \unicon'$ such that given any $\delta \in (0,1)$ and a
sample size $n \geq \unicon' d \log(1 / \delta)$, we have $\posterior
\Big( \enorm{ \theta - \thetastar} \geq \unicon
\parenth{\tfrac{d}{\numobs} + \tfrac{\log(\log(\numobs/
    \delta))}{\numobs} }^{1/ 4} +
\parenth{\tfrac{\boundconsprior}{\numobs}}^{1/3} \; \big| \;
X_{1}^{\numobs} \Big) \leq \delta$ with probability $1 - \delta$ over
the data $X_{1}^{n}$. Here, $B$ is the non-negative constant in
Assumption~\ref{item:smooth_log_prior}.
\end{corollary}
\noindent
See \cref{subsec:proof:cor:Gaussian_mixture} for the proof of Corollary~\ref{cor:Gaussian_mixture}. 

The dependence on $n$ in the posterior contraction rate of $\theta$ in
Corollary~\ref{cor:Gaussian_mixture} is consistent with the previous
result with location parameters in the overspecified Bayesian location
Gaussian mixtures~\cite{Chen1992, Ishwaran-2001, Nguyen-13}.  Novel
aspects of the bound include $d^{1/4}$-dependence on dimension $d$ and
the $B^{1/3}$-dependence on the smoothness parameter $B$.  Finally,
our result does not require the boundedness of the parameter space, in
contrast to past work~\cite{Chen1992, Ishwaran-2001, Nguyen-13}.




\section{Discussion}
\label{sec:discussion}

In this paper, we described an approach for analyzing the posterior
contraction rates of parameters based on the diffusion processes.  Our
theory depends on two important features: the convex-analytic
structure of the population log-likelihood function $\loglihood$ and
stochastic perturbation bounds between the gradient of $\loglihood$
and the gradient of its sample counterpart $\loglihood_{n}$.  We
studied the problem under both global and local assumptions on the
log-likelihood. For log-likelihoods that are globally strongly concave
around the true parameter $\thetastar$, we established posterior
convergence rates for parameter estimation of the order $(d/ n)^{1/
  2}$, valid under appropriate smoothness conditions on the prior
distribution $\prior$ and mild conditions on the perturbation error
between $\nabla \loglihood_{n}$ and $\nabla \loglihood$. On the other
hand, when the population log-likelihood function is globally weakly
concave, our analysis shows that convergence rates are more delicate:
they depend on an interaction between the degree of weak convexity,
and the stochastic error bounds.  In this setting, we proved that the
posterior convergence rate of parameter is upper bounded by the unique
positive solution of a non-linear equation determined by the previous
interplay.  We also provided results under weaker local conditions on
the growth of log-likelihood, and the empirical process defined by the
likelhood gradients over some neighborhood $\ball(\theta^{*}, r_{0})$
of the global maximum.  Finally, we demonstrated the utility of the
diffusion process approach by deriving non-asymptotic forms of
Bernstein-von Mises results for models with non-degenerate Fisher
information.

Let us now discuss a few directions that arise naturally from our
work.  First, in the weakly convex settting, though we have established non-asymptotic posterior contraction bounds, the current results do not provide information on the shape of the asymptotic posterior distribution. For example, when $\loglihood$ is locally strongly concave around
$\thetastar$, it is well-known from the Berstein-von Mises
theorem that the posterior distribution of parameter converges to a
multivariate normal distribution centered at the maximum likelihood
estimation (MLE) with the covariance matrix is given by $1/ \parenth{n
  I(\thetastar)}$ (e.g., see the book~\cite{vanderVaart-98}), where
$I(\thetastar)$ denotes the Fisher information matrix at $\thetastar$.
When the $\loglihood$ is only weakly concave, then the Fisher
information matrix $I( \thetastar)$ is degenerate, so that the
posterior distribution can no longer be approximated by a multivariate
Gaussian distribution.  It is interesting to consider how the
diffusion approach might provide insight into the
posterior behavior in this setting.

Second, the contraction rates given in this paper can give information
about the over-specification of the latent variable models, thereby
having potential applications for model selection.  As a concrete
example, for the symmetric two-component Gaussian mixture model
example discussed in \Cref{sec:over_Gaussian_mixture}, the posterior
distribution concentrates around $\thetastar = 0$ at a rate $O \big(
(d / n)^{1/4} \big)$ in the over-specified case.  On the other hand,
for a non-degenerate mixture with symmetric modes at $\thetastar$ and
$-\thetastar$ (with $\thetastar \neq 0$), it concentrates at the usual
rate $O \big( (d / n)^{1/2} \big)$.  Consequently, the degree of
dispersion in the posterior serves as an indicator of
over-specification. Furthermore, since our results are non-asymptotic,
they also give guidance on how this procedure could be performed with
finite sample size $n$.  Finally, whereas this paper focused on
posterior contraction for parametric models, we suspect that the
diffusion process approach used here might also be fruitfully applied
to non-parametric models.


\section*{Appendices}

\appendix

In our appendices, we provide the details of our general
theory applied to various examples, along with all details for the
proofs of our general results.  ~\Cref{sec:app-additional-examples}
covers the additional examples mentioned in the main text that serve
to illustrate the diffusion process approach to posterior
contraction. Additional general theory for the posterior convergence
rate of parameters when the population log-likelihood function is
non-convex is in~\cref{subsec:multiple_global_maxima}. The proofs of
theorems and propositions are given in~\cref{sec:proof}.  The proofs
of our main corollaries are in~\cref{sec:append_corollary_prof}, while
proofs of the remaining results in the paper are
in~\cref{subsec:auxiliary_results}.


\section{Additional examples}\label{sec:app-additional-examples}

This appendix continues the discussion of~\Cref{sec:examples},
providing consequence of our theorems for some additional
examples. Our discussion includes: Bayesian non-linear regression
models with polynomial link functions, general Bayesian Gaussian
mixture models, and one-dimensional location models with a singular
density function. These examples feature different aspects of the
diffusion process approach, covering local and global conditions, as
well as strongly and weakly concave log-likelihood functions.

\subsection{Bayesian non-linear regression models}
\label{sec:single_index}

We now turn to analyzing a certain type of non-linear regression
model, known as a single index model, but in a simplified form in
which the link function is known.  These models are a natural
generalization of linear regression, and have applications in
econometrics, biostatistics, and computational
imaging~\cite{Carroll-Hall-88}.

Given a collection of $d$-dimensional covariate vectors $\{X_i
\}_{i=1}^\numobs$, suppose that we observe responses of the form
\begin{align}
\label{eq:single_index}
    Y_{i} = g (X_{i}^{\top} \thetastar) + \epsilon_{i}, \qquad
    \mbox{for $i = 1, \ldots, \numobs$,}
\end{align}
for a known link function of the form $g$.  In the analysis given
here, we study the family $g(t) = t^{p}$ for some $p \geq 2$ given.
The special case $p = 2$ leads to an idealized instance of the problem
of noisy phase retrieval.

We assume moreover that the additive noise variables
$\{\epsilon_i\}_{i=1}^\numobs$ are i.i.d. and standard Gaussian,
whereas the covariate vectors $\{X_{i}\})_{i=1}^\numobs$ are also
i.i.d., independent of the noise, and standard multivariate Gaussian.
Conditioning on $X_{i}$ and $\theta$, we have
\begin{align}
\label{eq:Bayes_single_index}
Y_{i} \mid X_{i}, \theta \overset{\text{i.i.d.}}{\sim} \NORMAL
\parenth{ g \parenth{ X_{i}^{ \top} \theta}, 1}.
\end{align}

Moreover, we endow the parameter space $\real^d$ with a prior function
$\prior$ that satisfies the
Assumptions~\ref{item:smooth_population_condition}
and~\ref{item:smooth_log_prior}.  As in the previous example, we first
study the structure of the sample log-likelihood function around the
true parameter $\thetastar$, and then we establish a uniform
perturbation bound between the population and sample log-likelihood
functions.

Given the Bayesian single index model~\eqref{eq:Bayes_single_index},
the sample log-likelihood function $\loglihoodind_{n}$ of the samples
$Z_{1}^{n} = \{Z_{i}\}_{i = 1}^{n}$ admits the following form
\begin{align}
\label{eq:empi_likeli_index}
    \loglihoodind_{n}( \theta) \mydefn \frac{1}{n} \parenth{ \sum_{i =
        1}^{n} - \frac{\parenth{ Y_{i} - g \parenth{ X_{i}^{\top}
            \theta}}^2}{2} + \log \phi(X_{i})},
\end{align}
where $\phi$ is the standard normal density function of $X_{1},
\ldots, X_{n}$.  Hence, the population log-likelihood function
$\loglihoodind$ has the following form
\begin{align}
\label{eq:pop_likeli_index}    
    \loglihoodind( \theta) \mydefn \Exs_{(X,Y)} \brackets{ -
      \frac{\parenth{ Y - g\parenth{ X^{\top} \theta}}^2}{2} + \log
      \phi(X)},
\end{align}
where the outer expectation in the above display is taken with respect
to $X \sim \NORMAL(0, I_{d})$ and $Y| X = x \sim \NORMAL \parenth{g
  \parenth{x^{\top} \thetastar}, 1}$.

The interesting case to consider is $\thetastar = 0$, in which case,
for any link function of the function $g(t) = t^p$ with $p \geq 2$,
the function $\loglihoodind$ is weakly concave around $\thetastar$.
Given our choices of $g$ and $\thetastar$, the population log-
likelihood function takes on the closed-form expression
\begin{align*}
    \loglihoodind (\theta) = \frac{1 + (2 p - 1)!! \enorm{\theta -
        \thetastar}^{ 2 p}}{2} \qquad \mbox{for all $\theta
      \in \Rspace^{d}$.}
\end{align*}
Furthermore, in~\cref{subsec:proof:cor:single_index}, we prove that
there is a universal constant $\unicon_1 > 0$ such that
\begin{subequations}
\begin{align}
    \inprod{ \nabla \loglihoodind(\theta)}{ \thetastar - \theta} &
    \geq \unicon_{1} \enorm{ \theta - \thetastar}^{2 p} \quad
    \text{for all} \ \theta
    \in \Rspace^{d}, \label{eq:weak_conv_index}
\end{align}
and there are universal constants $(c, \unicon_2)$ such that for any
$r > 0$ and $\delta \in (0,1)$, as long as $n \geq \unicon \parenth{ d
  \log (d / \delta) }^{2 p}$, we have
\begin{align}
\sup_{\theta \in \ball (\thetastar, r)} \vecnorm{\nabla
  \loglihoodind_{n} (\theta) - \nabla \loglihoodind ( \theta) }{2} &
\leq \unicon_{2} \left( r^{p - 1} + r^{2p - 1} \right) \sqrt{\frac{d +
    \log(1/ \delta)}{n}}, \label{eq:empi_process_index}
\end{align}
\end{subequations}
with probability at least $1-\delta$. Therefore, the functions
$\weakcon$ and $\perturb$ in
Assumptions~\ref{item:weak_concavity_population}
and~\ref{item:weak_concavity_deviation} take the specific forms
\begin{align}
\label{eq:value_func_index}
\weakcon(r) = \unicon_{1} r^{2 p}, \quad \text{and} \quad \perturb(r)
= r^{p-1} + r^{2p - 1},
\end{align}
for all $r > 0$. Simple algebra shows that these functions satisfy
Assumptions~\ref{item:weak_concavity_growth_conditions}
and~\ref{item:weak_concavity_tail}.  With this set-up,
applying~\cref{theorem-main-weakly-convex} yields:
\begin{corollary}
\label{cor:single_index}
Consider the Bayesian single index model~\eqref{eq:single_index} with
true parameter $\thetastar = 0$ and link function $g(r) = r^{p}$ for
for some $p \geq 2$.  Then there are universal constants $\unicon,
\uniconprime$ such that for any $\delta \in (0,1)$, given a sample
size $n \geq \unicon' (d + \log(d / \delta))^{2p}$, we have
\begin{align*}
  \posterior \parenth{ \enorm{ \theta - \thetastar} \geq \unicon
    \parenth{\frac{d + \log(1 / \delta) + B}{n}}^{1/ (2 p)} \; \bigg|
    \; Z_{1}^{n}} \leq \delta
\end{align*}
with probability $1 - \delta$ over the data $Z_{1}^{n}$. Here, $B$ is the non-negative constant in Assumption~\ref{item:smooth_log_prior}.
\end{corollary}
\noindent
See~\cref{subsec:proof:cor:single_index} for the proof
of~\cref{cor:single_index}.

It is worth noting that the proof of~\cref{cor:single_index} actually
leads to the following stronger uniform perturbation bound:
\begin{multline*}
\sup_{\theta \in \ball (\thetastar, r)} \vecnorm{\nabla
  \loglihoodind_{n} (\theta) - \nabla \loglihoodind ( \theta) }{2}
\leq \unicon \; r^{p - 1} \left( \sqrt{\frac{d + \log
    \frac{1}{\delta}}{n} } + \frac{1}{n^{3/2}} \left( d + \log
\frac{n}{\delta} \right)^{p + 1} \right) \nonumber \\ + r^{2 p - 1}
\left( \sqrt{\frac{d + \log (1/\delta)}{n}} + \frac{1}{n^{3/2}} \left(
d + \log \frac{n}{\delta} \right)^{2 p + 1} \right),
\end{multline*}
valid for each $r > 0$ with probability $1 - \delta$.  The condition
$n \geq \unicon (d + \log(d / \delta))^{2p}$ is required to guarantee
that the RHS of the above display is upper bounded by the RHS of
equation~\eqref{eq:empi_process_index}; this bound permits us to
apply~\cref{theorem-main-weakly-convex} to establish the posterior
convergence rate of parameter under the Bayesian single index models.


\subsection{Bayesian Gaussian mixture models with multiple centers}\label{subsec:general-gauss-mixture}
We now consider a class of well-seperated location Gaussian mixture models. 
In particular, we consider $\mathrm{i.i.d.}$ data $X_1^\numobs$ from the mixture distribution $\frac{1}{K} \sum_{j = 1}^K \mathcal{N} \big( \gausscenterstar_j, I_d \big)$ 
for some $K \geq 2$ and $\gausscenterstar_1, \gausscenterstar_2, \cdots, \gausscenterstar_K$ are distinct parameters. 
We use the following Bayesian mixture model to fit the data:
\begin{align} 
\label{eq:Gaussian_mixture_general} 
    \gausscenter_1, \gausscenter_2, \cdots, \gausscenter_K \overset{\text{i.i.d.}}{\sim} \prior( \gausscenter), \quad c_{i} \in [K]
    \overset{\text{i.i.d.}}{\sim} \text{Cat}(1/K, 1/K, \cdots, 1/K), \quad
    X_{i}| c_{i}, \theta \overset{\text{i.i.d.}}{\sim} \NORMAL(
    \theta_{c_i}, I_{d}).
\end{align}
This model is well-specified in the sense that the true model belongs
to the class of models being considered, and the number of components
in the model equals the true number of (distinct) components.  The
model is identifiable only up to permutation of the labels, as there
are $M = K!$ many global minima of the population-level log-likelihood
that parametrizes the same probability distribution. Given a
permutation function $\sigma: [K] \rightarrow [K]$, we denote
$\thetastar_\sigma \mydefn \big(\gausscenterstar_{\sigma (1)},
\gausscenterstar_{\sigma(2)}, \cdots,
\gausscenterstar_{\sigma(K)}\big)$.  As an application of the results
from~\Cref{sec:without_global_concavity}, we establish the posterior
contraction rate of parameters as well as the Bernstein-von-Mises
phenomena for this model around $\thetastar_\sigma$, for each
permutation function $\sigma$.

To state the corollary, for each permutation function $\sigma: [K] \rightarrow [K]$, we define the Fisher information matrix:
    \begin{align*}
        \HessianStar_\sigma \mydefn \Exs_{\thetastar} \left[ \nabla_\theta \log p (X; \thetastar_\sigma) \cdot \nabla_\theta \log p (X; \thetastar_\sigma)^\top  \right].
    \end{align*}
By symmetry, the matrices $\HessianStar_\sigma$ are permutations of each other. In particular, for $\sigma_1, \sigma_2$, we have:
\begin{align*}
    \HessianStar_{\sigma_1} = (I_d \otimes P_{\sigma}) \HessianStar_{\sigma_2}  (I_d \otimes P_{\sigma})^\top,
\end{align*}
where $P_\sigma$ is a $K \times K$ permutation matrix defined by the permutation $\sigma$, and $\otimes$ denotes the Kronecker product. 

When $(\gausscenterstar_j)_{j \in [K]}$, it is known (see, e.g.~\cite{ho2019singularity}) that the Fisher information is positive definite. We denote its smallest eigenvalue $\mu \mydefn \lambda_{\min} \left( \HessianStar_{Id} \right) > 0$. For notational convenience, we also introduce the notation
\begin{align*}
  \sigma_X \mydefn \sup_{u \in \sphere^{d - 1}, j \in [K]} \vecnorm{\inprod{X - \gausscenterstar_j}{u}}{\psi_2},
\end{align*}
where $\vecnorm{Y}{\psi_2}$ denotes the Orlicz $\psi_2$ norm
  for a random variable $Y$. We can see that $1 \leq \sigma_X \leq c \left(1 + \sup_{j, \ell \in [K]} \vecnorm{\gausscenterstar_j - \gausscenterstar_\ell}{2} \right) < + \infty$.

The log-likelihood of this mixture model can have multiple global maxima due to the symmetry. We use $\thetamap_\sigma$ to denote the one corresponding to $\thetastar_\sigma$:
\begin{align*}
    \thetamap_\sigma \mydefn \arg\min_{\theta \in \arg\max \loglihood_\numobs} \vecnorm{\theta - \thetastar_\sigma}{2}.
\end{align*}
\begin{corollary}\label{cor:general-gaussian-mixture}
    Under the mixture of K location Gaussian distributions~\eqref{eq:Gaussian_mixture_general}, there exists $n_{\min} > 0$ depending on $\thetastar$, such that for any $\delta, \vartheta, \offpar \in (0, 1)$, given sample size $n \geq n_{\min} \big( \log \delta^{-1} + \log \vartheta^{-1} \big)^2$, we have the following concentration bounds on the posterior distribution:
\begin{subequations}
    \begin{align}
        \posterior& \biggr[\text{There exists permutation} \ \sigma: \vecnorm{\theta - \thetastar_\sigma}{2} \leq \frac{c K \sigma_X}{\strongconvex} \sqrt{\frac{K d \log (K d) + \log \delta^{-1}}{\numobs}} \nonumber \\
        & \hspace{20 em} + c \sqrt{\frac{\log \vartheta^{-1}}{\strongconvex \numobs}} \mid \DataX \biggr] \geq 1 - \vartheta, \quad \mbox{and} \label{eq:general-mixture-posterior-contraction} \\
        \posterior & \biggr[\text{There exists permutation} \ \sigma: \vecnorm{\theta - \thetamap_\sigma}{\HessianStar_\sigma} \geq  (1 + \offpar) \frac{d}{n} \nonumber \\
        & \hspace{6 em} + c \frac{1 + \log \kappa (\HessianStar_\sigma)}{\offpar} \Big( \frac{\log \vartheta^{-1}}{\numobs} + \frac{a' (\log \vartheta^{-1} + \log \delta^{-1})^2}{\numobs^2} \Big)  \mid \DataX \biggr] \geq 1 - \vartheta,\label{eq:general-mixture-bvm-type-contraction}
    \end{align}
    where $a' > 0$ is a constant depending on $K, d$ and $\thetastar$, while $c > 0$ is a universal constant.
\end{subequations}
\end{corollary}
\noindent See~\cref{app:subsec-proof-general-gaussian-mixture} for the
proof of this corollary.


\subsection{Bayesian location families with singularities}
\label{sec:ibra-kham}

Consider a one-dimensional location family $\{ f (\cdot -
\theta)\}_{\theta \in \real}$, where $f$ is a density function with
respect to the Lebesgue measure on $\real$. It was observed by
Ibragimov and Khasminskii~\cite{ibragimov1979asymptotic} that the
discontinuities and singularities in the density function $f$ reveal
more information about the location parameter, leading to rates of
parameters even faster than the usual $n^{-1/2}$-rates of regular
models. In this section, we show how~\cref{thm:local-weak-convex} can
be used to obtain the optimal posterior concentration rates of
parameters in such models. For our example, we only consider the
singularity of the second type (see Chapter 6.1 in the
book~\cite{ibragimov1979asymptotic}), while a similar argument can be
applied to the case of discontinuities in the densities. We leave an
extension of our framework to the first and third type of
singularities for the future work.

Without loss of generality, we assume the singularity happens at $0$. Following~\cite{ibragimov1979asymptotic}, given $\beta \in (0, 1 / 2)$, we assume the following representation:
\begin{align}
    f (x) =
    h (x) \exp \left(\ell (x) |x|^\beta \right) \quad \quad \quad \forall \ x \in \Rspace.\label{eq:singularity-second-type}
\end{align}
The function $h$ is assumed to be everywhere differentiable, with the following quantitative assumption:
\begin{align}
    \singconstone \mydefn \sup_{x \in \real} \abss{\frac{\partial}{\partial x} \log h (x)} < + \infty. \label{eq:assume-lip-log-h-in-singular-models}
\end{align}

 The function $\ell$ is smooth except for possible discontinuity at $0$, with additional assumption that $|\ell (0^-)| + |\ell (0^+)| > 0$. The fluctuations in $\ell$ can be absorbed into the pre-factor $h (\cdot)$. In such case, without loss of generality, we can assume that:
\begin{align*}
    \ell (x) = a \bm{1}_{\{x < 0\}} + b \bm{1}_{\{x > 0\}}.
\end{align*}

By the translation invariance of location families, we assume $\thetastar = 0$ without loss of generality. Though the empirical process condition~\ref{item:without_global_deviation} does not generally hold for the gradient of log-likelihood of singular location families, the analysis can still be done via the smoothing technique. In particular, let $\loglihoodsing_\numobs \mydefn \frac{1}{\numobs} \sum_{i = 1}^\numobs \log f (X_i - \theta)$, we define:
\begin{align}
    \smoothloglihood_n (\theta) \mydefn \frac{1}{2 a_n} \int_{\theta - a_n}^{\theta + a_n} \loglihoodsing_n (z) dz, \quad \smoothloglihood (\theta) \mydefn \Exs_{\thetastar} \brackets{ \smoothloglihood_n (\theta)}.
\end{align}
We can then define the smoothed posterior distribution:
\begin{align*}
    \widetilde{\posterior} (\theta) = \widetilde{Z}^{-1} \prior (\theta) \cdot \exp (- n \smoothloglihood_n (\theta) ), \quad \mbox{where} \quad \widetilde{Z} \mydefn \int \prior (\theta) \cdot \exp (- n \smoothloglihood_n (\theta) ) d \theta.
\end{align*}
For simplicity of presentation, we assume that the prior distribution $\prior$ is supported on the interval $[-1, 1]$, and satisfies the smoothness condition~\ref{item:smooth_log_prior} on its support.  Then, we prove in Appendix~\ref{subsec:proof-ibragimov} that there
exist constants $q_1, q_2, q_3, \localradius > 0$ that depend on the
density function $f$ but independent of $\numobs$ and $a_\numobs$,
such that:
\begin{subequations}
\begin{align*}
    - \inprod{\theta}{\nabla \widetilde{F}^{S} (\theta)} & \geq q_1
    |\theta|^{1 + 2 \beta} - q_2 a_\numobs^{1 + 2 \beta}, \quad
    \mbox{for } \theta \in (- \localradius/2,
    \localradius/2), \\ 
    \sup_{\theta
      \in [-1, 1]} \abss{\nabla \widetilde{F}^{S} (\theta) - \nabla
      \widetilde{F}_{n}^{S} (\theta)} & \leq q_3 \left(
    a_\numobs^{\beta - 1/2} \sqrt{\frac{\log \numobs /
        \delta}{\numobs}} + a_\numobs^{\beta - 1} \frac{\log \numobs /
      \delta}{\numobs}
    \right), 
\end{align*}
\end{subequations}
with probability $1 - \delta$. Based on these results, an application
of Theorem~\ref{thm:local-weak-convex} with local concavity assumption on the population
log-likelihood function leads to the following result on the posterior
convergence rates of parameters under model with density
function~\eqref{eq:singularity-second-type}.

\begin{corollary}
\label{cor:ibragimov}
Given a Bayesian location model with density specified in
equation~\eqref{eq:singularity-second-type} with $\beta \in (0, 1/2)$,
under above setup, there exists a pair of constants $(q_0, q')$
depending on the function $f$, such that given any $\delta \in (0,
1)$, for $\numobs \geq q_0 \log^{1 + \frac{1}{2 \beta}} \delta^{-1}$,
we have the following bound with probability $1 - \delta$:
    \begin{align*}
      \forall \vartheta \in (0, 1), \quad  \widetilde{\posterior} \left( |\theta| > q' \numobs^{- \frac{1}{1 + 2 \beta}} \big( \log^{\frac{1}{2 \beta}} \frac{\numobs}{\delta} + \log^{\frac{1}{1 + 2 \beta}} \vartheta^{-1} \big) \mid X_1^n \right) \leq \vartheta.
    \end{align*}
\end{corollary}
\noindent See~\cref{subsec:proof-ibragimov} for the proof of this
corollary.

\section{Multiple global maxima setting} 
\label{subsec:multiple_global_maxima}

When the population log-likelihood is non-convex, there may be
multiple global maxima.  Nonetheless, given some conditions on the
form of non-convexity, it is possible to establish a contraction
result that allows for multiple global maxima as a consequence
of~\cref{thm:local-weak-convex}. \\

More concretely, suppose that there is a finite collection
$\globalmaxima = \{\thetastar_1, \thetastar_2, \cdots, \thetastar_M\}$
of global maxima of the population log-likelihood function
$\loglihood$, and that the following conditions are in force:
\begin{enumerate}[label={\bf{(C.\arabic*)}}]
    \item\label{item:first_geometric_conditions_landscape} There
      exists $\localradius > 0$, such that for any $j \in [M]$, the
      conditions~\ref{item:without_global_geometry}
      and~\ref{item:without_global_deviation} hold true with
      parameters $\big( \lowerweakconcavemu_j, \alpha_j, \beta_j,
      \varepsilon_j(\numobs, \delta), \varsigma_j \big)$.
    \begin{subequations}
    \item\label{item:second_geometric_conditions_landscape} The gap $\gap$ for the log-likelihood outside the radius $\localradius$ is strictly positive, i.e.,
    \begin{align}
       \gap \mydefn \inf \left\{ \loglihood (\thetastar_1) - \loglihood (\theta): \theta \in \bigcap_{j = 1}^M \ball^{c} (\thetastar_j, \localradius) \right\} > 0. \label{eq:gap-def-in-multiple-global-max}
    \end{align}
    Furthermore, for any $R > 0$ and $\delta > 0$, there exists $\widebar{\noise}_{\numobs, \delta} (R) > 0$ with $\lim_{n \rightarrow + \infty} \widebar{\noise}_{\numobs, \delta} (R) \rightarrow 0$, such that with probability $1 - \delta$, we have:
    \begin{align}
        \sup_{\theta \in \ball (0, R)} \abss{\loglihood (\theta) - \loglihood_\numobs (\theta)} \leq \widebar{\noise}_{\numobs, \delta} (R). \label{eq:empirical-process-large-radius-in-multiple-global-max}
    \end{align}
    \end{subequations}
    In addition, the prior density function satisfies the lower bound $\max_{j \in [M]} \prior (\thetastar_j) \geq \prior_0$.
    \item\label{item:third_geometric_conditions_landscape} For any $\delta > 0$, there exists a radius $R_\delta > 0$, such that with probability $1 - \delta$, the following bound holds true:
       \begin{subequations}
    \begin{align}
        \forall \theta \notin \ball (0, R_0), \quad \inprod{\nabla \loglihood_n (\theta)}{\theta} \leq 0.\label{eq:condition-for-nonconvex-tail-loglihood}
    \end{align}
Additionally, there exists $c_\prior > 0$ such that
\begin{align}
\label{eq:condition-for-nonconvex-tail-prior}      
\forall \theta \notin \ball (0, R_\delta)^{c}, \quad - \inprod{\nabla
  \log \prior (\theta)}{\theta} \geq c_\prior \vecnorm{\theta}{2}^2.
\end{align}
\end{subequations}
\end{enumerate}
The last condition requires the prior density $\prior$ to have
sub-Gaussian tail. This condition is satisfied, for example, by any
Gaussian density. We use this condition to simplify the arguments of
unbounded parameter space, with quite weak
assumption~\eqref{eq:condition-for-nonconvex-tail-loglihood} required
on the log-likelihood function itself. Under stronger conditions on
the log-likelihood (for example, when the right-hand-side of
equation~\eqref{eq:condition-for-nonconvex-tail-loglihood} is replaced
by a quadratic function), this requirement on the prior density can be
removed.

An unconditional posterior concentration result can then be
established under this setup. Recall that the quantities $\Delta_0$
and $\bar{\noise}_{n, \delta}$ are defined in
equations~\eqref{eq:gap-def-in-multiple-global-max}
and~\eqref{eq:empirical-process-large-radius-in-multiple-global-max},
and for each $j \in [M]$, the parameters $(\mu_j, \alpha_j, \beta_j,
\varepsilon_j (n, \delta), \varsigma_j)$ are the parameters
in Assumptions~\ref{item:weak_concavity_population}
and~\ref{item:weak_concavity_deviation}.

\begin{corollary}\label{cor:final-nonconvex}
    Under Assumptions~\ref{item:first_geometric_conditions_landscape},
    ~\ref{item:second_geometric_conditions_landscape},
    and~\ref{item:third_geometric_conditions_landscape}, with
    probability $1 - 3 \delta$, denote $\widetilde{r}_0 \mydefn
    \localradius \wedge \sqrt{\frac{\Delta_0}{8 \smooth}}$, for sample
    size satisfying the inequality
    \begin{align*}
        & \numobs \geq \frac{4}{\Delta_0} \left( \log (1/\vartheta) +
      \log \prior_0^{-1} + d \log \frac{d}{\widetilde{r}_0} \right),
      \quad \mbox{and} \\ & \bar{\varepsilon}_{\numobs, \delta}\left(
      c (R_\delta \log (1/\vartheta) + \sqrt{(d + \log
        (1/\vartheta)) / c_\prior} \right) < \frac{\Delta_0}{4},
    \end{align*}
    we have the posterior concentration result for any $\vartheta >
    0$:
    \begin{align*}
        \posterior \left( \bigcup_{j = 1}^M \ball (\thetastar_j,
        r_n^{(j)}) \; \bigg| \; \DataX \right) \geq 1 - \vartheta,
    \end{align*}
    where the radius $r_n^{(j)}$ is defined as:
    \begin{align*}
          r_n^{(j)} \mydefn \left(\frac{\log (1/\vartheta) + d}{n
            \lowerweakconcavemu_j} +
          \frac{\smalloffset_j}{\lowerweakconcavemu_j}
          \right)^{\frac{1}{\alpha_j + 1}} + \left( {\frac{2
              \varepsilon_j (n,
              \delta)}{\lowerweakconcavemu_j}}\right)^{\frac{1}{\alpha_j
              - \beta_j}}  +
          \left(\frac{\boundconsprior}{n \lowerweakconcavemu_j}
          \right)^{\frac{1}{\alpha_j}}.
    \end{align*}
\end{corollary}
\noindent
The proof of~\cref{cor:final-nonconvex} is
in~\Cref{subsec:proof:cor:final-nonconvex}.

The contraction radius $r_\numobs^{(j)}$ around each $\thetastar_j$
corresponds to the contraction radius $r_\numobs$
in~\cref{thm:local-weak-convex} with corresponding local conditions on
the log-likelihood. In many examples such as well-specified Bayesian
mixture models, the global maxima of the population log-likelihood are
permutations of each other (see our example
in~\Cref{subsec:general-gauss-mixture}), and the contraction radii for
each center $\thetastar_j$ are the same. In general, however, the
global maxima of the population-level log-likelihood landscape can
have different geometric behaviors, leading to different contraction
radii around different centers.


\section{Proofs}
\label{sec:proof}

This section is devoted to the proofs of our main theorems.

\subsection{Proof of~\cref{thm:local-weak-convex}}
\label{subsec:proof:thm:local-weak-convex}

We begin with some notation and definitions that are central to the
analysis.  First, we define the annulus $\annulus(\thetastar,
\localradius) \defn \ball(\thetastar, \localradius) \setminus \ball
(\thetastar, \localradius/2)$.  Second, we define a pair of functions
with domain $\real^\usedim$ as follows:
\begin{subequations}\label{eq:func-pair-defn-in-local-weak-convex-proof}
\begin{align}
  \Psi (\theta) \mydefn \begin{cases} \loglihood (\theta),
    & \hspace{-0.6 em} \theta \in \ball (\thetastar,
    \localradius/2),\\ \left( 2 - \frac{2\vecnorm{\theta -
        \thetastar}{2}}{\localradius} \right) \loglihood \left(
    \frac{\localradius (\theta - \thetastar)}{2 \vecnorm{\theta -
        \thetastar}{2}} \right) + \left( \frac{2\vecnorm{\theta -
        \thetastar}{2}}{\localradius} - 1 \right) (\loglihood
    (\thetastar) - \frac{\smooth \localradius^2}{2}), & \hspace{-0.6
      em} \theta \in \annulus(\thetastar, \localradius), \\ F
    (\thetastar) - \frac{\smooth}{2} \vecnorm{\theta -
      \thetastar}{2}^2, & \hspace{-0.6 em} \theta \in
    \ball^{c}(\thetastar, \localradius),
  \end{cases}\label{eq:psi-func-defn-in-local-weak-convex-proof}
\end{align}
and
\begin{align}
  \zeta_n (\theta) \mydefn \begin{cases} (\loglihood_n(\theta) -
    \loglihood (\theta)) - (\loglihood_n (\thetastar) - \loglihood
    (\thetastar)) &\theta \in \ball(\thetastar,
    r_0/2),\\ 2\frac{\localradius - \vecnorm{\theta -
        \thetastar}{2}}{\localradius} \zeta_n \left( \thetastar +
    \frac{\localradius}{2}\cdot \frac{\theta -
      \thetastar}{\vecnorm{\theta - \thetastar}{2}} \right) & \theta
    \in \annulus(\thetastar, \localradius), \\ 0 & \theta \in
    \ball^{c}(\thetastar, \localradius).\label{eq:zeta-func-defn-in-local-weak-convex-proof}
        \end{cases}
    \end{align}
\end{subequations}
A few comments to provide intuition are in order.  Inside the ball
$\ball (\thetastar, \localradius / 2)$, the function $\Psi$ is the
population log-likelihood, whereas the function $\zeta_n$ specifies a
``noise'' term that can be controlled using empirical process
methods. On the other hand, outside of the ball $\ball(\thetastar,
\localradius)$, the function $\Psi$ corresponds to a quadratic upper
bound on the population log-likelihood $F$, whereas the function
$\zeta_n$ is identically zero.  In the annulus region between the two
balls, we interpolate linearly between the two behaviors.
    
It can be verified that both $\Psi$ and $\zeta_n$ are almost
everywhere continuously differentiable and locally Lipschitz
functions. Moreover, a direct computation yields
\begin{align*}
  \inprod{\nabla \Psi (\theta)}{ \theta - \thetastar} = \begin{cases}
    \inprod{\nabla F(\theta)}{\theta - \thetastar}, & \theta \in \ball
    (\thetastar, \localradius/2),\\ \frac{2 \vecnorm{\theta -
        \thetastar}{2}}{\localradius} \left(\loglihood (\thetastar) -
    \frac{\smooth}{2} \localradius^2 - \loglihood (\frac{\localradius
      (\theta - \thetastar)}{2 \vecnorm{\theta - \thetastar}{2}})
    \right), & \theta \in \annulus(\thetastar,
    \localradius),\\ -\smooth \vecnorm{\theta - \thetastar}{2}^2, &
    \theta \in \ball^{c}(\thetastar, \localradius).
  \end{cases}
\end{align*}
By Assumption~\ref{item:without_global_geometry}, we have the
following inequalities:
\begin{align*}
  \inprod{\nabla \Psi (\theta)}{ \theta - \thetastar}
  \leq
  \begin{cases} - \lowerweakconcavemu \vecnorm{\theta -
      \thetastar}{2}^{\alpha + 1} + \smalloffset, & \theta \in \ball
    (\thetastar, \localradius/2),\\ - \frac{3 \smooth \localradius}{8}
    \vecnorm{\theta - \thetastar}{2}, & \theta \in
    \annulus(\thetastar, \localradius), \end{cases}
    \end{align*}
Based on the above bounds, we define the following function:
\begin{align*}
  \psi(r) \mydefn \begin{cases} \lowerweakconcavemu r^{\alpha + 1},& r
    \in [0, \localradius / 2],\\ 2 (\localradius - r)
    \lowerweakconcavemu \localradius^{\alpha} + \smooth \localradius
    (2r - \localradius), & r \in (\localradius / 2,
    \localradius],\\ \smooth r^2, & r > \localradius.
  \end{cases}
\end{align*}
Since $\alpha \geq 1$, it is clear that $\psi$ is a convex function,
and we have:
\begin{align*}
  \inprod{\nabla \Psi (\theta)}{\theta - \thetastar} \leq - \psi
  (\vecnorm{\theta - \thetastar}{2}) + \smalloffset \quad \quad
  \forall \theta \in \real^d.
\end{align*}
For the function $\zeta_n$, we have
\begin{align*}
  \abss{\inprod{\nabla \zeta_n (\theta)}{\theta - \thetastar}}
  \leq \begin{cases} \vecnorm{\nabla \loglihood_n (\theta)}{2} \cdot
    \vecnorm{\theta - \thetastar}{2}, & \theta \in \ball (\thetastar,
    \localradius / 2), \\
\frac{2 \vecnorm{\theta - \thetastar}{2}}{\localradius} \abss{ \zeta_n
  \left( \thetastar + \frac{\localradius}{2}\cdot \frac{\theta -
    \thetastar}{\vecnorm{\theta - \thetastar}{2}} \right)}, & \theta
\in \annulus(\thetastar, \localradius),\\ 0,& \theta \in
\ball^{c}(\thetastar, \localradius).
  \end{cases}
\end{align*}
Note that for $\theta \in \annulus(\thetastar, \localradius)$,
conditionally on the event that the inequality in
Assumption~\ref{item:without_global_deviation} holds, we have
the bound
\begin{align*}
\frac{2 \vecnorm{\theta - \thetastar}{2}}{\localradius} \abss{ \zeta_n
  \left( \thetastar + \frac{\localradius}{2}\cdot \frac{\theta -
    \thetastar}{\vecnorm{\theta - \thetastar}{2}} \right)} & \\
& \hspace{-4 em} \leq 2 \int_0^1 \vecnorm{(\nabla \loglihood_n -
  \nabla \loglihood_n ) (\gamma \theta + (1 - \gamma) \thetastar) }{2}
\cdot \vecnorm{\theta - \thetastar}{2} d \gamma \\
& \hspace{-4 em}  \leq 2 \varepsilon (n, \delta) \vecnorm{\theta -
  \thetastar}{2}^{\beta + 1}.
\end{align*}
Therefore, on the event that
Assumption~\ref{item:without_global_deviation} holds, we have
\begin{align}
 \label{eq:key_inequality}  
  \abss{\inprod{\nabla \zeta_n (\theta)}{\theta - \thetastar}} \leq 2
  \varepsilon(n, \delta) \vecnorm{\theta - \thetastar}{2}^{\beta + 1}
  \bm{1}_{\{\theta \in \ball (\thetastar, \localradius)\}}.
\end{align}
Now we consider the distribution $\widetilde{\posterior}_n$ given by
\begin{align*}
 \widetilde{\posterior}_n (\theta) \mydefn \widetilde{Z}_n^{-1} \exp \left(
 n\Psi (\theta) + n \zeta_n (\theta) + \log \prior (\theta) \right),
 \quad \quad \mbox{for all $\theta \in \real^d$,}
 \end{align*}
where $\widetilde{Z}_n$ is a normalizing constant. Intuitively, the
density function $\widetilde{\posterior}_\numobs$ is a ``localized''
version of the posterior distribution: the distribution
$\widetilde{\posterior}_\numobs$ inherits the local behavior of the
posterior $\posterior$ itself, while behavior as Gaussian outside this
local neighborhood. This allows us to capture the effect of local
geometry of the log-likelihood function, and apply an argument similar
to the proof of~\cref{theorem-main-weakly-convex}.
    
Within the ball $\ball (\thetastar, \localradius / 2)$, we have
\begin{align*}
\frac{\widetilde{\posterior}_n (\theta)}{\posterior (\theta | X_1^n)} =
\frac{Z_n}{\widetilde{Z}_n} \exp \left( - n (\loglihood_n (\thetastar) -
\loglihood (\thetastar)) \right),
    \end{align*}
which is a fixed quantity independent of $\theta$. So for $r_n <
\localradius / 2$, we obtain
\begin{align*}
  \posterior \left( \ball (\thetastar, \localradius / 2) | X_1^n
  \right)^{-1} \posterior \left( \ball (\thetastar, r_n) | X_1^n
  \right) & = \widetilde{\posterior}_n \left( \ball (\thetastar,
  \localradius / 2) \right)^{-1} \widetilde{\posterior}_n \left( \ball
  (\thetastar, r_n) \right) \\ & \geq \widetilde{\posterior}_n \left(
  \ball (\thetastar, r_n) \right) .
\end{align*}
Conditionally on $X_1^n$, the distribution $\widetilde{\posterior}_n$ can
be seen as the stationary distribution for the following It\^{o}
diffusion process:
\begin{align}
\label{eq:diffusion-localized-posterior}  
d \diffusionforposteriorTilde_t = \nabla \left(\Psi
(\diffusionforposteriorTilde_t) + \zeta_n
(\diffusionforposteriorTilde_t) + \frac{1}{n}\log \prior \right) dt
+ \sqrt{\frac{2}{\numobs}} d B_t, \quad \diffusionforposteriorTilde_0 =
\thetastar.
\end{align}
On the other hand, for $t \geq 0$, we define the Lyapunov function
$\lyap$ as:
\begin{align*}
  \lyap_t \mydefn \Exs
  \brackets{\vecnorm{\diffusionforposteriorTilde_t - \thetastar}{2}^{p
      - 2} \psi \left(\vecnorm{\diffusionforposteriorTilde_t -
      \thetastar}{2} \right)}.
\end{align*}
For $q \in (0, p - 1]$, we define $g_q (z) \mydefn z^{\frac{p - 2}{q}}
  \psi (z^{\frac{1}{q}})$ for all $z$.  Now, we claim that the
  one-dimensional function $g_{q}$ is strictly increasing and
  convex. Furthermore, we have
\begin{align}
  \Exs \vecnorm{\diffusionforposteriorTilde_T - \thetastar}{2}^p
  &\leq - p \int_0^T \lyap_t dt + \frac{p \boundconsprior}{n}
  \int_0^T g_{p - 1}^{-1} (\lyap_t) dt + 2 \varepsilon(n,
  \delta) \lowerweakconcavemu^{- \frac{p + \beta -1}{p + \alpha
      - 1}} \int_0^T \lyap_t^{\frac{p + \beta - 1}{p + \alpha -
      1}} dt \nonumber \\ & \hspace{13 em} + \left( p
  \smalloffset + \frac{p (p + d - 1)}{n} \right) \int_0^T g_{p -
    2}^{-1} (\lyap_t) dt, \label{eq:claim_second}
\end{align}
when $\alpha > \beta$.

Taking the above claim as given for the moment, let us now complete
the proof of the theorem. When $\alpha > \beta$, we define the
function $\phi_1$ as follows:
\begin{align*}
  \phi_1 (h) \mydefn - h + \frac{\boundconsprior}{n} g_{p -
    1}^{-1} (h) + 2 \varepsilon(n, \delta)
  \lowerweakconcavemu^{- \frac{p + \beta -1}{p + \alpha - 1}}
  h^{\frac{p + \beta - 1}{p + \alpha - 1}} + \big( \smalloffset
  + \frac{p + d - 1}{n} \big) g_{p - 2}^{-1} (h).
\end{align*}
By~\cref{lemma-integral-ineq-limit-control}, since both $\lim_{t
  \rightarrow + \infty} \Exs \vecnorm{\diffusionforposteriorTilde_t -
  \thetastar}{2}^p$ and $\lim_{t \rightarrow + \infty} \lyap_t$ exist,
we have
\begin{align*}
\lim_{T \rightarrow + \infty} \lyap_T \leq \inf \left\{ h > 0:
~\forall h' >h, \phi_1 (h) < 0 \right\}.
\end{align*}
Note that $\phi_1$ is a concave function and the equation $\phi_{1}(h)
= 0$ only admits two solutions.  One of them is $h = 0$ and the other
one is the RHS of the above bound of $\lim_{T \rightarrow + \infty}
\lyap_T$.  Therefore, if $\frac{\boundconsprior}{n
  \lowerweakconcavemu}\leq (\localradius / 2)^\alpha$, $\frac{p + d}{n
  \lowerweakconcavemu} \leq (\localradius / 2)^{\alpha + 1}$ and
$\frac{2 \varepsilon (n, \delta)}{\lowerweakconcavemu} \leq
(\localradius / 2)^{\alpha - \beta}$, we have
\begin{align*}
  \lim_{T \rightarrow + \infty} \lyap_T \leq \lyap_* \mydefn
  \lowerweakconcavemu^{- \frac{p - 1}{\alpha}}
  \left(\frac{\boundconsprior}{n} \right)^{\frac{p - 1 +
      \alpha}{\alpha}} \vee \lowerweakconcavemu^{- \frac{p -
      2}{\alpha + 1}} \left(\frac{p + d}{n} + \smalloffset
  \right)^{\frac{p - 1 + \alpha}{\alpha + 1}} \vee (2
  \varepsilon (n, \delta))^{\frac{p + \alpha - 1}{\alpha -
      \beta}} \lowerweakconcavemu^{- \frac{p + \beta - 1}{\alpha
      - \beta}}.
\end{align*}
An application of Jensen's inequality shows that
\begin{align*}
  \Exs_{\widetilde{\posterior}_n} \left( \vecnorm{\theta -
    \thetastar}{2}^{p - 1} \right) = \lim_{T \rightarrow + \infty}
  \Exs \vecnorm{\diffusionforposteriorTilde_T - \thetastar}{2}^{p - 1}
  \leq g^{-1} \left( \lim_{T \rightarrow + \infty} \lyap_T \right).
\end{align*}
Putting the above results together, for $ \lyap_* \leq
\lowerweakconcavemu (\localradius / 2)^{p - 1 + \alpha}$ we obtain
that
\begin{align*}
  \left( \Exs_{\widetilde{\posterior}_n} \vecnorm{\theta -
    \thetastar}{2}^{p - 1} \right)^{\frac{1}{p - 1}} \leq \left(
  \frac{\lyap_*}{\lowerweakconcavemu} \right)^{\frac{1}{p + \alpha -
      1}} \leq \left(\frac{\boundconsprior}{n \lowerweakconcavemu}
  \right)^{\frac{1}{\alpha}} \vee \left(\frac{p + d}{n
    \lowerweakconcavemu} + \smalloffset \right)^{\frac{1}{\alpha + 1}}
  \vee \left( {\frac{2 \varepsilon (n,
      \delta)}{\lowerweakconcavemu}}\right)^{\frac{1}{\alpha -
      \beta}}.
    \end{align*}
Hence, we obtain the conclusion of the theorem when $\alpha > \beta$.
    
    
\subsubsection{Proof of claim~\eqref{eq:claim_second}}

By It\^{o}'s formula, for any $p \geq 2$, we have
\begin{align}
& \Exs \vecnorm{\diffusionforposteriorTilde_T - \thetastar}{2}^{p}
  \leq p I_{1} + \frac{p}{n} I_{2} + p I_{3} + \frac{p (p + d - 1)}{n}
  I_{4}, \label{eq:key_inequality_first_regime}
\end{align}
where
\begin{align*}
I_{1} & \mydefn \Exs \int_0^T \inprod{\nabla \Psi
  (\widetilde{\Theta}_t^{(n)})}{\widetilde{\Theta}_t^{(n)} -
  \thetastar} \vecnorm{\widetilde{\Theta}_t^{(n)} - \thetastar}{2}^{p
  - 2} dt, \\
I_{2} & \mydefn \Exs \int_0^T \inprod{\nabla \log
  \prior}{\diffusionforposteriorTilde_t - \thetastar}
\vecnorm{\diffusionforposteriorTilde_t - \thetastar}{2}^{p - 2} dt, \\
I_{3} & \mydefn \Exs \int_0^T \abss{ \inprod{ \nabla \zeta_n (
    \widetilde{\Theta}_t^{(n)})}{ \widetilde{\Theta}_t^{(n)} - \thetastar}}
\cdot \vecnorm{\widetilde{\Theta}_t^{(n)} - \thetastar}{2}^{p - 2} dt,
\quad \mbox{and} \\
I_{4} & \mydefn \Exs \int_0^T \vecnorm{\widetilde{\Theta}_t^{(n)} -
  \thetastar}{2}^{p - 2} dt.
\end{align*}

\begin{subequations}
Beginning with the first term $I_1$, by using the properties of the
function $\Psi$, we have the bound
\begin{align}
  I_1 & \leq - \Exs \int_0^T \psi
  \left(\vecnorm{\diffusionforposteriorTilde_t - \thetastar}{2}
  \right) \vecnorm{\diffusionforposteriorTilde_t - \thetastar}{2}^{p -
    2} dt + \smalloffset \cdot \Exs \int_0^T
  \vecnorm{\diffusionforposteriorTilde_t - \thetastar}{2}^{p - 2} dt
  \nonumber \\
& = - \int_0^T \lyap_t dt + \smalloffset \int_0^T \Exs \left[
    \vecnorm{\diffusionforposteriorTilde_t - \thetastar}{2}^{p -
      2}\right] dt. \label{eq:bound_I1_first_regime}
\end{align}
Turning to the second term $I_2$, applying
Assumption~\ref{item:smooth_log_prior} yields the upper bound
$$I_2 \leq \boundconsprior \Exs \int_0^T
  \vecnorm{\diffusionforposteriorTilde_t - \thetastar}{2}^{p - 1}
  dt.$$ 
  Applying Jensen's inequality then leads to
\begin{align*}
  g_{p - 1}\left( \Exs \vecnorm{\diffusionforposteriorTilde_t -
    \thetastar}{2}^{p - 1} \right) & \leq \Exs g_{p - 1} \left(
  \vecnorm{\diffusionforposteriorTilde_t - \thetastar}{2}^{p - 1}
  \right) \\ & = \Exs \left(\vecnorm{\diffusionforposteriorTilde_t -
    \thetastar}{2}^{p - 2} \psi
  \left(\vecnorm{\diffusionforposteriorTilde_t - \thetastar}{2}
  \right) \right) = \lyap_t.
\end{align*}
Collecting the above results, we find that
\begin{align}
  I_2 \leq B_{2} \int_0^T g_{p - 1}^{-1} (\lyap_t) dt. \label{eq:bound_I2_first_regime}
\end{align}
For the third term $I_3$, from the bound~\eqref{eq:key_inequality} we have
\begin{align*}
  I_3 \leq 2 \varepsilon (n, \delta) \cdot \Exs \int_0^T
  \vecnorm{\diffusionforposteriorTilde_t - \thetastar}{2}^{p + \beta -
    1} \bm{1}_{\{\diffusionforposteriorTilde_t \in \ball (\thetastar,
    \localradius)\}} dt.
\end{align*}
Since $\alpha > \beta$, invoking Jensen's inequality leads to
\begin{align*}
  \Exs \left(\vecnorm{\diffusionforposteriorTilde_t -
    \thetastar }{2}^{\beta + p - 1}
  \bm{1}_{\{\diffusionforposteriorTilde_t \in \ball
    (\thetastar, \localradius)\}} \right) & \leq \left( \Exs
  \left(\vecnorm{\diffusionforposteriorTilde_t - \thetastar
  }{2}^{\alpha + p - 1} \bm{1}_{\{\diffusionforposteriorTilde_t
    \in \ball (\thetastar, \localradius)\}} \right)
  \right)^{\frac{\beta + p - 1}{\alpha + p - 1}}\\ & \hspace{-
    2 em} \leq \lowerweakconcavemu^{- \frac{p + \beta - 1}{p +
      \alpha - 1} } \Exs \left( \psi\left(
  \vecnorm{\diffusionforposteriorTilde - \thetastar}{2} \right)
  \vecnorm{\diffusionforposteriorTilde - \thetastar}{2}^{p - 2}
  \right)^{\frac{p + \beta - 1}{p + \alpha - 1}} \\ & \hspace{-
    2 em} = \lowerweakconcavemu^{- \frac{p + \beta - 1}{p +
      \alpha - 1} } \lyap_t^{\frac{p + \beta - 1}{p + \alpha -
      1}}.
\end{align*}
Consequently, the term $I_{3}$ is upper bounded as
\begin{align}
\label{eq:bound_I3_first_regime}  
I_3 \leq 2 \varepsilon(n, \delta) \lowerweakconcavemu^{- \frac{p +
    \beta -1}{p + \alpha - 1}} \int_0^T \lyap_t^{\frac{p + \beta -
    1}{p + \alpha - 1}} dt.
\end{align}
For the fourth term $I_{4}$, invoking Jensen's inequality yields
\begin{align*}
g_{p - 2}\left( \Exs \vecnorm{\diffusionforposteriorTilde_t -
  \thetastar}{2}^{p - 2} \right) & \leq \Exs g_{p - 2} \left(
\vecnorm{\diffusionforposteriorTilde_t - \thetastar}{2}^{p - 2}
\right) \nonumber \\
& = \Exs \left(\vecnorm{\diffusionforposteriorTilde_t -
    \thetastar}{2}^{p - 2} \psi
  \left(\vecnorm{\diffusionforposteriorTilde_t - \thetastar}{2}
  \right) \right) = \lyap_t.
\end{align*}
The above inequality shows that
\begin{align}
  I_{4} \leq \int_0^T g_{p - 2}^{-1} (\lyap_t)
  dt. \label{eq:bound_I4_first_regime}
\end{align}
\end{subequations}

Collecting the bounds for $I_1$---$I_4$ given in
equations~\eqref{eq:bound_I1_first_regime}--\eqref{eq:bound_I4_first_regime}
respectively, we find that $\Exs
\vecnorm{\diffusionforposteriorTilde_T - \thetastar}{2}^p$ is at most
\begin{multline*}
- p \int_0^T \lyap_t dt + \frac{p \boundconsprior}{n} \int_0^T g_{p -
  1}^{-1} (\lyap_t) dt \\
+ 2 \varepsilon(n, \delta) \lowerweakconcavemu^{- \frac{p + \beta
    -1}{p + \alpha - 1}} \int_0^T \lyap_t^{\frac{p + \beta - 1}{p +
    \alpha - 1}} dt + \left(p \smalloffset + \frac{p (p + d - 1)}{n}
\right) \int_0^T g_{p - 2}^{-1} (\lyap_t) dt.
\end{multline*}
Thus, we have established the claim~\eqref{eq:claim_second}.

\subsubsection{Structure of the function $g_q$}

For $q \in (0, p - 1]$, we define $g_q (z) \mydefn z^{\frac{p -
      2}{q}}\psi (z^{\frac{1}{q}})$ for all $z$.  Since $\psi$ is
  strictly increasing and $p \geq 2$, we can check that $g_{q}$ is
  strictly increasing.  By taking the derivative of $g_{q}$, we have
\begin{align*}
  \frac{dg_q (z)}{dz} = \frac{p - 2}{q} z^{\frac{p - q - 1}{q}}
  \frac{\psi (z^{\frac{1}{q}})}{z^{\frac{1}{q}}} + \frac{1}{q}
  z^{\frac{p - 1 - q}{q}} \psi' (z^{\frac{1}{q}}).
\end{align*}
By the construction, $\psi$ is a convex function on $\real_+$ and
$\psi(0) = 0$.  Therefore, $\psi'$ is non-decreasing, and therefore
$\frac{1}{r} \psi (r) = \frac{1}{r} \int_0^r \psi' (s) ds$ is also
non-decreasing.  For $q \leq p - 1$, the function
$z^{\frac{p-q-1}{q}}$ is also non-decreasing in $z$, and apparently,
for $r \geq 0$, both $\psi' (r)$ and $\psi (r) / r$ are
non-negative. Therefore, for any $q \in (0, p - 1]$, the function
  $\frac{dg_q}{dz}$ is non-decreasing in $z$. Therefore, $g_q$ is a
  convex function.
  

\subsection{Proof of~\cref{thm:non-asymp-credible-set}}
\label{subsec:proof:thm:non-asymp-credible-set}

For any fixed $T > 0$, we define the sequence of potential functions
$\lyap_t: \real^d \rightarrow \real$
\begin{align*}
  \lyap_t(\theta) & \mydefn (\theta - \thetamap)^\top \HessianStar
  e^{\HessianStar (t - T)} (\theta - \thetamap), \quad \mbox{for each
    $t \in [0, T]$.}
\end{align*}
Once again, we consider the diffusion process with the initial
condition $\theta_0 = \thetamap$:
\begin{align*}
    d \theta_t = - \nabla \loglihood_\numobs (\theta_t) dt +
    \frac{1}{\numobs} \nabla \log \prior (\theta_t) dt + d B_t.
\end{align*}
Using It\^{o}'s formula, for $t \in [0, T]$, we have
\begin{align}
  \lyap_t (\theta_t) & = \int_0^t \frac{\partial \lyap_s}{\partial s}
  (\theta_s) ds - \int_0^t \inprod{\nabla \lyap_s (\theta_s)}{ \nabla
    F_n (\theta_s) - \frac{\nabla \log \prior (\theta_s)}{n} } ds
  \nonumber \\
  & \qquad + \sqrt{\frac{2}{n}}\int_0^t \inprod{\nabla \lyap_s
    (\theta_s)}{d B_s} + \frac{1}{n}\int_0^t \Delta \lyap_s (\theta_s)
  ds \nonumber \\
& = \underbrace{\int_0^t \left( \HessianStar (\theta_s - \thetamap) -
    \nabla F_n (\theta_s) + \frac{\nabla \log \prior (\theta_s)}{n}
    \right)^\top \HessianStar e^{\HessianStar (s - T)} (\theta_s -
    \thetamap) ds}_{\mydefn I_1 (t)} \nonumber \\
\label{eq:ito-formula-in-non-asymp-credible-set}  
& \qquad + \underbrace{\sqrt{\frac{2}{n}} \int_0^t (\theta_s -
  \thetamap)^\top \HessianStar e^{(s - T) \HessianStar} d B_s}_{I_2
  (t)} + \underbrace{\frac{1}{n} \int_0^t \mathrm{Tr} \left(
  \HessianStar e^{\HessianStar (s - T)} \right) ds}_{I_3 (t)}.
\end{align}
Note that the matrices $\HessianStar$ and $e^{(s - T) \HessianStar}$
commute, so that we may write their product in an arbitrary order.

Defining the linearization error
\begin{align*}
\Delta_s \mydefn (\holderconst + \noiseone^{(2)} (n, \delta)) \left(
\vecnorm{\theta_s - \thetastar}{2} + \vecnorm{\thetamap -
  \thetastar}{2} \right) + \noisetwo^{(2)} (n, \delta) +
\frac{\smoothprior}{n},
\end{align*}
we claim that the following bounds hold for each $t \in [0, T]$:
\begin{subequations}
\begin{align}
I_{1}(t) & \leq \tfrac{2 + \log \kappa (\HessianStar)}{a} \sup_{0 \leq
  s \leq t} \lyap_s (\theta_s) \nonumber \\
  & \hspace{ 8 em} + a \int_0^t \Delta_s^2 \left(
\vecnorm{\theta_s - \thetastar}{2}^2 + \vecnorm{\thetamap -
  \thetastar}{2}^2 \right) e^{- \frac{\lambda_{\min}
    (\HessianStar)}{2} (s - T)} ds, \label{eq:I1-bound-in-non-asymp-credible-set} 
\end{align}
\begin{align}
\label{eq:I2-bound-in-non-asymp-credible-set}   
\left( \Exs \sup_{0 \leq t \leq T} |I_2 (t)|^p \right)^{1/p} & \leq
\unicon \sqrt{\frac{p \left(1 + \log \kappa (\HessianStar)
    \right)}{n}} \left( \Exs \sup_{0\leq t \leq T} \lyap_t
(\theta_t)^{p/2} \right)^{1/p}, \quad \mbox{and} \\
\label{eq:I3-bound-in-non-asymp-credible-set}    
I_3 (t) & \leq \frac{d}{n}.
\end{align}
\end{subequations}
Here $\unicon > 0$ is an universal constant. We prove all of these
bounds in the subsections to follow.

Taking these bounds as given for the moment, let us complete the proof
of the theorem. By Jensen's inequality, for an even integer $p \geq
2$, the moments of the integral term in
equation~\eqref{eq:I1-bound-in-non-asymp-credible-set} can be bounded
as
\begin{multline}
\label{eq:linearization-error-integral-in-non-asymp-credible-set}  
\Exs \left( \int_0^T \Delta_s^2 \left( \vecnorm{\theta_s -
  \thetastar}{2}^2 + \vecnorm{\thetamap - \thetastar}{2}^2 \right)
e^{- \frac{\lambda_{\min} (\HessianStar)}{2} (s - T)} ds
\right)^p\\ \leq \left( \frac{c}{\lambda_{\min} (\HessianStar)}
\right)^{p - 1} \cdot \Exs \int_0^T \Delta_s^{2p} \left(
\vecnorm{\theta_s - \thetastar}{2}^{2p} + \vecnorm{\thetamap -
  \thetastar}{2}^{2p} \right) e^{- \frac{\lambda_{\min}
    (\HessianStar)}{2} (s - T)} ds,
\end{multline}
for a universal constant $c > 0$.

For any $\offpar \in (0, 1)$, by taking supremum on both sides of the
decomposition~\eqref{eq:ito-formula-in-non-asymp-credible-set},
combining with the
bounds~\eqref{eq:I1-bound-in-non-asymp-credible-set}
and~\eqref{eq:I3-bound-in-non-asymp-credible-set}, and taking $a =
c\frac{2 + \log \kappa (\HessianStar)}{\offpar}$, we arrive at the
inequality
\begin{multline*}
    \sup_{0 \leq t \leq T} \lyap_t (\theta_t) \leq (1 + \offpar)
    \left( \frac{d}{n} + \sup_{0 \leq t \leq T} I_2 (t) \right)\\ +
    \frac{c(2 + \log \kappa (\HessianStar))}{\offpar} \int_0^T
    \Delta_t^2 \left( \vecnorm{\theta_t - \thetastar}{2}^2 +
    \vecnorm{\thetamap - \thetastar}{2}^2 \right) e^{-
      \frac{\lambda_{\min} (\HessianStar)}{2} (t - T)} dt.
\end{multline*}
Taking $p$-th moment on both sides of the inequality, combining with
the bounds~\eqref{eq:I2-bound-in-non-asymp-credible-set}
and~\eqref{eq:linearization-error-integral-in-non-asymp-credible-set},
and applying Minkowski's inequality, we arrive at the bound
\begin{multline*}
    \left( \Exs \sup_{0 \leq t \leq T} \lyap_t (\theta_t)^p
    \right)^{1/p} \leq (1 + \offpar) \frac{d}{n} + \sqrt{\frac{ c p (1
        + \log \kappa (\HessianStar))}{n}} \cdot \left( \Exs \sup_{0
      \leq t \leq T} \lyap_t (\theta_t)^p \right)^{\frac{1}{2p}}\\ +
    \frac{c(2 + \log \kappa (\HessianStar))}{\offpar \lambda_{\min}
      (\HessianStar)} \left( \sup_{0 \leq t \leq T} \Exs \left[
      \Delta_t^{2p} \left( \vecnorm{\theta_t - \thetastar}{2}^{2p} +
      \vecnorm{\thetamap - \thetastar}{2}^{2p} \right) \right]
    \right)^{1/p}.
\end{multline*}
Substituting with the definition of the last term, and applying
Young's inequality, we find that
\begin{align*}
     \left( \Exs \sup_{0 \leq t \leq T} \lyap_t (\theta_t)^p
     \right)^{1/p} \leq (1 + \offpar) \frac{ d}{n} + c \frac{1 + \log
       \kappa (\HessianStar)}{\offpar} \left( \frac{p}{n} +
     \frac{\mathcal{H}_n (p, \delta)}{\lambda_{\min} (\HessianStar)}
     \right),
\end{align*}
where the high-order term $\mathcal{H}_n (p, \delta)$ is defined as
\begin{align*}
    \mathcal{H}_n (p, \delta) \mydefn & (\holderconst +
    \noiseone^{(2)} (n, \delta))^2 \left( \Exs_\posterior
    \vecnorm{\theta - \thetastar}{2}^{4p} \right)^{1/p}\\ &\quad\quad+
    \vecnorm{\thetamap - \thetastar}{2}^2 \left( \noisetwo^{(2)} (n,
    \delta)^2 + \frac{\smoothprior^2}{n^2} + (\holderconst +
    \noiseone^{(2)} (n, \delta))^2 \vecnorm{\thetamap -
      \thetastar}{2}^2 \right).
\end{align*}
Putting together the pieces yields the conclusion of the theorem.


\subsubsection{Proof of claim~\eqref{eq:I1-bound-in-non-asymp-credible-set}}

We first bound the term $I_1(t)$.  Noting the defining identity
$\nabla F_n (\thetamap) + \frac{1}{n} \nabla \log \prior (\thetamap) =
0$, we have the following bound:
\begin{align*}
    &\vecnorm{\HessianStar (\theta_s - \thetamap) - \nabla F_n (\theta_s) + \nabla \log \prior (\theta_s) / n}{2}\\
    &= \vecnorm{\int_0^1 \left( \HessianStar  - \nabla^2 F_n \big( \gamma \theta_s + (1 - \gamma) \thetamap \big) 
    + \nabla^2 \log \prior \big( \gamma \theta_s + (1 - \gamma) \thetamap \big) / n \right) (\theta_s - \thetamap) d \gamma}{2}\\
    &\leq \int_0^1 \opnorm{ \HessianStar  - \nabla^2 F_n \big( \gamma \theta_s + (1 - \gamma) \thetamap \big) 
    + \nabla^2 \log \prior \big( \gamma \theta_s + (1 - \gamma) \thetamap \big) / n } \cdot \vecnorm{\theta_s - \thetamap}{2} d \gamma.
\end{align*}
By Assumptions~\ref{item:Bvm_population},~\ref{item:Bvm_deviation}, and~\ref{item:smooth_population_condition}, for any $\theta \in \real^d$, we have the bound
\begin{align*}
    & \hspace{- 3 em} \opnorm{ \HessianStar  - \nabla^2 F_n (\theta) + \nabla^2 \log \prior (\theta) / n }\\
    &\leq \opnorm{\HessianStar - \nabla^2 F (\theta)} + \opnorm{\nabla^2 F (\theta) - \nabla^2 F_n (\theta)} + \opnorm{\nabla^2 \log \prior (\theta) / n}\\
    &\leq \holderconst \vecnorm{\theta - \thetastar}{2} + \noiseone^{(2)} (n, \delta) \vecnorm{\theta - \thetastar}{2} + \noisetwo^{(2)} (n, \delta) + \frac{\smoothprior}{n}.
\end{align*}

Substituting into the bound for $I_1 (t)$, for any $a > 0$, we have that
\begin{align*}
    I_1 (t) & \leq \int_0^t \opnorm{(\HessianStar)^{1/2} e^{\HessianStar (s - t) / 2}} \\
    & \hspace{2 em} \times \vecnorm{\HessianStar  - \nabla^2 F_n (\theta_s) 
    + \nabla^2 \log \prior (\theta_s) / n }{2} \vecnorm{\theta_s - \thetamap}{2} \sqrt{\lyap_s (\theta_s)} ds \nonumber\\
   & \leq a^{-1} \sup_{0 \leq s \leq t} \lyap_s (\theta_s) \cdot \int_0^t \opnorm{(\HessianStar)^{1/2} e^{\HessianStar (s - T) / 4}}^2 ds \nonumber\\
   & \hspace{2 em} + a \int_0^t \opnorm{\HessianStar  - \nabla^2 F_n (\theta_s)^2 + \nabla^2 \log \prior (\theta_s) / n }^2 \cdot \vecnorm{\theta_s - \thetamap}{2}^2 \opnorm{e^{\HessianStar (s - T) / 4}}^2 ds \nonumber\\
   &\leq \frac{2 + \log \kappa (\HessianStar)}{a} \sup_{0 \leq s \leq t} \lyap_s (\theta_s) 
   \\
   & \hspace{2 em} + a \int_0^t \Delta_s^2 \left( \vecnorm{\theta_s - \thetastar}{2}^2 + \vecnorm{\thetamap - \thetastar}{2}^2 \right) e^{- \frac{\lambda_{\min} (\HessianStar)}{2} (s - T)} ds.
\end{align*}
Therefore, claim~\eqref{eq:I1-bound-in-non-asymp-credible-set} follows.
  
\subsubsection{Proof of claim~\eqref{eq:I2-bound-in-non-asymp-credible-set}}

Note that $I_2(t)$ is a martingale with respect to the Brownian
filtration. Applying the Burkholder-Gundy-Davis inequality for an
arbitrary $p \geq 2$ yields
\begin{align*}
\left( \Exs \sup_{0 \leq t \leq T} |I_2 (t)|^p \right)^{1/p} & \leq
\unicon \sqrt{\frac{p}{n}} \left( \Exs \left(\int_0^T
\vecnorm{\HessianStar e^{(t - T) \HessianStar} (\theta_t -
  \thetamap)}{2}^2 dt \right)^{\frac{p}{2}} \right)^{1/p}\\ &\leq C
\sqrt{\frac{p}{n}} \left( \Exs \left(\int_0^T
\opnorm{(\HessianStar)^{1/2} e^{\frac{t - T}{2} \HessianStar}}^2
\lyap_t (\theta_t) dt \right)^{\frac{p}{2}} \right)^{1/p} \\
& \leq \unicon \sqrt{\frac{p}{n}} \left( \Exs \sup_{0\leq t \leq T}
\lyap_t (\theta_t)^{p/2} \right)^{1/p} \cdot \sqrt{\int_0^T
  \opnorm{(\HessianStar)^{1/2} e^{\frac{t - T}{2} \HessianStar}}^2
  dt}.
\end{align*}

We now observe that
\begin{align*}
\opnorm{(\HessianStar)^{1/2} e^{\frac{t - T}{2} \HessianStar}}^2 =
\opnorm{\HessianStar e^{(t - T) \HessianStar}} = \max_{i \in [d]}
\left( \lambda_i (\HessianStar) e^{(t - T) \lambda_i (\HessianStar)}
\right).
\end{align*}
Taking the time integral leads to the bound
\begin{align*}
\int_0^T \opnorm{(\HessianStar)^{1/2} e^{\frac{t - T}{2}
    \HessianStar}}^2 dt & \leq \int_0^{+ \infty} \max_{i \in [d]}
\left( \lambda_i (\HessianStar) e^{- t \lambda_i (\HessianStar)}
\right) dt \\
& \leq \underbrace{\int_0^{+ \infty} \max_{\lambda_{\min}
    (\HessianStar) \leq \lambda \leq \lambda_{\max} (\HessianStar)}
  \left( \lambda e^{- t \lambda} \right) dt.}_{ = \, : J}
\end{align*}
We now split the integral $J$ into three parts, thereby obtaining
\begin{align}
J & \leq \int_0^{\lambda_{\max} (\HessianStar)^{-1}} \lambda_{\max}
(\HessianStar) e^{- t \lambda_{\max} (\HessianStar) } dt \nonumber \\
& \hspace{10 em} +
\int_{\lambda_{\max} (\HessianStar)^{-1}}^{{\lambda_{\min}
    (\HessianStar)^{-1}}} \frac{dt}{e t} + \int_{{\lambda_{\min}
    (\HessianStar)^{-1}}}^{+ \infty} \lambda_{\min} (\HessianStar)
e^{- t \lambda_{\min} (\HessianStar) } dt \nonumber \\
& \leq 1 + \frac{1}{e} \log \frac{\lambda_{\max}
  (\HessianStar)}{\lambda_{\min} (\HessianStar)}. \label{eq:matrix-integral-log-bound}
\end{align}
Denote $\kappa(M) \mydefn \frac{\lambda_{\max} (M)}{\lambda_{\min}
  (M)}$ for a positive definite matrix $M$. Collecting the above
inequalities, we find that the term $I_2(t)$ is upper bounded as
\begin{align*}
\left( \Exs \sup_{0 \leq t \leq T} |I_2 (t)|^p \right)^{1/p} \leq
\unicon \sqrt{\frac{p \left(1 + \log \kappa (\HessianStar)
    \right)}{n}} \left( \Exs \sup_{0\leq t \leq T} \lyap_t
(\theta_t)^{p/2} \right)^{1/p}
\end{align*}
for a universal constant $\unicon > 0$.  This completes the proof of
the claim~\eqref{eq:I2-bound-in-non-asymp-credible-set}.


\subsubsection{Proof of claim~\eqref{eq:I3-bound-in-non-asymp-credible-set}}

Finally, the term $I_3 (t)$ is straightforward to upper bound as
\begin{align*}
I_3(t) \leq \frac{1}{n} \mathrm{Tr} \left( \HessianStar \int_0^T
e^{\HessianStar (s - T)} ds \right) \leq \frac{1}{n} \mathrm{Tr}
\left( \HessianStar \int_0^{+ \infty} e^{- s\HessianStar} ds \right) =
\frac{d}{n},
\end{align*}
which establishes the
claim~\eqref{eq:I3-bound-in-non-asymp-credible-set}.


\label{sec:append_remain_theorem}

In this Appendix, we provide proofs of remaining theorems and
propositions in the main text.

\subsection{Proof of~\cref{thm-one-point-SC}}
\label{subsec:strongly_convex_proof}

Throughout the proof, in order to simplify notation, we omit the
conditioning on the $\sigma$-field \mbox{$\mathcal{F}_n \mydefn
  \sigma(\DataX)$;} it should be taken as given.  For $\constrong =
\frac{1}{2} \strongconvex - \noiseone(n, \delta) >
\frac{\strongconvex}{6}$, we claim that
\begin{align}
\label{eq:key_claim_strong_concave}
    \frac{1}{2} e^{\constrong t} \vecnorm{\theta_t - \thetastar}{2}^2 
    \leq \frac{1}{\sqrt{n}} M_t + U_{n} \frac{(e^{\constrong t} - 1)}{ 2 \constrong },
\end{align}
where $U_n \mydefn \frac{3 \boundconsprior^2}{n^2} +
 \frac{3 \noisetwo^2 (n, \delta)}{
  \strongconvex} + \frac{d}{n}$ and $M_t \mydefn \int_0^t
e^{\constrong s} \inprod{\theta_s - \thetastar}{ dB_s}$, which is a martingale. 

Assume that the above claim is given at the moment (the proof of that
claim is deferred to the end of the proof of the proposition). In order to
bound the moments of martingale $M_t$, for any $p \geq 4$, we invoke
the Burkholder-Gundy-Davis inequality~\cite{MR1725357} to find that
\begin{align*}
  \Exs \brackets{ \sup_{0 \leq t \leq T} |M_t|^{\frac{p}{2}}} \leq (p
  C)^{ \frac{p}{4} } \mathbb{E} \brackets{ \quadvar{M}{T}^{
      \frac{p}{4}}} & = (p C)^{ \frac{p}{4}} \mathbb{E} \left(\int_0^T
  e^{2 \constrong s} \vecnorm{\theta_s - \thetastar}{2}^2 ds
  \right)^{\frac{p}{4}} \\ & \leq (p C)^{\frac{p}{4}} \mathbb{E}
  \left( \sup_{0 \leq t \leq T} e^{\constrong t} \vecnorm{\theta_t -
    \thetastar}{2}^2 \int_0^T e^{\constrong s} ds\right)^{\frac{p}{4}} \\ &
  \leq \parenth{ \frac{p C e^{\constrong T}}{\constrong}}^{\frac{p}{4}} \mathbb{E}\left(
  \sup_{0 \leq t \leq T} e^{\constrong t} \vecnorm{\theta_s - \thetastar}{2}^2
  \right)^{\frac{p}{4}},
\end{align*}
where $C$ is a universal constant. Therefore, we arrive at the following bound:
\begin{align*}
    \Exs  \brackets{ \left( \sup_{0 \leq t \leq T} e^{\constrong t} \vecnorm{\theta_t - \thetastar}{2} \right)^p }  &
    \leq \Exs \left(\frac{2}{\sqrt{n}} M_t \right)^{\frac{p}{2}} +
    \left( U_n\frac{(e^{\constrong T} - 1)}{ \constrong } \right)^{\frac{p}{2}} \\ & \leq
    \left(U_n \frac{ e^{\constrong T}}{ \constrong}\right)^{\frac{p}{2}} + \left(\frac{p
      C e^{\constrong T} }{\constrong n} \right)^{\frac{p}{4}} \mathbb{E}\left( \sup_{0
      \leq s \leq T} e^{\constrong s} \vecnorm{\theta_s - \thetastar}{2}^2
    \right)^{\frac{p}{4}}.
\end{align*}
For the right hand side of the above inequality, we can relate it to
the left hand side by using Young's inequality, which is given by
\begin{align*}
  \left(\frac{p C e^{\constrong T} }{\constrong n} \right)^{\frac{p}{4}}
        \mathbb{E} \left( \sup_{0 \leq s \leq T} e^{\constrong s}
        \vecnorm{\theta_s - \thetastar}{2}^2 \right)^{\frac{p}{4}}
        \leq \frac{1}{2}\left(\frac{p C e^{\constrong T} }{\constrong n}
        \right)^{\frac{p}{2}} + \frac{1}{2} \mathbb{E}\left( \sup_{0
          \leq s \leq T} e^{\constrong s} \vecnorm{\theta_s - \thetastar}{2}^2
        \right)^{\frac{p}{2}}.
\end{align*}
Putting the above results together, and let $\alpha = \frac{\mu}{2}$, we find that
\begin{align*}
    \left( \mathbb{E} \brackets{ \vecnorm{ \theta_T - \thetastar
      }{2}^p } \right)^{ \frac{1}{p} } \leq e^{- \constrong T} \left(
    \mathbb{E} \sup_{0 \leq t \leq T}\left( e^{\constrong t} \vecnorm{ \theta_t
      - \thetastar }{2}^p\right) \right)^{ \frac{1}{p} } \leq
    C'\left(\sqrt{\frac{U_n}{ \strongconvex}} + \sqrt{\frac{2 p }{n
        \mu }}\right),
\end{align*}
for universal constant $C' > 0$. Therefore, the
  diffusion process defined in equation~\eqref{eq-diffusion-main} satisfies the
  following inequality
  \begin{align*}
    \sup_{t \geq 0} \parenth{ \mathbb{E} \brackets{ \vecnorm{ \theta_t
          - \thetastar } {2}^p} }^{ \frac{1}{p} } \leq c \;
    \left( \sqrt{\frac{ d }{ \strongconvex n}} + \frac{\boundconsprior}{\strongconvex n} +
    \frac{\noisetwo(n, \delta)}{\strongconvex} + \sqrt{\frac{p}{n
        \strongconvex}} \right)
  \end{align*}
  for any $p \geq 1$. Combining the above inequality with the
  inequality~\eqref{eq:key_claim_weakly_concave} yields the conclusion of the
  proposition.

\vspace{0.5 em}
\noindent
\textit{Proof of claim~\eqref{eq:key_claim_strong_concave}:} For the given choice $\constrong > 0$, an application of It\^{o}'s
formula yields the decomposition
\begin{align}
  \frac{1}{2} e^{\constrong t} \vecnorm{\theta_t - \thetastar}{2}^2 = & -
  \frac{1}{2} \int_0^t \inprod{\thetastar - \theta_s}{ \nabla
    \loglihood_{n} (\theta_s) e^{\constrong s} } ds + \frac{1}{2 n} \int_0^t
  \inprod{ \theta_s - \thetastar}{\nabla \log \prior (\theta_s ) e^{\constrong
      s}} ds \notag \\ + & \frac{d}{2 n} \int_0^t e^{\constrong s} ds +
  \frac{1}{\sqrt{n}}\int_0^t e^{\constrong s} \inprod{\theta_s - \thetastar}{
    dB_s} + \frac{1}{2} \int_0^t \constrong e^{\constrong s} \vecnorm{ \theta_s -
    \thetastar }{2} ^2ds \nonumber \\ = & J_{1} + J_{2} + J_{3} +
  J_{4} + J_{5}. \label{eq-proof-geomtry-convex-ito}
\end{align}

We begin by bounding the term $J_{1}$ in
equation~\eqref{eq-proof-geomtry-convex-ito}.  Based on
Assumption~\ref{item:strong_concavity_deviation} regarding the
perturbation error between $\loglihood_{n}$ and $\loglihood$ and the
strong convexity of $\loglihood$, we have
\begin{align*}
& \hspace{- 1 em} J_{1} = - \frac{1}{2} \int_0^t \inprod{ \thetastar -
    \theta_s}{ \nabla \loglihood_{n} (\theta_s) e^{\constrong s} } ds
  \nonumber \\ \leq & - \frac{1}{2} \int_0^t \inprod{ \thetastar -
    \theta_s}{ \nabla \loglihood (\theta_s) e^{\constrong s} } ds +
  \frac{1}{2} \int_0^t \vecnorm{ \theta_s - \thetastar}{2} \vecnorm{
    \nabla \loglihood (\theta_s) - \nabla \loglihood_{n} (\theta_s)
  }{2} e^{\constrong s} ds \nonumber \\ \leq & - \frac{1}{2} \int_0^t
  \strongconvex \vecnorm{\theta_s - \thetastar}{2}^2 e^{\constrong s}
  ds + \frac{1}{2} \int_0^t \vecnorm{ \theta_s - \thetastar}{2} (
  \noiseone(n, \delta) \vecnorm{ \theta_s - \thetastar}{2} +
  \noisetwo(n, \delta)) e^{\constrong s} ds \notag\\ \leq & -
  \frac{1}{2} \int_0^t \strongconvex \vecnorm{\theta_s -
    \thetastar}{2}^2 e^{\constrong s} ds + \frac{1}{2} \int_0^t
  \vecnorm{ \theta_s - \thetastar}{2}^2 ( \noiseone(n, \delta) +
  \strongconvex/3 ) e^{\constrong s} ds + \frac{3 \noisetwo^2(n,
    \delta)}{2 \strongconvex} \int_0^t e^{\constrong s} ds.
\end{align*}
The second term $J_{2}$ involving prior $\prior$ can be controlled in
the following way:
\begin{multline*}
    J_2  = \frac{1}{2 n} \int_0^t \inprod{ \theta_s -
      \thetastar}{\nabla \log \prior (\theta_s ) e^{\constrong s}} ds \leq \frac{1}{2n} \int_0^t \boundconsprior \vecnorm{\theta_s - \thetastar}{2}  e^{\constrong s} ds\\
      \leq \int_0^t  \frac{\strongconvex}{6} \vecnorm{\theta_s - \thetastar}{2}^2 e^{\constrong s} ds + \frac{3\boundconsprior^2}{n^2 \strongconvex} \int_0^t e^{\constrong s} ds.
\end{multline*}
For the third term $J_{3}$, a direct calculation leads to
\begin{align*}
    J_3 = \frac{d (e^{\constrong t} - 1)}{ 2 \constrong n }.
\end{align*}
Moving to the fourth term $J_{4}$, it is a martingale as $J_{4} =
M_{t}/ \sqrt{n}$.  Putting the above results together, as $\constrong =
\frac{1}{2} \strongconvex - \noiseone(n, \delta) >
\frac{\strongconvex}{6}$, we obtain that
\begin{align*}
    \frac{1}{2} e^{\constrong t} \vecnorm{\theta_t - \thetastar}{2}^2 
    \leq \frac{1}{\sqrt{n}} M_t + U_{n} \frac{(e^{\constrong t} - 1)}{ 2 \constrong }.
\end{align*}
Putting together the pieces yields the
claim~\eqref{eq:key_claim_strong_concave}.

\subsection{Proof of~\cref{thm:contraction-langevin-alg}}\label{subsec:proof-contraction-langevin}

Let the pair $(\Psi, \zeta_\numobs)$ to be the functions defined in
equations~\eqref{eq:func-pair-defn-in-local-weak-convex-proof} in the
proof of~\cref{thm:local-weak-convex}. We denote $\loglihood_\numobs
\mydefn \Psi + \zeta_n + \frac{1}{n}\log \prior$, and consider the
process generated by running Langevin algorithm on the modified
posterior distribution:
\begin{align}
  \thetatil_{k + 1} = \thetatil_k + \stepsize \nabla \left(\Psi + \zeta_n
 + \frac{1}{n}\log \prior \right) (\thetatil_k) + \sqrt{\frac{2 \stepsize}{\numobs}} W_k. \label{eq:modified-ula-in-algo-proof}
\end{align}
Note that the potential function $\widetilde{\loglihood}$ is exactly the same as $\loglihood_\numobs$ within the ball $\ball (\thetastar, \localradius)$. Defining the event:
\begin{align}
  \Event_k \mydefn \Big\{ \max_{1 \leq i \leq k} \vecnorm{\theta_i - \thetastar}{2} \leq \localradius \Big\}.\label{eq:good-event-in-langevin-alg}
\end{align}
On the event $\Event_k$, the process $(\thetatil_i)_{1 \leq i \leq k}$
has the same law as $(\theta_i)_{1 \leq i \leq k}$. In the following,
we analyze the moments of the process $(\thetatil_i)_{1 \leq i \leq
  k}$. As with the proof of~\cref{thm:local-weak-convex}, condition on
the random data $(X_i)_{i = 1}^\numobs$.

Defining $\Delta_k = \thetatil_k - \thetastar$, for any integer $p \geq 1$, a direct expansion of the iterates yields:
\begin{align*}
  &\Exs \Big[ \vecnorm{\Delta_{k + 1}}{2}^{2p} \Big]\\
   &\leq \sum_{q = 0}^{p} \binom{2p}{2 q} \Big( \sqrt{\frac{2 \stepsize}{\numobs}} \Big)^{2q} \Exs [\vecnorm{ W_k }{2}^{2q}] \cdot \Exs \Big[ \vecnorm{\Delta_k + \stepsize \nabla \widetilde{\loglihood}_\numobs (\thetatil_k)}{2}^{2p - 2q}\Big]\\
   &\leq  \sum_{q = 0}^{p} \binom{p}{q} \frac{(p + 1) \cdots (2p)}{(q + 1) \cdots (2q) \cdot (p - q + 1) \cdots (2p - 2q)} \Big( \sqrt{\frac{c \stepsize (d + q)}{\numobs}} \Big)^{2q} \cdot \Exs \Big[ \vecnorm{\Delta_k + \stepsize \nabla \widetilde{\loglihood}_\numobs (\thetatil_k)}{2}^{2p - 2q}\Big]\\
   &\leq  \sum_{q = 0}^{p} \binom{p}{q} \Big( \sqrt{\frac{c \stepsize p^2 (d + p)}{\numobs}} \Big)^{2q} \cdot \Big\{ \Exs \big[ \vecnorm{\Delta_k + \stepsize \nabla \widetilde{\loglihood}_\numobs (\thetatil_k)}{2}^{2p}\big] \Big\}^{\frac{p - q}{p}}\\
  &\leq \Big(\frac{c \stepsize p^2 (d + p)}{\numobs} + \Big\{ \Exs \big[ \vecnorm{\Delta_k + \stepsize \nabla \widetilde{\loglihood}_\numobs (\thetatil_k)}{2}^{2p}\big] \Big\}^{1 / p} \Big)^p.
\end{align*}
Using the shorthand notation $\lambda_{2p} \mydefn \Big\{ \Exs \Big[ \vecnorm{\Delta_k + \stepsize \nabla \widetilde{\loglihood}_\numobs (\thetatil_k)}{2}^{2p} \Big] \Big\}^{\frac{1}{2p}}$, we conclude that:
\begin{align}
   \Exs \Big[ \vecnorm{\Delta_{k + 1}}{2}^{2p} \Big]  \leq \Big( \lambda_{2p}^{2} + \frac{c \stepsize p^2 (d + p)}{\numobs} \Big)^p. \label{eq:recursive-moment-bound-in-langevin-alg-proof}
\end{align}

By the local growth
conditions~\ref{item:without_global_geometry},~\ref{item:without_global_deviation},
and the global smoothness
assumptions~\ref{item:smooth_population_condition}
and~\ref{item:smooth_log_prior}, we note that:
\begin{align*}
&\vecnorm{\Delta_k + \stepsize \nabla \widetilde{\loglihood}_\numobs
    (\thetatil_k)}{2}^{2}\\ &= \vecnorm{\Delta_k}{2}^2 + \stepsize
  \inprod{\thetatil_k - \thetastar}{\nabla \Psi (\thetatil_k)} +
  \stepsize^2 \vecnorm{\nabla \widetilde{\loglihood}_\numobs
    (\thetatil_k)}{2}^2\\ &\leq \vecnorm{\Delta_k}{2}^2 +
  \inprod{\Delta_k}{\Psi (\thetatil_k)} + \stepsize
  \vecnorm{\Delta_k}{2} \cdot \big( \vecnorm{\nabla \zeta_\numobs
    (\thetatil_k)}{2} + \numobs^{-1} \vecnorm{\nabla \log \prior
    (\thetatil_k)}{2} \big) + \stepsize^2 \smoothness^2
  \vecnorm{\Delta_k}{2}^2\\ &\leq \left(1 - \stepsize
  \lowerweakconcavemu + \stepsize^2 \smoothness^2 \right)
  \vecnorm{\Delta_k}{2}^2 + \stepsize \big( \varepsilon (\numobs,
  \delta) + \frac{\boundconsprior}{\numobs} \big)
  \vecnorm{\Delta_k}{2}\\ &\leq \left(1 - 2 \stepsize
  \lowerweakconcavemu / 3 + \stepsize^2 \smoothness^2 \right)
  \vecnorm{\Delta_k}{2}^2 + \frac{3 \stepsize}{\strongconvex} \Big(
  \varepsilon (\numobs, \delta) + \frac{\boundconsprior}{\numobs}
  \Big)^2.
\end{align*}
Given the stepsize $\stepsize <
\frac{\lowerweakconcavemu}{3\smoothness^2}$, we have that:
\begin{align*}
  \lambda_{2p}^{2p} &= \Exs \big[ \vecnorm{\Delta_k + \stepsize \nabla
      \widetilde{\loglihood}_\numobs (\thetatil_k)}{2}^{2p} \big]
  \\ &\leq \Exs \Big\{ (1 - \stepsize \lowerweakconcavemu / 3 )
  \vecnorm{\Delta_k}{2}^2 + \frac{3 \stepsize}{\strongconvex} \Big(
  \varepsilon (\numobs, \delta) + \frac{\boundconsprior}{\numobs}
  \Big)^2 \Big\}^{p} \\ &\leq \Big\{ ( 1 - \lowerweakconcavemu
  \stepsize / 3) \big( \Exs [\vecnorm{\Delta_k}{2}^{2p}]
  \big)^{\frac{1}{p}} + \frac{3 \stepsize}{\strongconvex} \Big(
  \varepsilon (\numobs, \delta) + \frac{\boundconsprior}{\numobs}
  \Big)^2 \Big\}^p.
\end{align*}
Combining with the bound~\eqref{eq:recursive-moment-bound-in-langevin-alg-proof}, we conclude that:
\begin{align*}
   \Big\{ \Exs \big[ \vecnorm{\Delta_{k + 1}}{2}^{2p} \big]
   \Big\}^{1/p} \leq ( 1 - \lowerweakconcavemu \stepsize / 3) \big(
   \Exs [\vecnorm{\Delta_k}{2}^{2p}] \big)^{\frac{1}{p}} + \frac{3
     \stepsize}{\strongconvex} \Big( \varepsilon (\numobs, \delta) +
   \frac{\boundconsprior}{\numobs} \Big)^2 + \frac{c \stepsize p^2 (d
     + p)}{\numobs}.
\end{align*}
Solving this recursion, we arrive at the following bound for $k =
0,1,2, \cdots$
\begin{align}
  \Big\{ \Exs \big[ \vecnorm{\Delta_{k}}{2}^{2p} \big] \Big\}^{1/p}
  \leq e^{- k \strongconvex \stepsize / 3} \vecnorm{\Delta_0}{2}^2 +
  \frac{9}{\strongconvex^2} \Big( \varepsilon (\numobs, \delta) +
  \frac{\boundconsprior}{\numobs} \Big)^2 + \frac{3 c p^2 (d +
    p)}{\strongconvex
    \numobs} \label{eq:moment-bound-for-modified-iterate-in-langevin-alg}
\end{align}
By
equation~\eqref{eq:moment-bound-for-modified-iterate-in-langevin-alg}
and a union bound over $k = 0,1,2,\cdots, T$, with probability $1 -
\vartheta$, we have that:
\begin{align*}
  \max_{0 \leq k \leq T} \vecnorm{\Delta_{k}}{2} \leq
  \vecnorm{\Delta_0}{2} + \frac{3 \varepsilon (\numobs,
    \delta)}{\strongconvex} + \frac{3 \boundconsprior}{\numobs
    \strongconvex} + \sqrt{\frac{3c d}{\numobs \strongconvex} \log^3
    \frac{T}{\vartheta}}.
\end{align*}
Under the condition $\vecnorm{\theta_0 - \thetastar}{2} \leq
\localradius / 2$ and the sample size
condition~\eqref{eq:sample-size-req-in-langevin-alg}, we have the
uniform bound:
\begin{align*}
  \Prob \big( \Event_T \big) \geq 1 - \delta / 2,
\end{align*}
Consequently, on the event $\Event_T$, we conclude the following
moment bound on the last iterate of the Langevin algorithm:
\begin{align*}
   \Big\{ \Exs \big[ \vecnorm{\Delta_{T}}{2}^{2p} \cdot
     \bm{1}_{\Event_T} \big] \Big\}^{1/p} \leq e^{- T \strongconvex
     \stepsize / 3} \vecnorm{\Delta_0}{2}^2 +
   \frac{9}{\strongconvex^2} \Big( \varepsilon (\numobs, \delta) +
   \frac{\boundconsprior}{\numobs} \Big)^2 + \frac{3 c p^2 (d +
     p)}{\strongconvex \numobs},
\end{align*}
which can be readily converted into the following bound with
probability $1 - \delta$:
\begin{align*}
  \vecnorm{\Delta_{T}}{2} \leq e^{- \frac{T \strongconvex
      \stepsize}{12 \log (1 / \delta)} } \vecnorm{\Delta_0}{2} + c
  \Big\{ \frac{ \varepsilon (\numobs, \delta)}{\strongconvex} +
  \frac{\boundconsprior}{\strongconvex \numobs} + \log (1 / \delta)
  \cdot \sqrt{\frac{\usedim + \log (1 / \delta)}{\strongconvex
      \numobs}} \Big\}.
\end{align*}

\subsection{Proof of~\cref{theorem-main-weakly-convex}}
\label{subsection:weakly_convex_proof}

As in the proof of~\cref{thm-one-point-SC}, we omit the conditioning
on $\mathcal{F}_n \mydefn \sigma(\DataX)$.  For any $p \geq 2$, we
define the functions on the positive real line $(0, \infty)$
\begin{align*}
\nu_{(p)}(r) \mydefn \psi \left( r^{ \frac{1}{p - 1} } \right)
r^{\frac{p - 2}{p - 1}}, \quad \mbox{and} \quad \tau_{(p)} \big(r^{p -
  1} \perturb(r) \big ) \mydefn r^{p - 2} \weakcon(r).
\end{align*}
By Assumption~\ref{item:weak_concavity_deviation}, the function $r
\mapsto r^{p - 1} \perturb (r)$ is strictly increasing and surjective
function that maps from $[0, +\infty)$ to $[0, +\infty)$. Therefore,
    it is invertible and the function $\tau_{(p)}^{-1}$ is
    well-defined.

Now we claim that for any $p \geq 2$, the functions $\nu_{(p)}$ and
$\tau_{(p)}$ are convex and strictly increasing, and that furthermore,
the expectation $\Exs \brackets{ \vecnorm{\theta_t -
    \thetastar}{2}^p}$ is upper bounded by the integral
\begin{align} 
 \frac{p}{2} \int_0^t \biggr( - R_p(s) + \noise(n, \delta)
 \tau_{(p)}^{- 1} (R_p(s)) & + \frac{\boundconsprior}{n} \nu_{(p)}^{
   -1} (R_p(s)) + \frac{p - 1 + d }{n} \nu_{(p)}^{ -1}
 (R_p(s))^{\frac{p - 2}{p - 1}} \biggr)
 ds, \label{eq:key_claim_weakly_concave}
\end{align}
where $R_p(s) \mydefn \Exs \brackets{ \vecnorm{\theta_s -
    \thetastar}{2}^{p - 2} \weakcon (\vecnorm{\theta_s -
    \thetastar}{2})}$.

Taking the above claims as given for the moment, let us now complete
the proof of the theorem. Since for each finite $q \geq 1$, the
process $(\theta_t: t\geq 0)$ converges in $\mathbb{L}^q$ norm, the
limit $\lim_{t \rightarrow +\infty} R_p(t)$ exists. Since the
functions $\tau_{(p)}$ and $\nu_{(p)}$ are convex and strictly
increasing, their inverse functions are concave.  Moreover, simple
calculation leads to
\begin{align}
    \nabla_{r}\parenth{ \nu_{(p)}^{-1} (r)^{\frac{ p - 2}{ p - 1}}} =
    \frac{p - 2}{p - 1} \cdot \frac{ \nu_{(p)}^{-1} (r)^{ - \frac{1}{p - 1}}
    }{\nu_{(p)}' (\nu_{(p)}^{-1} (r))}. \label{eq:deriv_equ}
\end{align}
Since $\nu_{(p)}$ is convex and increasing, the numerator is a
decreasing positive function of $r$. Additionally, the denominator is
an increasing positive function of $r$.  Therefore, the derivative in
equation~\eqref{eq:deriv_equ} is a decreasing function of $r$, and the
function $r \mapsto \nu_{(p)}^{-1} (r)^{\frac{ p - 2}{ p - 1}}$ is
concave. Define the function
\begin{align*}
    \phi(r) & \mydefn - r + \noise(n, \delta) \tau_{(p)}^{- 1} (r) +
    \frac{\boundconsprior}{\numobs} \nu_{(p)}^{ -1} (r) + \frac{p - 1
      + d}{n} \nu_{(p)}^{ -1} (r)^{\frac{p - 2}{p - 1}},
\end{align*}
and observe that $\phi$ is concave and $\phi(0) = 0$.  Let $r_*$ be
the smallest positive solution to the equation
\begin{align*}
    r = \noise(n, \delta) \tau_{(p)}^{- 1} (r) +
    \frac{\boundconsprior}{\numobs} \nu_{(p)}^{ -1} (r) + \frac{p - 1
      + d}{n} \nu_{(p)}^{ -1} (r)^{\frac{p - 2}{p - 1}}.
\end{align*}
We then have $\phi(r) < 0$ for $r > r_*$ and $\phi(r) > 0$ for $r \in
(0, r_*)$.  By~\cref{lemma-integral-ineq-limit-control}, we have
$\lim_{t \rightarrow +\infty} R_p(t) \leq r_*$.

Since $\nu_{(p)}$ is a convex and strictly increasing function,
Jensen's inequality implies that
\begin{align}
 R_p(t) = \mathbb{E} \left( \vecnorm{\theta_t - \thetastar}{2}^{p - 2}
 \psi(\vecnorm{\theta_t - \thetastar}{2}) \right) \geq \nu_{(p)}
 \left( \Exs \vecnorm{\theta_t - \thetastar}{2}^{p -
   1}\right). \label{eq:Rp-bound-by-power-p-minus-1}
\end{align}
Therefore, if we define $z_* \mydefn \lim_{t \rightarrow +\infty}
\left(\Exs \vecnorm{ \theta_t - \thetastar}{2}^{p - 1}
\right)^{\frac{1}{p - 1}}$, we have $z_*^{p - 1} \leq
\nu_{(p)}^{-1}(r_*)$.  Hence, we arrive at the following inequality
\begin{align*}
    z_*^{p - 2} \psi(z_*) & \leq \noise(n, \delta) \tau_{(p)}^{- 1}
    \left( \nu_{(p)} (z_*^{p - 1}) \right) + \frac{\boundconsprior}{n}
    z_*^{p - 1} + \frac{p - 1 + d }{n} z_*^{p - 2}\\ & = \noise(n,
    \delta) z_*^{p - 1} \zeta (z_*) + \frac{\boundconsprior}{n} z_*^{p
      - 1} + \frac{p - 1 + d }{n} z_*^{p - 2}.
\end{align*}
As a consequence, we find that
\begin{align*}
     \psi(z_*) \leq \noise(n, \delta) \zeta (z_*) z_* +
     \frac{\boundconsprior + (p - 1) d }{ n}.
\end{align*}
Now, we claim that there exists a unique positive solution to
equation~\eqref{eq:key_equation}. Given this claim, replacing $p$ by
$(p + 1)$ and putting the above results together yields
\begin{align*}
  \lim_{t \rightarrow +\infty} \left( \mathbb{E} \left ( \vecnorm{
    \theta_t - \thetastar }{2}^p\right) \right)^{ \frac{1}{p} } \leq
  z_p^*,
\end{align*}
where $z_p^*$ is the unique positive solution to the following
equation:
\begin{align*}
  \psi (z) = \noise(n, \delta) \zeta (z) z +
  \frac{\boundconsprior}{\numobs} z + \frac{p + d }{ n}.
\end{align*}
Combining the above inequality with the inequality~\eqref{eq:key_claim_weakly_concave}
yields the conclusion of the theorem. \\

\noindent We now return to prove our earlier claims about the behavior
of the functions $\nu_{(p)}$, $\tau_{(p)}$, the moment
bound~\eqref{eq:key_claim_weakly_concave}, and the existence of unique
positive solution to equation~\eqref{eq:key_equation}.

\subsubsection{Structure of the function $\nu_{(p)}$}

Since $\weakcon$ is a convex and strictly increasing function, by
taking the second derivative, we find that
\begin{align*}
\nu_{(p)}''(r) & = \nabla_{r}^2 \left( \weakcon \parenth{
  r^{\frac{1}{p - 1}} } r^{\frac{p - 2}{p - 1}} \right) \\ & =
\frac{1}{p - 1} r^{\frac{1}{p - 1} - 1} \weakcon'' \parenth{
  r^{\frac{1}{p - 1}} } + \frac{1}{p - 1} r^{-1} \left( \weakcon'
\parenth{ r^{\frac{1}{p - 1} } } - r^{ - \frac{1}{p - 1} } \weakcon
\parenth{ r^{\frac{1}{p - 1} } } \right) \geq 0
\end{align*}
for all $r > 0$. As a consequence, the function $\nu_{(p)}$ is convex.


\subsubsection{Structure of the function $\tau_{(p)}$}

This proof exploits
Assumption~\ref{item:weak_concavity_growth_conditions} on the
functions $\weakcon$ and $\perturb$.  For any $p \geq 2$, we denote
$\perturb_{(p)}: r \rightarrow r^{p - 1} \perturb (r)$ and
$\weakcon_{(p)}: r \rightarrow r^{p - 2} \weakcon (r) $ two strictly
increasing functions.  Therefore, we can define a function $\tau_{(p)}
\mydefn \weakcon_{(p)} \circ \perturb_{(p) }^{ -1 }$, namely, $
\tau_{(p)} (r^{p - 1} \perturb(r) ) = r^{p - 2} \weakcon(r)$, for any
$r > 0$.  Following some calculation, we find that
\begin{align*}
  \nabla_{r} \parenth{ \tau_{(p)} ( r^{p - 1} \perturb (r) )} & =
  \brackets{ ( p -1 ) r^{ p - 2 } \perturb (r) + r^{p - 1 }
    \perturb'(r) } \tau_{(p)}' (r^{p - 1} \perturb (r) ) \\ & = (p -
  2) r^{p - 3} \weakcon (r) + r^{p - 2} \weakcon' (r).
\end{align*}
Setting $z = \perturb_{(p)} (r)$ leads to
\begin{align*}
  \nabla_{z} \tau_{(p)} (z) = \frac{(p - 2) \weakcon (r) + r
    \weakcon'(r) } { (p - 1) r \perturb (r) + r^2 \perturb' (r) }.
\end{align*}
Taking another derivative of the above term, we find that
\begin{align*}
    \nabla_{z}^2 \tau_{(p)} (z) = & \left( \perturb_{(p)}'(r)
    \right)^{-1} \frac{ g(r, p)}{ \left( (p - 1) r \perturb (r) + r^2
      \perturb' (r) \right)^2 },
\end{align*}
where we denote
\begin{align*}
 g(r, p) \mydefn \brackets{ (p - 1) r \perturb(r) + r^2 \perturb'(r)
 } \cdot & \brackets{ (p - 1) \weakcon'(r) + r \weakcon''(r) } \\
& \hspace{- 4 em} - \brackets{ (p - 1) \perturb(r) + (p + 1) r
   \perturb'(r) + r^2 \perturb''(r) } \cdot \brackets{ (p - 2)
   \weakcon(r) + r \weakcon'(r)}.
\end{align*}
According to Assumption~\ref{item:weak_concavity_growth_conditions},
the function $\tau_{(2)} = \weakcon_{(2)} \circ \perturb_{(2)}^{-1}$
is convex. Therefore, we have $g(r, 2) \geq 0$ for any $r > 0$. Simple
algebra with first order derivative of function $g$ with respect to
parameter $p$ leads to
\begin{align*}
  \nabla_{p} \parenth{ g(r, p )} = & \perturb (r) \cdot \brackets{ (p
    - 1) r \weakcon'(r) + r^2 \weakcon''(r) - (p - 2)\weakcon(r) - r
    \weakcon' (r) } \\
- & r \perturb'(r) \brackets{ (p - 2) \weakcon(r) + r \weakcon'(r) } +
r \weakcon'(r) \cdot \brackets{ (p - 1) \perturb(r) + r \perturb'(r) }
\\
- & \weakcon(r) \cdot \brackets{ (p - 1) \perturb(r) + (p + 1) r
  \perturb' (r) + r^2 \perturb'' (r) }\\ = & 2 (p - 2) \brackets{ r
  \weakcon'(r) \perturb (r) - \weakcon (r) \perturb(r) - r
  \perturb'(r) \weakcon(r) } \\
+ & \brackets{ r^2 \perturb(r) \weakcon''(r) + r \weakcon'(r)
  \perturb(r) - 3 \weakcon(r) \perturb(r) - r^2 \weakcon(r)
  \perturb''(r) } \geq 0
\end{align*}
for all $r > 0$. Here the last inequality follows from
Assumption~\ref{item:weak_concavity_growth_conditions}.  Therefore,
the function $g$ is increasing function in terms of $p$ when $p \geq
2$, so that $g(r, p) \geq g(r, 2) \geq 0$ for all $r > 0$. Given this
inequality, we have $\frac{d^2}{dz^2} \tau_{(p)} (z) \geq 0$ for any
$z \geq 0$, $p \geq 2$, i.e., the function $\tau_{(p)}(z)$ is a convex
function for $z = \perturb_{(p)}(r)$.


\subsubsection{Proof of claim~\eqref{eq:key_claim_weakly_concave}}

For any $p \geq 2$, an application of It\^{o}'s formula yields the
bound $\vecnorm{ \theta_t - \thetastar }{2}^p \leq \sum_{j=1}^5
\Term_j$, where
\begin{subequations}
  \begin{align}    
\label{EqnTerm1}
  \Term_1 & \defn - \frac{p}{2} \int_0^t \inprod{\thetastar -
    \theta_s} {\nabla \loglihood(\theta_s)} \vecnorm{\theta_s -
    \thetastar}{2}^{p - 2} ds, \\
\label{EqnTerm2}
\Term_2 & \defn \frac{p}{2} \int_0^t \inprod{\thetastar - \theta_s}
     {\nabla \loglihood(\theta_s) - \nabla \loglihood_n (\theta_s)}
     \vecnorm{\theta_s - \thetastar}{2}^{p - 2} ds \\
\label{EqnTerm3}
\Term_3 & \defn \frac{p}{2n} \int_0^t \inprod{\theta_s - \thetastar}
     {\nabla \log \pi( \theta_s)} \vecnorm{\theta_s -
       \thetastar}{2}^{p - 2} ds \\
\label{EqnTerm4}
\Term_4 & \defn p \int_0^t \vecnorm{\theta_s - \thetastar}{2}^{p - 2}
\inprod{\theta_s - \thetastar}{d B_s} \\
\label{EqnTerm5}
\Term_5 & \defn \frac{p (p - 1 + d)}{2 n} \int_0^t \vecnorm{\theta_s -
  \thetastar}{2}^{p - 2} ds.
  \end{align}
\end{subequations}
We now upper bound the terms $\{T_j\}_{j=1}^5$ in terms of functionals
of the quantity $R_p$.  From the weak convexity of $\loglihood$
guaranteed by Assumption W.1, we have
\begin{subequations}
\begin{align}
\Exs \brackets{ T_1} & = - \frac{p}{2} \Exs \brackets{ \int_0^t
  \inprod{\thetastar - \theta_s} {\nabla \loglihood(\theta_s)}
  \vecnorm{\theta_s - \thetastar}{2}^{p - 2} ds} \leq - \frac{p}{2}
\int_0^t R_p (s) ds. \label{eq:T1_control}
\end{align}
Based on Assumption~\ref{item:weak_concavity_deviation}, we find that
\begin{align*}
    \Exs \brackets{ T_{2}} = \frac{p}{2} \Exs \brackets{ \int_0^t
      \inprod{\thetastar - \theta_s}{\nabla \loglihood(\theta_s) -
        \nabla \loglihood_n (\theta_s)} \vecnorm{\theta_s -
        \thetastar}{2}^{p - 2} ds} & \\
    & \hspace{- 9 em} \leq \frac{p}{2} \noise(n, \delta) \int_0^t
    \mathbb{E} \brackets{ \vecnorm{\theta_s - \thetastar}{2}^{p - 1}
      \perturb (\vecnorm{\theta_s - \thetastar}{2} )} ds.
\end{align*}
Since the function $\tau_{(p)}$ is convex, invoking Jensen's
inequality, we obtain the following inequalities:
\begin{align*}
  \int_0^t \mathbb{E} \brackets{ \vecnorm{\theta_s - \thetastar}{2}^{p
      - 1} \perturb \parenth{ \vecnorm{\theta_s - \thetastar}{2} }} ds
  & \leq \int_0^t \tau_{(p)}^{- 1} \mathbb{E} \brackets{ \tau_{(p)}
    \left( \vecnorm{\theta_s - \thetastar}{2}^{p - 1} \perturb
    (\vecnorm{\theta_s - \thetastar}{2}) \right)} ds \\
& = \int_0^t \tau_{(p)}^{ - 1} \parenth{ R_{p}(s)} ds.
\end{align*}
In light of the above inequalities, we have
\begin{align}
\Exs \brackets{ T_{2}} \leq \frac{p}{2} \noise(n, \delta) \int_0^t
\tau_{(p)}^{-1} \parenth{ R_{p}(s)} ds. \label{eq:T2_control}
\end{align}
Moving to $T_{3}$ in equation~\eqref{EqnTerm3}, given
Assumption~\ref{item:smooth_log_prior} which controls the growth of
prior distribution $\prior$, its expectation is bounded as
\begin{align}
  \Exs \brackets{ T_{3}} & = \frac{p}{2 n} \mathbb{E} \brackets{
    \int_0^t \inprod{\theta_s - \thetastar}{\nabla \log \prior(
      \theta_s)} \vecnorm{\theta_s - \thetastar}{2}^{p - 2} ds}
  \nonumber \\
\label{eq:bound_T3_first}  
& \leq \frac{p \boundconsprior }{2 n} \int_0^t \mathbb{E}
\brackets{\vecnorm{ \theta_s - \thetastar}{2}^{p - 1}} ds.
\end{align}

By exploiting the bound~\eqref{eq:Rp-bound-by-power-p-minus-1} along
with the fact that $\nu_{(p)}$ is strictly increasing on $[0, +
  \infty)$, we find that
\begin{align}
\label{eq:moment_control_second}    
\int_0^t \mathbb{E} \left(\vecnorm{\theta_s - \thetastar}{2}^{p - 1}
\right) ds \leq \int_0^t \nu_{(p)}^{-1} \left( R_{p}(s) \right) ds.
\end{align}
Combining the inequalities~\eqref{eq:bound_T3_first}
and~\eqref{eq:moment_control_second}, we have
\begin{align}
    \Exs \brackets{ T_{3}} \leq \frac{p \boundconsprior }{2 n}
    \int_0^t \nu_{(p)}^{-1} \left( R_{p}(s) \right)
    ds. \label{eq:bound_T3}
\end{align}
Moving to the fourth term $T_{4}$ from equation~\eqref{EqnTerm4}, we
have
\begin{align}
\label{eq:bound_T4}  
\mathbb{E} \brackets{ T_{4}} = \mathbb{E} \brackets{ \int_0^t
  \vecnorm{\theta_s - \thetastar}{2}^{p - 2} \inprod{\theta_s -
    \thetastar} {d B_s}} = 0,
\end{align}
where we have used the martingale structure.

For the last term $T_5$, invoking H\"{o}lder's inequality and the
bound~\eqref{eq:Rp-bound-by-power-p-minus-1}, we have the moment
estimate:
\begin{align*}
   \mathbb{E} \left(\vecnorm{\theta_s - \thetastar}{2}^{p - 2} \right)
   \leq \left( \Exs \left[\vecnorm{\theta_s - \thetastar}{2}^{p - 1}
     \right] \right)^{\frac{p - 2}{p - 1}} \leq \nu_{(p)}^{-1} \left(
   R_{p}(s) \right)^{\frac{p - 2}{p - 1}}.
\end{align*}
Consequently, the term $T_5$ can be bounded in expectation as
\begin{align}
\label{eq:bound_T5}  
\Exs \brackets{ T_{5}} \leq \frac{p (p - 1 + d)}{2 n} \int_0^t
\nu_{(p)}^{-1} \left( R_{p}(s) \right)^{ \frac{p - 2}{p - 1}} ds.
\end{align}
\end{subequations}
Collecting the bounds on the expectations of the terms
$\{\Term_j\}_{j=1}^5$ from
equations~\eqref{eq:T1_control}-\eqref{eq:bound_T5}, respectively, yields the
claim~\eqref{eq:key_claim_weakly_concave}.

\subsubsection{Unique positive solution to equation~\eqref{eq:key_equation}}

We now establish that equation~\eqref{eq:key_equation} has a unique
positive solution under the stated assumptions.  Define the function
\begin{align*}
\vartheta(z) & \mydefn \psi (z) - \left( \noise(n, \delta) \zeta(z) z
+ \frac{\boundconsprior + d \log(1/ \delta) }{ n} \right).
\end{align*}
Since $\psi(0) = 0$, we have $\vartheta (0) < 0$. On the other hand,
based on Assumption~\ref{item:weak_concavity_tail}, $\lim \inf_{z \to
  + \infty} \vartheta (z) > 0$. Therefore, there exists a positive
solution to the equation $\vartheta(z) = 0$.

Recall that $\inver: \real_+ \rightarrow \real$ is an inverse function
of the strictly increasing function $z \mapsto z \perturb
(z)$. Therefore, we can write the function $\vartheta$ as follows:
\begin{align*}
\vartheta (z) = \widetilde{\vartheta}r) & \mydefn \psi (\xi(r)) -
\noise(n, \delta) r - \frac{\boundconsprior + d \log(1/ \delta)}{n},
\end{align*}
where $r = z \cdot \zeta (z)$.  Given the convexity of function $r
\mapsto \weakcon( \inver (r))$ guaranteed by
Assumption~\ref{item:weak_concavity_growth_conditions}, the functions
$\widetilde{\vartheta}$ and $\vartheta$ are convex.  Putting the above
results together, there exists a unique positive solution to
equation~\eqref{eq:key_equation}.


\subsection{Proof of~\cref{thm-non-asymp-bvm}}  
\label{subsec:proof:thm-non-asymp-bvm}

We introduce the shorthand \mbox{$\targetdensity \mydefn \mathcal{N}
  \big(\thetamap, (n \HessianStar)^{-1} \big)$} for the target
density. Since $\HessianStar \succ 0$, the Gaussian log-Sobolev
inequality implies that
\begin{align}
\label{eq:log-sobolev-in-bvm-proof}  
\kull{\posterior (\cdot \mid X_1^n)}{\targetdensity} \leq
\frac{1}{n \lambda_{\min} (\HessianStar)} \int_{\real^d}
\vecnorm{\nabla \log \posterior (\theta \mid X_1^n) - \nabla \log
  \targetdensity (\theta)}{2}^2 ~\posterior (d \theta \mid X_1^n).
\end{align}
Since $\targetdensity$ is a Gaussian density, we find that
\begin{align*}
    \nabla \log \targetdensity (\theta) = - n \HessianStar (\theta -
    \thetamap).
\end{align*}
For the posterior density $\posterior (\cdot \mid X_1^n)$, we note
that
\begin{align*}
\nabla \log \posterior (\theta | X_1^n) & = - n \nabla F_n (\theta) +
\nabla \log \prior (\theta) \\ &= \int_0^1 \left( - n \nabla^2 F_n (
\gamma \theta + (1 - \gamma) \thetamap) + \nabla^2 \log \prior (\gamma
\theta + (1 - \gamma) \thetamap) \right) \\ & \hspace{23 em} \times
(\theta - \thetamap) d \gamma.
\end{align*}
Putting together the above equations together yields
\begin{multline*}
    \vecnorm{\nabla \log \posterior (\theta \mid X_1^n) - \nabla \log
      \targetdensity (\theta)}{2} \\
\leq n \int_0^1 \opnorm{\nabla^2 F_n ( \gamma \theta + (1 - \gamma)
  \thetamap) - \HessianStar + \nabla^2 \log \prior (\gamma \theta + (1
  - \gamma) \thetamap) / n} \cdot \vecnorm{\theta - \thetamap}{2} d
\gamma.
\end{multline*}
By Assumptions~\ref{item:Bvm_population},~\ref{item:Bvm_deviation},
and~\ref{item:smooth_population_condition}, we have the bounds
\begin{align*}
& \opnorm{\nabla^2 F_n ( \gamma \theta + (1 - \gamma) \thetamap) +
    \nabla^2 \log \prior (\gamma \theta + (1 - \gamma) \thetamap) / n
    - \HessianStar} \\
& \leq \opnorm{\nabla^2 F ( \gamma \theta + (1 - \gamma) \thetamap) -
    \HessianStar} \\
& \hspace{8 em} + \opnorm{\nabla^2 F_n ( \gamma \theta + (1 - \gamma)
    \thetamap) - \nabla^2 F_n ( \gamma \theta + (1 - \gamma)
    \thetamap)} + \frac{\smoothprior}{n} \\
& \leq \holderconst \vecnorm{\gamma \theta + (1 - \gamma) \thetamap -
    \thetastar}{2} + \noiseone^{(2)} (n, \delta )
  \vecnorm{\theta - \thetamap}{2} + \noisetwo^{(2)} (n, \delta) +
  \frac{\smoothprior}{n}.
\end{align*}
Substituting this bound into the
bound~\eqref{eq:log-sobolev-in-bvm-proof} yields
\begin{align*}
    \kull{\posterior (\cdot \mid X_1^n)}{\targetdensity} & \leq
    \frac{n}{\lambda_{\min} (\HessianStar)} \left(\holderconst \cdot
    \Exs_\posterior \left[ \vecnorm{\theta - \thetastar}{2}^4 \mid X_1^n \right] + \holderconst
    \vecnorm{\thetamap - \thetastar}{2}^4 \right)
    \\ & + \frac{n \noiseone^{(2)} (n, \delta )}{\lambda_{\min}
      (\HessianStar)} \Exs_\posterior \left[ \vecnorm{\theta -
        \thetamap}{2}^3 \mid X_1^n \right] \\ & + \big(
    \noisetwo^{(2)} (n, \delta) + \smoothprior / n \big) \cdot \Exs
    \left[ \vecnorm{\theta - \thetamap}{2}^2 \mid X_1^n \right] .
\end{align*}
As a consequence, we obtain the conclusion of the proposition.

\section{Proofs of corollaries}
\label{sec:append_corollary_prof}
In this appendix, we collect the proofs of several corollaries stated
in the main text and~\cref{sec:examples}.  To summarize, we make use
of
Theorems~\ref{theorem-main-weakly-convex},~\ref{thm:local-weak-convex},
and~\ref{thm:non-asymp-credible-set} to establish the posterior
contraction rates of parameters and non-asymptotic Bernstein-von Mises
theorem in the examples in~\cref{sec:examples}. The crux of the proofs
of these corollaries involves a verification of assumptions to invoke
the respective theorems.  Note that the values of universal constants
may change from line-to-line.


\subsection{Proof of~\cref{cor:logit_regres}}
\label{subsec:cor:logit_regres}

We begin by verifying claim~\eqref{eq:weak_conv_logit} about the
structure of the negative population log-likelihood function
$\loglihoodlogit$ and claim~\eqref{eq:empi_process_logit} about the
uniform perturbation error between $\nabla \loglihoodlogit$ and
$\nabla \loglihoodlogit_{n}$.


\subsubsection{Proof of claim~\eqref{eq:weak_conv_logit}}

Following some algebra, we find that
\begin{align*}
- \loglihoodlogit (\theta) & = \Exs \brackets{ - Y \log \parenth{ 1 +
    e^{- \inprod{X}{\theta} }} - (1 - Y)\log \parenth{1 +
    e^{\inprod{X}{\theta} }} }\\ & \hspace{- 2 em} = - \Exs \brackets{ \frac{1}{ 1 +
    e^{-\inprod{X} {\thetastar} }} \log \parenth{ 1 + e^{-
      \inprod{X}{\theta} }} + \frac{1}{ 1 + e^{ \inprod{X}{\thetastar}
  }} \log \parenth{ 1 + e^{\inprod{X}{\theta} }} },
\end{align*}
where the above expectations are taken with respect to $X \sim
\NORMAL(0, \sigma^2 I_{d})$ and $Y| X$ following probability
distribution generated from logistic
model~\eqref{eq:Bayes_logistic_regress}. Taking the derivative of
$\loglihoodlogit$ with respect to $\theta$ yields
\begin{align*}
    \inprod{ \nabla \loglihoodlogit (\theta)}{ \thetastar - \theta} & \\
    & \hspace{- 3 em} = \Exs \brackets{ \left( \frac{1 + e^{ \inprod{X}{\theta} } }
    {1 + e^{ \inprod{X}{\thetastar} } } - \frac{1 + e^{ - 
    \inprod{X}{\theta} } }{1 + e^{ - \inprod{X}{\thetastar} } } 
    \right)\frac{ e^{-\inprod{X}{\theta}} }{( 1 + e^{- 
    \inprod{X}{\theta} })^2 }\inprod{X}{\theta - \thetastar} }.
\end{align*}
By the mean value theorem, there exists $\xi$ between $0$ and
$\inprod{X}{\theta - \thetastar}$ such that
\begin{align*}
     \frac{1 + e^{ \inprod{X}{\theta} } }{1 + e^{ \inprod{X}
         {\thetastar} } } - \frac{1 + e^{ - \inprod{X}{\theta} }}{1 +
       e^{ - \inprod{X}{\thetastar} } } = \inprod{X}{\theta -
       \thetastar} \parenth{ \frac{ e^{\inprod{X}{\thetastar} + \xi} }
       {1 + e^{\inprod{X}{\thetastar} } } + \frac{ e^{- \inprod{X}
           {\thetastar} - \xi} }{1 + e^{- \inprod{X}{\thetastar} } }
     }.
\end{align*}
In light of the above equality, we arrive at the following inequalities:
\begin{align*}
    \inprod{ \nabla \loglihoodlogit (\theta)}{ \thetastar - \theta} &
    \geq \Exs \biggr[ \inf_{ \abss{ \xi} \in [0, \abss{
            \inprod{X}{\theta - \thetastar}} ] } \left( \frac{
        e^{\inprod{X}{\thetastar} + \xi} }{1 + e^{\inprod{X}
          {\thetastar} } } + \frac{ e^{- \inprod{X}{\thetastar} - \xi}
      } {1 + e^{- \inprod{X}{\thetastar} } } \right) \\ & \hspace{10
        em} \times \frac{ e^{-\inprod{X}{\theta}} }{( 1 + e^{-
          \inprod{X}{\theta} })^2 } |\inprod{X}{\theta -
        \thetastar}|^2 \biggr] \\ & \geq \Exs \brackets{
      \frac{1}{2}e^{- |\inprod{X}{\theta - \thetastar}|} \frac{
        e^{-\inprod{X}{\theta}} }{( 1 + e^{- \inprod{X}{\theta} })^2 }
      |\inprod{X}{\theta - \thetastar}|^2 } \\ & \geq \frac{1}{8} \Exs
    \brackets{ e^{- |\inprod{X}{\theta - \thetastar}| - |\inprod{X}{
          \theta}|} |\inprod{X}{\theta - \thetastar}|^2 }\\ & \geq
    \frac{1}{8 e^4} \Exs \brackets{ \bm{1}_{\{ \abss{
          \inprod{X}{\theta}} \leq 2, \ \abss{ \inprod{X}{\theta -
            \thetastar}} \leq 2 \} }|\inprod{X}{\theta -
        \thetastar}|^2 }.
\end{align*}
Since $X \sim \NORMAL(0, I_{d})$, we have 
\begin{align*}
\left[ \begin{matrix}
    \inprod{X}{\theta}\\ \inprod{X}{\theta - \thetastar}
\end{matrix} \right] \sim \mathcal{N}\left(0, \left[ 
\begin{matrix} \vecnorm{\theta}{2}^2& \inprod{\theta}{\theta - 
\thetastar}\\ \inprod{\theta}{\theta - \thetastar} & 
\vecnorm{\theta - \thetastar}{2}^2 \end{matrix} \right]\right).
\end{align*} 
Given that result, direct calculation leads to
\begin{align*}
    \mathbb{E}\left(\bm{1}_{\{ |\inprod{X}{\theta}| \leq 2, |
    \inprod{X}{\theta - \thetastar}| \leq 2\} }|\inprod{X}
    {\theta - \thetastar}|^2 \right) & \\
    & \hspace{- 3 em} \geq \frac{c}{(1 + 
    \vecnorm{\theta}{2} )(1 + \vecnorm{\theta - \thetastar}{2} )} 
    \vecnorm{\theta - \thetastar}{2}^2,
\end{align*}
for a universal constant $c > 0$. Collecting the above results, 
for all $\theta$ such that $\vecnorm{\theta - \thetastar}{2} 
\leq 1$, we achieve that
\begin{align*}
    \inprod{ \nabla \loglihoodlogit (\theta)}{ \thetastar - \theta} &
    \geq \frac{c}{(1 + \vecnorm{\theta}{2} )(1 + \vecnorm{\theta -
        \thetastar}{2} )} \vecnorm{\theta - \thetastar}{2}^2 \\ & \geq
    \unicon \frac{1}{ 1 + \vecnorm{\thetastar}{2}} \enorm{ \theta -
      \thetastar}^{2}.
\end{align*}
\noindent
For $\theta$ with $\vecnorm{ \theta - \thetastar}{2} > 1$, let $
\widetilde{\theta} = \thetastar + \frac{ \theta - \thetastar} {\vecnorm{
    \theta - \thetastar}{2}}$. Then, we find that
\begin{align*}
    \inprod{\nabla \loglihoodlogit (\theta)}{ \thetastar - 
    \theta} \geq  \inprod{\nabla \loglihoodlogit 
    (\widetilde{\theta})}{ \thetastar - \theta} \geq \frac{c}{2 (1 + 
    \vecnorm{\thetastar}{2})} \vecnorm{\theta - \thetastar}{2},
\end{align*}
which yields the claim~\eqref{eq:weak_conv_logit}.


\subsubsection{Proof of the bound~\eqref{eq:empi_process_logit}}
\label{subsec:uniform_perturb_logit}

In this appendix, we prove the uniform
bound~\eqref{eq:empi_process_logit} between the empirical and
population likelihood gradients.  It suffices to establish the
following stronger result:
\begin{align}
\label{eq:strong_empi_process_logit}
Z \; \mydefn \; \sup_{\theta \in \real^d} \vecnorm{\nabla
  \loglihoodlogit_{n} (\theta) - \nabla \loglihoodlogit ( \theta) }{2}
& \leq \unicon \; \left \{ \sqrt{\frac{\usedim}{\numobs}} +
\sqrt{\frac{\log(1/\delta)}{\numobs}} + \frac{\log(1/\delta)}{\numobs}
\right \},
\end{align}
with probability at least $1-\delta$ for any $\frac{\numobs}{\log
  \numobs} \geq c_0 \usedim \log(1/ \delta)$ where $c_0$ is a
universal constant.

In order to prove the
  claim~\eqref{eq:strong_empi_process_logit}, we exploit a
  concentration inequality due to Adamczak~\cite{adamczak2008tail}; it
  gives tight tail bounds for supremum of unbounded empirical
  processes. Throughout our derivation, we use
  $\vecnorm{X}{\psi_\alpha}$ to denote the Orlicz $\psi_\alpha$ norm
  for a random variable $X$, for any $\alpha \in (0, 2]$.  Let us
state a simplified version of a theorem due to Adamczak:
\begin{proposition}[Theorem 4 of~\cite{adamczak2008tail}, simplified version]
  \label{prop:adamczak}
  Let $(x, \theta) \mapsto f (\theta; x)$ be a function with domain
  $\Theta \times \mathcal{X}$, and suppose that there is a function
  $\bar{F} : \mathcal{X} \rightarrow \real$ such that $|f (\theta, x)|
  \leq \bar{F} (x)$ for any $\theta \in \Theta$. Let $X_1, X_2,
  \cdots, X_\numobs \stackrel{\mathrm{i.i.d.}}{\sim} \Prob_X$, and
  suppose that $\vecnorm{\bar{F}}{\psi_\alpha} < + \infty$ for some
  $\alpha \leq 1$. Then the random variable $Z_\numobs \mydefn
  \frac{1}{\numobs}\sup_{\theta \in \Theta} \abss{\sum_{i = 1}^\numobs
    f (\theta; X_i) - \Exs [f (\theta; X)] }$ satisfies the bound:
    \begin{align*}
        \Prob \left( Z_\numobs > 2 \Exs [Z_\numobs] + t \right) \leq
        \exp \left(- \frac{t^2}{2 \Exs [\bar{F} (X)^2]} \right) + 3
        \exp \left( - \left( \frac{t}{c \vecnorm{\max_{i \in
              [\numobs]}\bar{F} (X_i)}{\psi_{\alpha}}}
        \right)^{\alpha} \right),
    \end{align*}
    for a universal constant $c > 0$.
\end{proposition}

In order to prove the claim~\eqref{eq:strong_empi_process_logit}, we
begin by writing $Z$ as the supremum of a stochastic process. Let
$\sphere^{\usedim-1}$ denote the Euclidean sphere in $\real^\usedim$,
and define the stochastic process
\begin{align*}
 Z_{u, \theta} & \mydefn \left| \frac{1}{\numobs} \sum_{i=1}^\numobs
 f_{u, \theta}(X_i, Y_i) - \Exs[f_{u,\theta}(X, Y)] \right|,
\end{align*}
where $f_{u,\theta}(x,y) = \dfrac{y \inprod{x}{u} e^{y
    \inprod{x}{\theta}}}{1 + e^{y \inprod{x}{\theta}}}$, indexed by
vectors $u \in \sphere^{\usedim-1}$ and $\theta \in \RBALL$.  The
outer expectation in the above display is taken with respect to $(X,
Y)$ drawn from the logistic model~\eqref{eq:Bayes_logistic_regress}

Observe that $Z = \sup
\limits_{u \in \sphere^{\usedim-1}} \sup \limits_{\theta \in \real^d}
Z_{u, \theta}$. Let $\{u^1, \ldots, u^N\}$ be a $1/8$-covering of
$\sphere^{\usedim-1}$ in the Euclidean norm; there exists such a set
with $N \leq 17^\usedim$ elements.  By a standard discretization
argument (see Chapter 6, ~\cite{Wainwright_nonasymptotic}), we have
\begin{align*}
Z \leq 2 \max_{j=1, \ldots, N} \sup_{\theta \in \real^d} Z_{u^j,
  \theta}.
\end{align*}
Accordingly, the remainder of our argument focuses on bounding the
random variable \mbox{$V \mydefn \sup_{\theta \in \real^d} Z_{u, \theta}$,}
where the vector $u \in \sphere^{\usedim-1}$ should be understood as
arbitrary but fixed. For each $u \in \sphere^{d - 1}$ fixed, we note that $\bar{F} (X, Y) = |\inprod{X}{u}|$ is an envelop function for the class $(f_{u, \theta} (X, Y))_{\theta \in \real^d}$. Additionally, by standard tail bounds for maximum of Gaussian random variables, we know that:
\begin{align*}
    \vecnorm{\max_{1 \leq i \leq \numobs} \bar{F} (X_i, Y_i)}{\psi_1} \leq \sqrt{\log \numobs}.
\end{align*}
Consequently, invoking~\cref{prop:adamczak} yields that
\begin{align}
    V \leq 2 \Exs[V] + \sqrt{\frac{2 \log(1/\delta)}{\numobs}} +
    \frac{c \log(1/\delta)}{\numobs} \sqrt{\log
      \numobs}\label{EqnNhatPunk} 
\end{align}
with probability at least $1 - \delta$. 

Now define the symmetrized random variable
\begin{align*}
V' & \defn \sup_{ \theta \in \real^d} \left| \frac{1}{\numobs}
\sum_{i=1}^\numobs \rade_i f_{\theta, u} (X_i, Y_i) \right|.
\end{align*}
where $\{\rade_i\}_{i=1}^\numobs$ is an i.i.d. sequence of Rademacher
variables.  By standard symmetrization arguments, we have
\begin{align*}
  \Exs \left[V\right] & \leq 2\Exs \left[ V' \right].
\end{align*}

We now bound the expectation of $V'$, first over the Rademacher
variables. Consider the function class
\begin{align*}
\mathcal{G} \mydefn \left\{ g_{\theta}: (x, y) \mapsto \inprod{x}{u}
\varphi_\theta (x, y) \; \mid \; \theta \in \real^d \right\}.
\end{align*} 
It is clear that the function class $\mathcal{G}$ has the envelope
function $\bar{G}(x) \mydefn \abss{\inprod{x}{u}}$.  We claim that the
$L_2$-covering number of $\mathcal{G}$ can be bounded as
\begin{align}
\label{eq:Wenlong_diffusion_wizard}
\bar{N}(t) \mydefn \sup_{Q} \abss{\mathcal{N} \left( \mathcal{G},
  \vecnorm{\cdot}{L^2 (Q)}, t \vecnorm{\bar{G}}{L^2 (Q)} \right)} \leq
\left( \frac{1}{t} \right)^{c(d + 1)} \qquad \mbox{for all $t > 0$,}
\end{align}
where $c > 0$ is a universal constant.

Let us take the claim~\eqref{eq:Wenlong_diffusion_wizard} as given for
the moment, and use it to bound the expectation of $V'$, first over
the Rademacher variables.  Define the empirical expectation
\mbox{$\mprob_{n}(\bar{G}^2) \defn \frac{1}{n} \sum_{i = 1}^{n}
  \inprod{X_{i}}{u}^2$.}  Invoking Dudley's entropy integral bound
(e.g., Theorem 5.22, ~\cite{Wainwright_nonasymptotic}), we find that
there are universal constants $C, C'$ such that
\begin{align*}
  \Exs_\rade[V'] = \Exs_{\rade} \left[ \sup_{ g \in \mathcal{G} }
    \left| \frac{1}{\numobs} \sum_{i=1}^\numobs \rade_i g (X_i, Y_i)
    \right| \right] & \leq C \sqrt{\frac{\mprob_n(\bar{G}^2)}{n} }
  \int_0^1 \sqrt{1 + \log \bar{N} (t)} d t \\ & \leq C' \sqrt{\mprob_n
    (\bar{G}^2) } \sqrt{\frac{d}{n}}.
\end{align*}
Up to this point, we have been conditioning on the observations
$\{X_{i}\}_{i = 1}^{n}$.  Taking expectations over them as well yields
\begin{align}
\label{eq:Martin_the_Grumpy_Dinosaur}
    \Exs_{\varepsilon, X_1^n} [V'] \leq C' \sqrt{\frac{d}{n}} \cdot
    \Exs_{X_{1}^{n}} \brackets{ \sqrt{\mprob_n(\bar{G}^2)}}
    \stackrel{(i)}{\leq} C' \sqrt{\frac{d}{n}} \cdot
    \sqrt{\Exs_{X_{1}^{n}} \brackets{ \mprob_n(\bar{G}^2)}}
    \stackrel{(ii)}{=} C' \sqrt{\frac{d}{n}},
\end{align}
where step (i) follows from Jensen's inequality; and step (ii) uses
the fact that $\Exs_{X_1^n}[\mprob_n(\bar{G}^2)] = 1$. Putting
together the bounds~\eqref{EqnNhatPunk}
and~\eqref{eq:Martin_the_Grumpy_Dinosaur} yields the following bound
with probability $1 - \delta$:
\begin{align*}
    V \leq c \sqrt{\frac{d + \log \delta^{-1}}{\numobs}} + c
    \frac{\log \delta^{-1}}{\numobs} \sqrt{\log \numobs}.
\end{align*}
This probability bound holds for each $u \in \sphere^{\usedim-1}$.  By
taking the union bound over the $1/8$-covering set $\{u^{1}, \ldots,
u^{N}\}$ of $\sphere^{\usedim-1}$ where $N \leq 17^{d}$ and applying
above bound with $\delta' = \delta / N$, we obtain the
claim~\eqref{eq:strong_empi_process_logit} for sample size satisfying
$\frac{\numobs}{\log \numobs} \geq c d \log(1/\delta)$.


\subsubsection{Proof of claim~\eqref{eq:Wenlong_diffusion_wizard}}

We consider a fixed sequence $(x_i, y_i, t_i)_{i = 1}^{m}$ where
$y_{i} \in \{- 1, 1 \}$, $x_{i} \in \Rspace^{d}$ and $t_{i}
\in \Rspace$ for $i \in [m]$.  Now, we suppose that for any binary
sequence $(z_i)_{i = 1}^{m} \in \{0, 1\}^{m}$, there exists $\theta
\in \real^d$ such that
\begin{align*}
    z_i = \Ind \left[ \inprod{X_i}{u} \varphi_\theta (X_i, Y_i) \geq
      t_i \right] \qquad \text{for all} \ i \in [m].
\end{align*}
Following some algebra, we find that
\begin{align*}
 y_i x_i^T\theta - \log \frac{Y_i t_i}{\inprod{X_i}{u} - Y_i
   t_i} \begin{cases} \geq 0& z_i = 1\\ < 0 & z_i = 0
 \end{cases}.
\end{align*}
Consequently, the set $\{ [y_i x_i , \log ({Y_i t_i} /
  (\inprod{X_i}{u} - Y_i t_i) )] \}_{i = 1}^m$ of $(d +
1)$-dimensional points can be shattered by linear separators.
Therefore, we have $m \leq d + 2$, which leads to the VC subgraph
dimension of $\mathcal{G}$ to be at most $d + 2$ (e.g., see the
book~\cite{Vaart_Wellner_2000}).  As a consequence, we obtain the
conclusion of the claim~\eqref{eq:Wenlong_diffusion_wizard}.
\subsection{Proof of~\cref{cor:single_index}}
\label{subsec:proof:cor:single_index}

The claim~\eqref{eq:weak_conv_index} of weak convexity for the
negative population log-likelihood function $\loglihoodind$ is
straightforward. Therefore, we only need to establish the
claim~\eqref{eq:empi_process_index} about the uniform perturbation
bound between $\nabla \loglihoodind$ and $\nabla \loglihoodind_{n}$.


\subsubsection{Bounding the difference $\nabla \loglihoodind - \nabla \loglihoodind_{n}$}

It is convenient to introduce the shorthand
\begin{align*}
  p_{ \theta} (x, y) = \parenth{ y - \parenth{ x^{\top}
      \theta}^{p}}^{2}/ 2 \qquad \mbox{for all $(x, y) \in \Rspace^{d
      + 1}$.}
\end{align*}
We then compute the gradient
\begin{align*}
  \nabla \log p_\theta (x, y) 
  = p \parenth{ y - \parenth{ x^{ \top} \theta}^p } 
  \parenth{x^{\top} \theta}^{p - 1} x.
\end{align*}
Fix an arbitrary $r > 0$, by applying the triangle inequality, we find
that
\begin{align*}
  \sup_{\theta \in \ball( \thetastar, r)} \enorm{\nabla
    \loglihoodind_{n} (\theta) - \nabla \loglihoodind ( \theta) } & \\
     & \hspace{- 7 em} =
  \sup_{\theta \in \ball( \thetastar, r)} \enorm{\frac{1}{n} \sum_{i =
      1}^{n} \nabla \log p_ \theta (X_{i}, Y_{i}) - \Exs_{(X, Y)}
    \brackets{\nabla \log p_\theta (X, Y)}} \\
  & \hspace{- 7 em} \leq p  \left \{ J_1 + J_2 \right \},
\end{align*}
where we define
\begin{subequations}
\label{eq:empi_process_bound_index}  
  \begin{align}
J_1 & \mydefn p \sup_{\theta \in \ball( \thetastar, r)}
\enorm{\frac{1}{n} \sum_{i = 1}^{n} Y_{i} X_{i} \parenth{X_{i}^{\top}
    \theta}^{p - 1}}, \quad \mbox{and} \\
J_2 & \mydefn p \sup_{\theta \in \ball( \thetastar, r)}
\enorm{\frac{1}{n} \sum_{i = 1}^{n} X_{i} \parenth{X_{i}^{\top}
    \theta}^{2 p - 1} - \Exs_{X} \brackets{X \parenth{ X^{\top}
      \theta}^{2 p - 1}}}.
  \end{align}
\end{subequations}
We claim that there is a universal constant $\unicon$ such that for
any $\delta \in (0,1)$, the quantities $J_{1}$ and $J_{2}$ can be
bounded as
\begin{subequations}
\begin{align}
 \label{eq:empi_process_J1_index}   
  J_{1} & \leq \unicon \; r^{p - 1} \left( \sqrt{\frac{d + \log
      \frac{1}{\delta}}{n} } + \frac{1}{n^{3/2}} \left( d + \log
  \frac{n}{\delta}\right)^{p + 1} \right), \quad \mbox{and} \\
\label{eq:empi_process_J2_index}  
  J_{2} & \leq \unicon \; r^{2 p - 1} \left( \sqrt{\frac{d + \log
      \frac{1}{\delta}}{n} } + \frac{1}{n^{3/2}} \left( d + \log
  \frac{n}{\delta} \right)^{2 p + 1} \right),
\end{align}
\end{subequations}
with probability at least $1 - \delta$.

Assume that the above claims are given at the moment.  We proceed to
finish the proof of the uniform perturbation bound between $\nabla
\loglihoodind_{n}$ and $\nabla \loglihoodind$
in~\eqref{eq:empi_process_index}.  In fact, plugging the concentration
bounds~\eqref{eq:empi_process_J1_index}
and~\eqref{eq:empi_process_J2_index}
into~\eqref{eq:empi_process_bound_index}, we obtain that
\begin{align*}
  \sup_{\theta \in \ball( \thetastar, r)} \enorm{\nabla \loglihoodind_{n} (\theta) 
      - \nabla \loglihoodind ( \theta) } & \\
      & \hspace{- 10 em} \leq c \parenth{r^{p - 1} + r^{2 p - 1}} \sqrt{ \frac{d + \log(1 / \delta)}{n}}\\
      & \hspace{- 10 em} + \frac{r^{p - 1} (d + \log (1 / \delta) + \log n)^{p + 1} 
      + r^{2p - 1} (d + \log (1 / \delta) + \log n)^{2p + 1} }{n^{\frac{3}{2}}},
\end{align*}
for any $r > 0$ with probability at least $1 - 2 \delta$ where $c$ is
a universal constant.  When $n \geq c' \parenth{ d + \log (d / \delta)
}^{2 p}$ for some universal constant $c'$, it is clear that the the
second term is dominated by the first term in the RHS of the above
inequality. As a consequence, we have proved the
claim~\eqref{eq:empi_process_index}.


\subsubsection{Proof of claim~\eqref{eq:empi_process_J1_index}}

Following some algebra, we find that
\begin{align} 
  \sup_{r > 0} \frac{\sup \limits_{\theta \in \ball( \thetastar, r)}
    \enorm{\frac{1}{n} \sum_{i = 1}^{n} Y_{i} X_{i}
      \parenth{X_{i}^{\top} \theta}^{p - 1}}}{r^{p - 1}} & \nonumber \\
      & \hspace{- 6 em} \leq
  \sup_{r > 0} \sup_{\theta \in \ball( \thetastar, r)}
  \enorm{\frac{1}{n} \sum_{i = 1}^{n} Y_{i} X_{i}
    \parenth{X_{i}^{\top} \frac{\theta}{\enorm{ \theta}}}^{p - 1}}
  \nonumber \\
  & \hspace{- 6 em} = \underbrace{\sup_{\theta \in \sphere^{d - 1}} \enorm{\frac{1}{n}
      \sum_{i = 1}^{n} Y_{i} X_{i} \parenth{X_{i}^{\top} \theta}^{p -
        1}}}_{= : \: Z}.
  \label{eq:normalize_J1_index}
\end{align}
Thus, in order to establish the
claim~\eqref{eq:empi_process_J1_index}, it suffices to show that there
is a universal constant $\unicon$ such that
\begin{align}
\label{eq:surogate_J1_index}
\Prob \parenth{ Z \leq \unicon \sqrt{\frac{d + \log(1/ \delta)}{n}} +
  \frac{1}{n^{3/2}} \left( d + \log \frac{n}{\delta} \right)^{p + 1} }
\geq 1 - \delta.
\end{align}
By the variational definition of the Euclidean norm, we have
\begin{align*}
  Z & = \sup_{\theta \in \sphere^{d - 1}} \enorm{\frac{1}{n} \sum_{i =
      1}^{n} Y_{i} X_{i} \parenth{X_{i}^{\top} \theta}^{p - 1}} \\
      & =
  \sup_{u \in \sphere^{d - 1}} \underbrace{\sup_{\theta \in \sphere^{d
        - 1}} \abss{\frac{1}{n} \sum_{i = 1}^{n} Y_{i} X_{i}^{ \top} u
      \parenth{X_{i}^{\top} \theta}^{p - 1}}}_{\mydefn Z_{u}}.
\end{align*}
Using a discretization argument as
in~\cref{subsec:uniform_perturb_logit}, we find that
\begin{align*}
  Z \leq 2 \sup_{u \in \mathcal{N} 
  \parenth{ \frac{1}{8}, \sphere^{d - 1}, \|.\|_{2}}} Z_{u},
\end{align*}
where $\mathcal{N} \parenth{ \frac{1}{8}, \sphere^{d - 1}, \|.\|_{2}}$ is
the $\frac{1}{8}$-covering of $\sphere^{d - 1}$ under $\|.\|_{2}$ norm. 
Therefore, it is sufficient to bound $Z_{u}$ for any fixed 
$u \in \mathcal{N} \parenth{ \frac{1}{8}, \sphere^{d - 1}, \|.\|_{2}}$.

For any even integer $q \geq 2$, a symmetrization argument (e.g.,
Theorem 4.10, ~\cite{Wainwright_nonasymptotic}) yields
\begin{align*}
  \Exs \parenth{\sup_{\theta \in \sphere^{d - 1}} \abss{\frac{1}{n} 
  \sum_{i = 1}^{n} Y_{i} X_{i}^{ \top} u \parenth{X_{i}^{\top} \theta}^{p - 1}} }^{q} & \\ 
  & \hspace{- 3 em} \leq \Exs \parenth{\sup_{\theta \in \sphere^{d - 1}} \abss{\frac{2}{n} \sum_{i = 1}^{n} 
  \rade_i Y_{i} X_{i}^{ \top} u \parenth{X_{i}^{\top} \theta}^{p - 1}} }^q ,
\end{align*}
where $\{\rade_i\}_{i=1}^\numobs$ is an i.i.d. sequence of Rademacher
variables.  In order to facilitate the proof argument, for any $t >
0$, we introduce the shorthand $\mathcal{N}(t) \mydefn \mathcal{N}
\parenth{t, \sphere^{d - 1}, \|.\|_{2}} = \{ \theta_{1}, \ldots,
\theta_{\bar{N}(t)}\}$ where $\bar{N}(t) = \abss{\mathcal{N}
  \parenth{t, \sphere^{d - 1}, \|.\|_{2}}}$.  For any compact set
$\Omega \subseteq \real^d$, we define the following random variable:
\begin{align*}
    \mathcal{R} (\Omega) \mydefn \sup_{\theta 
    \in \Omega, p' \in [1, p]} \abss{\frac{2}{n} 
  \sum_{i = 1}^{n} \rade_i Y_{i} X_{i}^{ \top} u \parenth{X_{i}^{\top} \theta}^{p' - 1}}.
\end{align*}
By the definition of $t$-covering, we obtain that
\begin{align}
\mathcal{R} (\sphere^{d - 1}) & = \sup_{\theta \in \sphere^{d - 1}, p'
  \in [1, p]} \abss{\frac{2}{n} \sum_{i = 1}^{n} \rade_i Y_{i} X_{i}^{
    \top} u \parenth{X_{i}^{\top} \theta}^{p' - 1}} \nonumber \\
& \leq \sup_{\theta_k \in \mathcal{N}( t), \enorm{ \eta} \leq
  t, p' \in [1, p]} \abss{\frac{2}{n} \sum_{i = 1}^{n}
  \rade_i Y_{i} X_{i}^{ \top} u \parenth{X_{i}^{\top} (\theta_{k} +
    \eta)}^{p' - 1}} \label{eq:bound_J1_index_second} \\
& \leq \sup_{\theta_k \in \mathcal{N}( t), p' \in [1, p]} \abss{\frac{4}{n}
  \sum_{i = 1}^n \rade_i Y_i X_{i}^{\top} u \parenth{ X_i^{ \top}
    \theta}^{p' - 1}} \nonumber \\
& + \max_{p' \in [1, p]} \sum_{b = 1}^{p' - 1}
\binom{p' - 1}{b} \cdot \sup_{\vecnorm{\eta}{2} \leq t}
\abss{\frac{4}{n} \sum_{i = 1}^n \rade_i Y_i \inprod{X_i}{u}
  \inprod{X_i}{\eta}^{b} } \nonumber \\ & \leq \mathcal{R}
(\mathcal{N}( t)) + 2^{p + 1} t \cdot \mathcal{R}
(\sphere^{d - 1}). \nonumber
\end{align}
By choosing $t = 2^{- (p + 2)}$, the above inequality leads to
\begin{align*}
\mathcal{R} (\sphere^{d - 1}) \leq 2 \mathcal{R} (\mathcal{N} (2^{-
    (p + 2)})).
\end{align*}
In order to obtain a high-probability upper bound on $\mathcal{R}
(\mathcal{N}(2^{- (p + 2)}))$, we bound its moments. By the union
bound, for any $q \geq 1$, we have
\begin{align*}
    \Exs \brackets{ \mathcal{R}^{q} \parenth{ \mathcal{N} (2^{- (p +
            2)})}} \leq p \cdot \abss{ \mathcal{N} (2^{- (p + 2)})} & \\
    & \hspace{- 8 em} \times \sup_{\theta \in \sphere^{d - 1}, p' \in [1, p]}
    \underbrace{ \Exs \brackets{ \parenth{ \abss{\frac{4}{n} \sum_{i =
              1}^n \rade_i Y_i X_{i}^{\top} u \parenth{ X_i^{ \top}
              \theta}^{p' - 1}} }^q} }_{\mydefn T_1 (\theta, p')}.
\end{align*}
In order to upper bound $T_{1}( \theta, p')$, we apply Khintchine's
inequality~\cite{Lugosi2016}; it guarantees that there is a universal
constant $C$ such that
\begin{subequations}
\begin{align}
  T_{1} (\theta, p') \leq \Exs \brackets{ \parenth{ \frac{C
        q}{n^2}\sum_{i = 1}^n Y_i^2 (X_i^{ \top} u)^2 (X_i^{\top}
      \theta)^{2 (p' - 1)}}^\frac{q}{2} }, \label{eq:bound_T1_index}
\end{align}
for any $p' \in [1, p]$ .  In order to further upper bound the right
hand side, we define the function \mbox{$g_{\theta, u} (x, y) \defn
  y^2 (x^{ \top} u)^2 (x^{\top} \theta)^{2 (p' - 1)}$.} For any $i \in
[n]$, we can verify that
\begin{align*}
\Exs \brackets{g_{\theta, u} (X_i, Y_i)} = & \Exs \brackets{Y_i^2
  \cdot \Exs \left( (X_i^{ \top} u)^2 (X_i^{\top} \theta)^{2 (p - 1)}
  \right)} \leq (2 p')^{p'},\\ \Exs \brackets{g_{\theta, u} (X_i,
  Y_i)^q} = & \Exs \brackets{Y_i^{2 q} \cdot \Exs \left( (X_i^{ \top}
  u)^{2 q} (X_i^{\top} \theta)^{2 (p' - 1) q} \right)} \leq (2 q)^q
(2p' q)^{p' q}.
\end{align*}
Given the above bounds, invoking the result
of~\cref{lemma-truncated-bernstein} leads to the following probability
bound
\begin{align*}
    \Prob \biggr( \abss{\frac{1}{n} \sum_{i = 1}^n g_{\theta, u} (X_i,
      Y_i) - \Exs_{(X, Y)} \brackets{g_{\theta, u} (X, Y)}} & \\
      & \hspace{- 6 em} > (8
    p')^{p'} \sqrt{\frac{\log 4 / \delta}{n}} + \frac{1}{n} \left(2 p'
    \log \frac{n}{\delta} \right)^{p + 1} \biggr) \leq \delta,
\end{align*}
for all $\delta \in (0,1)$.  Here the outer expectation in the above
display is taken with respect to $(X,Y)$ such that $X \sim \NORMAL(0,
I_{d})$ and $Y \mid X = x \sim \NORMAL( (x^{\top} \thetastar)^{p},
1)$.  Combining the previous bounds yields
\begin{align}
     & \hspace{- 3 em} \Exs \brackets{\left( \frac{1}{n} \sum_{i = 1}^n
    g_{\theta, u} (X_i, Y_i) \right)^{q/ 2}} \nonumber \\
  & \leq 2^{q/2} \left( \Exs_{(X, Y)} \brackets{g_{\theta, u} (X, Y)}
  \right)^{q/2} \nonumber \\
  & \hspace{8em} + 2^{q/ 2} \Exs \brackets{\abss{\frac{1}{n} \sum_{i =
        1}^n g_{\theta, u} (X_i, Y_i) - \Exs_{(X, Y)}
      \brackets{g_{\theta, u} (X, Y)} }^{q/ 2}} \nonumber \\
& \leq (4 p')^{p q} + q \int_{0}^{+ \infty} \lambda^{q - 1} \Prob
  \left( \abss{ \frac{1}{n} \sum_{i = 1}^n g_{\theta, u} (X_i, Y_i) -
    \Exs_{(X, Y)} \brackets{g_{\theta, u} (X, Y)}} > \lambda \right) d
  \lambda \nonumber \\
& \leq (4 p')^{p' q} + q \int_{0}^{1} ( p' + 1) \left( (8 p')^{p'}
  \sqrt{\frac{\log 4 / \delta}{n}} + \frac{1}{n}\left(2 p' \log
  \frac{n}{\delta} \right)^{p' + 1} \right)^{q} \log^{-1}
  \frac{4}{\delta} d \delta \nonumber \\
& \leq (4 p')^{p' q} + C p' q \biggr( \frac{(16 p')^{p' q} }{n^{
      \frac{q}{2}}} \Gamma (q / 2) \nonumber \\
  & \hspace{8 em} + \frac{(2 p')^{(p' + 1) q} }{n^q} \left( (2\log
  n)^{(p' + 1) q} + \Gamma \left({(p' + 1) q}\right) \right)
  \biggr) \label{eq:bound_T1_index_second},
\end{align}
\end{subequations}
where $\Gamma$ denotes the Gamma function.  Combining the
bounds~\eqref{eq:bound_T1_index} and~\eqref{eq:bound_T1_index_second},
we reach to the following upper bound for $T_{1}( \theta, p')$:
\begin{align}
T_{1} (\theta, p') & \leq \parenth{\frac{C q}{n}}^{q/ 2} \biggr[(4
  p')^{p' q} + C p' q \biggr( \frac{(16 p')^{p q} }{n^{ \frac{q}{2}}}
  \Gamma (q / 2) \nonumber \\ & \hspace{6 em} + \frac{(2 p' )^{(p' +
      1) q} }{n^q} \left( (2\log n)^{(p' + 1) q} + \Gamma \left({(p' +
    1) q}\right) \right) \biggr) \biggr]. \label{eq:final_bound_T1}
\end{align}

Plugging the upper bounds of $T_{1}$ in
equation~\eqref{eq:final_bound_T1} into
equation~\eqref{eq:bound_J1_index_second} and taking the union bound
over all $\theta_{k} \in \mathcal{N} \parenth{ 2^{- (p + 2)},
  \sphere^{d - 1}, \| \cdot \|_{2}}$, we find that
\begin{align*}
\Exs \brackets{ \mathcal{R}^{q} (\sphere^{d - 1})} & \leq 2^q \Exs
\brackets{ \mathcal{R}^{q} \parenth{\mathcal{N} (2^{- (p + 2)}) }}
\\ & \leq 2^q p \left( 2^{p + 3} \right)^{d} \sup_{\theta \in
  \sphere^{d - 1}, p' \in [1, p]} T_1 (\theta, p')\\ & \leq 2^q p
\left( 2^{p + 3} \right)^{d} \left( \frac{C q}{n}\right)^{\frac{q}{2}}
\biggr[ (4p)^{p q} + C p q \biggr( \frac{(16 p)^{p q} }{n^{
      \frac{q}{2}}} \Gamma (q / 2) \\
& \hspace{6 em} + \frac{(2 p )^{(p + 1) q} }{n^q} \left( (2\log
  n)^{(p + 1) q} + \Gamma \left({(p + 1) q}\right) \right) \biggr)
  \biggr],
\end{align*}
for any given $u \in \mathcal{N} \parenth{ \frac{1}{8}, \sphere^{d -
    1}, \|.\|_{2}}$.  
    
    Taking the supremum over $u \in \mathcal{N}
\parenth{ \frac{1}{8}, \sphere^{d - 1}, \|.\|_{2}}$ of both sides in
the above bound and applying Minkowski's inequality, we obtain that
\begin{align*}
    \left(\Exs |Z|^q \right)^{\frac{1}{q}} \leq & \left(\frac{64}{7}
    \right)^{d / q} \left(\Exs \brackets{\sup_{\theta \in \sphere^{d -
          1}} \abss{\frac{2}{n} \sum_{i = 1}^{n} \rade_{i} Y_{i}
        X_{i}^{ \top} u \parenth{X_{i}^{\top} \theta}^{p -
          1}}^{q}}\right)^{\frac{1}{q}}\\ \leq& 2 \left(10 \cdot 2^{p
      + 3}\right)^{d / q} \left[ \sqrt{\frac{C_p q}{n}} + \frac{C_p
        q}{n} + \frac{C_p}{n^{\frac{3}{2}}} \left( \log n + q
      \right)^{p + 1} \right],
\end{align*}
where $C_p$ is a universal constant depending only on $p$.  By
choosing $q = d (p + 7) + \log \frac{2}{\delta}$ and using Markov
inequality, we find that
\begin{align*}
    \Prob \left( |Z| \geq C_p \left( \sqrt{\frac{d + \log
        \frac{1}{\delta}}{n}} + \frac{1}{n^{3/2}} \left( d + \log
    \frac{n}{\delta} \right)^{p + 1} \right) \right) \leq \delta.
\end{align*}
Thus, we have establish the claim~\eqref{eq:empi_process_J1_index}.

\vspace{0.5em}
\noindent
\textit{Proof of claim~\eqref{eq:empi_process_J2_index}:} In order to obtain a uniform concentration bound for $J_{2}$, we use
an argument similar to that from the proof of
claim~\eqref{eq:empi_process_J1_index}.  In particular, since
polynomial $(x^{\top} \theta)^{2 p - 1}$ is homogeneous in terms of
$\theta$, using the same normalization as in
equation~\eqref{eq:normalize_J1_index}, it suffices to demonstrate
that
\begin{align}
\label{eq:surogate_J2_index}
\Prob \parenth{ W \leq \unicon r^{2 p - 1} \left( \sqrt{\frac{d + \log
      \frac{1}{\delta}}{n}} + \frac{1}{n^{3/2}} \left( d + \log
  \frac{n}{\delta}\right)^{2 p + 1} \right) } \geq 1 - \delta,
\end{align} 
for any $\delta > 0$ where we define
\begin{align*}
  W \mydefn \sup_{\theta \in \sphere^{d -
    1}} \enorm{\dfrac{1}{n} \sum_{i = 1}^{n} X_{i}
  \parenth{X_{i}^{\top} \theta}^{2 p - 1} - \Exs_{X} \brackets{X
    \parenth{ X^{\top} \theta}^{2 p - 1}}}.
\end{align*}

For each $u \in \real^d$, define the random variable
\begin{align*}
  W_{u} \mydefn \sup_{\theta \in \sphere^{d - 1}}
\abss{\frac{1}{n} \sum_{i = 1}^{n} X_{i}^{\top} u
  \parenth{X_{i}^{\top} \theta}^{2 p - 1} - \Exs_{X}
  \brackets{X^{\top} u \parenth{ X^{\top} \theta}^{2 p - 1}}}.
\end{align*}
It suffices to bound $W_{u}$ for fixed $u \in \mathcal{N} \parenth{
  \frac{1}{8}, \sphere^{d - 1}, \|.\|_{2}}$.  We bound $W_{u}$ by
controlling its moments.  By a symmetrization argument, we have
\begin{align*}
  \Exs \brackets{ \sup_{\theta \in \sphere^{d - 1}}
          \abss{\frac{1}{n} \sum_{i = 1}^{n} X_{i}^{\top} u
            \parenth{X_{i}^{\top} \theta}^{2 p - 1} - \Exs_{X}
            \brackets{X^{\top} u \parenth{ X^{\top} \theta}^{2 p -
                1}}}^{q}} & \\ & \hspace{- 13 em} \leq \Exs \brackets{
          \sup_{\theta \in \sphere^{d - 1}} \abss{\frac{2}{n} \sum_{i
              = 1}^{n} \rade_{i} X_{i}^{\top} u \parenth{X_{i}^{\top}
              \theta}^{2 p - 1}}^{q}}.
\end{align*} 
From here, we can use the same technique as that in and
after inequality~\eqref{eq:bound_J1_index_second} to bound the RHS term in the
above display. Therefore, we will only highlight the main differences
here. For any compact set $\Omega \subseteq \real^d$, we define the
random variable
\begin{align*}
\mathcal{Q} (\Omega) \mydefn \sup_{\theta \in \Omega, p' \in [1, p]}
\abss{\frac{2}{n} \sum_{i = 1}^{n} \rade_i X_{i}^{ \top} u
  \parenth{X_{i}^{\top} \theta}^{2 p' - 1}}.
\end{align*}
Following the similar argument as that
in equation~\eqref{eq:bound_J1_index_second}, we can check that $\mathcal{Q}
(\sphere^{d - 1}) \leq 2 \mathcal{Q} \parenth{ \mathcal{N} (2^{- (2 p
    + 2)}) }$. A direct application of union bound leads to
\begin{align*}
    \Exs \brackets{\mathcal{Q}^{q} \parenth{ \mathcal{N} (2^{- (2 p +
          2)}) }} \leq 2 p \cdot \abss{ \mathcal{N} (2^{- (2 p + 2)})} & \\
    & \hspace{- 9 em} \times \sup_{\theta \in \sphere^{d - 1}, p' \in [1, p]}
    \underbrace{ \Exs \brackets{ \parenth{ \abss{\frac{4}{n} \sum_{i =
              1}^n \rade_i X_{i}^{\top} u \parenth{ X_i^{ \top}
              \theta}^{2 p' - 1}} }^q} }_{\mydefn T_2 (\theta, p')}.
\end{align*}
We control $T_{2}( \theta, p')$ using the same approach as that the
proof of claim~\eqref{eq:empi_process_J1_index}.  For the convenience
of notation, we denote $h_{\theta, u} (x) \mydefn (x^\top u)^2
(x^\theta)^{2 (2 p' - 1)}$. Simple algebra lead to the following upper
bounds:
\begin{align*}
    \Exs \brackets{ h_{\theta, u} (X_i)} \leq (4 p')^{2 p'}, \ \quad
    \ \Exs \brackets{h_{\theta, u} (X_i)^q} \leq (4 p' q)^{2 p' q}.
\end{align*}
Invoking the result of~\cref{lemma-truncated-bernstein}, the above
bounds lead to the following probability bound:
\begin{align*}
    \Prob \biggr( \abss{ \frac{1}{n} \sum_{i = 1}^n h_{\theta, u}(X_i)
      - \Exs_{X} \brackets{h_{\theta, u}(X)}} & \\
      & \hspace{- 6 em} \leq (16 p')^{2 p'}
    \sqrt{\frac{\log 4 / \delta}{n} } + \parenth{ 4 p' \log
      \frac{n}{\delta}}^{2 p'} \frac{\log 4 / \delta}{n} \biggr) \leq
    \delta.
\end{align*}
Therefore, we further obtain that
\begin{align*}
    \Exs \brackets{ \left( \frac{1}{n} \sum_{i = 1}^n h_{\theta, u}
      (X_i) \right)^{q/ 2}} \leq (8 p')^{2 p' q} + C p' q \biggr(
    \frac{(32 p')^{2 p' q} }{n^{ \frac{q}{2}}} \Gamma (q / 2) &
    \\ & \hspace{- 18 em} + \frac{(4 p' )^{ (2 p' + 1) q} }{n^q} \left(
    (2\log n)^{(2 p' + 1) q} + \Gamma \left({(2 p' + 1) q}\right)
    \right) \biggr).
\end{align*}
Combining the above bound and an upper bound of $T_{2} (\theta, p')$
based on Khintchine's inequality, we obtain the following inequality:
\begin{align*}
    T_{2}( \theta, p') & \leq \parenth{\frac{C q}{n}}^{q/ 2} \biggr[ (8 p')^{2 p' q} 
    + C p' q \biggr(  \frac{(32 p')^{2 p' q} }{n^{ \frac{q}{2}}} \Gamma (q / 2) \\
    & \hspace{ 6 em} + \frac{(4 p' )^{ (2 p' + 1) q} }{n^q} 
    \left( (2 \log n)^{(2 p' + 1) q} + \Gamma \left({(2 p' + 1) q}\right) \right) \biggr) \biggr].
\end{align*}
Collecting the above bounds leads to
\begin{align*}
    \Exs \brackets{ \mathcal{Q}^{q} (\sphere^{d - 1})} & \leq 2^{q +
      1} p \left( 2^{2 p + 3} \right)^{d} \sup_{\theta \in \sphere^{d
        - 1}, p' \in [1, p]} T_2 (\theta, p')\\ & \leq 2^{q + 1} p
    \left( 2^{2 p + 3} \right)^{d} \left( \frac{C
      q}{n}\right)^{\frac{q}{2}} \biggr[ (8 p)^{2 p q} + C p q \biggr(
      \frac{(32 p)^{2 p q} }{n^{ \frac{q}{2}}} \Gamma (q / 2)
      \\ & \hspace{5 em} + \frac{(4 p )^{(2 p + 1) q} }{n^q} \left(
      (2\log n)^{(2 p + 1) q} + \Gamma \left({(2 p + 1) q}\right)
      \right) \biggr) \biggr],
\end{align*}
for any fixed $u \in \mathcal{N} \parenth{ \frac{1}{8}, \sphere^{d -
    1}, \|.\|_{2}}$.  Taking supremum over $u \in \mathcal{N}
\parenth{ \frac{1}{8}, \sphere^{d - 1}, \|.\|_{2}}$ of both sides in
the above bound and applying Minkowski's inequality, we arrive at the
following bound:
\begin{align*}
    (\Exs \brackets{ |W|^q})^{\frac{1}{q}} \leq & \left( \frac{64}{7}
  \right)^{\frac{d}{q}} \left( \Exs \brackets{ \sup_{\theta \in
      \sphere^{d - 1}} \abss{\frac{2}{n} \sum_{i = 1}^n \sigma_i
      X_i^\top u \left( X_i^\top \theta \right)^{2 p - 1} }^q}
  \right)^{\frac{1}{q}}\\ \leq & \left( \frac{10}{\varepsilon}
  \right)^{\frac{d}{q}} \left[ \sqrt{\frac{C_p q}{n} } + \frac{C_p
      q}{n} + \frac{C_p}{n^{\frac{3}{2}}} (\log n + q)^{2 p + 1}
    \right],
\end{align*}
where $C_p$ is a universal constant depending only upon $p$. With the
choice of $q = d (2 p + 7) + \log \frac{2}{\delta}$, we obtain that
\begin{align*}
    \Prob \left( |W| \geq C_p \left( \sqrt{\frac{d + \log
        \frac{1}{\delta}}{n}} + \frac{1}{n^{3/2}} \left( d + \log
    \frac{n}{\delta} \right)^{2 p + 1} \right) \right) \leq \delta.
\end{align*}
Thus, we have established the claim~\eqref{eq:empi_process_J2_index}.


\subsection{Proof of~\cref{cor:Gaussian_mixture}}
\label{subsec:proof:cor:Gaussian_mixture}

We prove~\cref{cor:Gaussian_mixture} by verifying the
claims~\eqref{eq:weak_conv_gaus} and~\eqref{eq:empi_process_gaus}.


\subsubsection{Structure of $\loglihoodgaus$}
\label{subsec:proof_geometry_weak_con}

Direct algebra leads to the following equation
\begin{align}
\label{eq:weak_con_gaus_first}
  \inprod{ \nabla \loglihoodgaus(\theta)}{ \thetastar - \theta } & =
  \parenth{\theta - \Exs \brackets{ X \tanh \parenth{ X^{ \top}
        \theta}}}^{\top} ( \theta - \thetastar) \notag \\ & \geq
  \enorm{ \theta}^2 - \enorm{\theta} \enorm{\Exs \brackets{ X \tanh
      \parenth{ X^{\top} \theta}}}
\end{align}
where $\tanh(x) \mydefn \frac{\exp( x) - \exp( - x)}{\exp( x) + \exp(
  - x)}$ for all $x \in \Rspace$.  From Theorem 2 in Dwivedi et
al.~\cite{Raaz_Ho_Koulik_2018}, we have
\begin{align*}
  \enorm{\Exs \brackets{ X \tanh \parenth{ X^{\top} \theta}}} \leq
  \parenth{1 - p + \frac{p}{1 + \frac{ \enorm{ \theta}^2}{2} }}
  \enorm{ \theta}
\end{align*} 
for all $\theta \in \Rspace^{d}$ where $p \mydefn \Prob \parenth{
  \abss{Y} \leq 1} + \frac{1}{2} \Prob \parenth{ \abss{Y} > 1}$ where
$Y \sim \NORMAL(0, 1)$.  Plugging the above inequality
into equation~\eqref{eq:weak_con_gaus_first} leads to
\begin{align*}
  \inprod{ \nabla \loglihoodgaus(\theta)}{ \thetastar - \theta} 
  \geq \frac{ p \enorm{ \theta}^4}{2 + \enorm{ \theta}^2} 
  \geq \begin{cases} \frac{p}{4} \enorm{ \theta}^4, 
  \quad & \text{for} \ \enorm{ \theta} \leq \sqrt{2} 
    \\ \frac{p}{2} \parenth{\enorm{ \theta}^2 - 1}, \quad & \text{otherwise} \end{cases}.
\end{align*}
As a consequence, we achieve the conclusion of
claim~\eqref{eq:weak_conv_gaus}.


\subsubsection{Perturbation error between
  $\nabla \loglihoodgaus$ and $\nabla \loglihoodgaus_{n}$}
\label{subsec:uniform_perturb_mixgaus}

Direct calculation indicates the following equation:
\begin{align*}
  \nabla \loglihoodgaus_n (\theta) - \nabla \loglihoodgaus (\theta) =
  \frac{1}{n} \sum_{i = 1}^{n} X_{i} \tanh(X_{i}^{\top} \theta) - \Exs
  \brackets{ X \tanh \parenth{ X^{ \top} \theta}}.
\end{align*}
The outer expectation in the above display is taken with respect to $X
\sim \NORMAL(\thetastar, \sigma^2 I_{d})$ where $\thetastar =
0$. Based on the proof argument of Lemma 1 from the
paper~\cite{Raaz_Ho_Koulik_2018}, for each $r > 0$, we have the
following concentration inequality
\begin{align}
  \Prob \biggr( \sup_{ \theta \in \ball(\thetastar, r)} 
  \enorm{ \frac{1}{n} \sum_{i = 1}^{n} X_{i} \tanh(X_{i}^{\top} \theta) 
  - \Exs \brackets{ X \tanh \parenth{ X^{ \top} \theta}}} 
  & \nonumber \\
  & \hspace{- 8 em} \leq c r \sqrt{\frac{d + \log(1 / \delta)}{n}} \biggr) \geq 1 - \delta, \label{eq:empirical_process_gauss_fix_radius}
\end{align}
for any $\delta > 0$ as long as the sample size $n \geq c' d \log(1/
\delta)$ where $c$ and $c'$ are universal constants. For any $M \in
\mathbb{N}_{+}$, by the concentration
bound~\eqref{eq:empirical_process_gauss_fix_radius} and the union
bound, we find that
\begin{align} 
\Prob \biggr( \forall r \in [2^{-M}, 1], ~ \sup_{ \theta \in
  \ball(\thetastar, r)} \vecnorm{\nabla \loglihoodgaus_n (\theta) -
  \loglihoodgaus (\theta)}{2} & \nonumber \\
  & \hspace{- 4 em} \leq \unicon \; r \; \sqrt{\frac{d +
    \log (M/\delta)}{n}} \biggr) \geq 1 - \delta. \label{eq:empirical_gauss_first} 
\end{align}

On the other hand, based on the standard inequality $\abss{ \tanh(x)}
\leq \abss{ x}$ for all $x \in \Rspace$, we find that
\begin{align*}
  \enorm{\nabla \loglihoodgaus_n (\theta) 
  - \nabla \loglihoodgaus (\theta) } 
  & \leq \frac{1}{n} \sum_{i = 1}^{n} \enorm{ X_{i}} \abss{ 
  \tanh \parenth{ X_{i}^{ \top} \theta}} + \Exs \brackets{ 
  \enorm{X} \abss{ \tanh \parenth{ X^{ \top} \theta}}} \\
  & \leq \frac{1}{n} \sum_{i = 1}^{n} \enorm{ X_{i}} \abss{ 
  X_{i}^{ \top} \theta} + \Exs \brackets{ \enorm{X} \abss{ X^{ \top} 
  \theta}} \\
  & \leq \parenth{ \frac{1}{n} \sum_{i = 1}^{n} \enorm{ X_{i}}^2 
  + \Exs \brackets{ \enorm{X}^2}} \enorm{ \theta}.
\end{align*}
Therefore, we have $\enorm{\nabla \loglihoodgaus_n (\theta) 
- \nabla \loglihoodgaus (\theta) } 
\leq 2 d \vecnorm{\theta}{2} \log (1/ \delta)$ 
with probability $ 1 - \delta$. By choosing 
$M_1 : = \log (2nd)$, based on the previous bound, we obtain that
\begin{align}
    \Prob \left( \forall r < 2^{- M_1}, ~\sup_{ \theta \in \ball(\thetastar, r)} 
    \enorm{\nabla \loglihoodgaus_n (\theta) - \nabla 
    \loglihoodgaus (\theta) } \leq \frac{ \log (1/ \delta)}{n} \right) 
    \geq 1 - \delta. \label{eq:empirical_gauss_second}
\end{align}
Furthermore, for vector $\theta \in \Rspace^{d}$ with large norm, by
the concentration bound~\eqref{eq:empirical_process_gauss_fix_radius}
combined with the union bound, for any $M' \in \mathbb{N}_{+}$, we find
that
\begin{align*}
  \mathbb{P} \biggr( \forall r \in [1, 2^{M'}], \sup_{\theta \in \ball
    (\thetastar, r)} \vecnorm{\nabla \loglihoodgaus_n (\theta) -
    \loglihoodgaus (\theta)}{2} & \\
    & \hspace{ - 3 em} \leq \unicon \; r \; \sqrt{\frac{d +
      \log (M'/ \delta)}{n}} \biggr) \geq 1 -\delta.
\end{align*}
When $r$ in the above bound is too large, we can simply use the fact
that $\tanh$ is a bounded function. We thus have the upper bound
\begin{align*}
\vecnorm{\nabla \loglihoodgaus_n (\theta) - \nabla \loglihoodgaus
  (\theta))}{2} \leq \mathbb{E} \brackets{\vecnorm{X}{2}} +
\frac{1}{n} \sum_{i = 1}^n \vecnorm{X_i}{2},
\end{align*}
for any $\theta$. Given the above bound, by choosing $M_2 \defn \log
(2 \sqrt{n})$, we obtain that
\begin{align}
 & \Prob \left( \forall r > 2^{M_2}, \sup_{ \theta \in
    \ball(\thetastar, r)} \vecnorm{ \nabla \loglihoodgaus_n (\theta) -
    \nabla \loglihoodgaus (\theta))}{2} \leq r \sqrt{\frac{d + \log(
      1/ \delta)}{n}} \right) \nonumber \\
  & \hspace{5 em} \geq \Prob \left( \mathbb{E} \brackets{
  \vecnorm{X}{2}} + \frac{1}{n} \sum_{i = 1}^n \vecnorm{X_i}{2} \leq
2^{M_2} \sqrt{ \frac{d + \log( 1/ \delta)}{n}} \right) \geq 1 -
\delta. \label{eq:empirical_gauss_third}
\end{align}
Putting the bounds~\eqref{eq:empirical_gauss_first},~\eqref{eq:empirical_gauss_second},
and~\eqref{eq:empirical_gauss_third} together, for $n \geq \unicon d
\log(1/ \delta)$, the following probability bound holds
\begin{align*}
  \Prob \biggr( \forall r > 0, \sup_{ \theta \in \ball(\thetastar, r)}
  \vecnorm{ \nabla \loglihoodgaus_n (\theta) - \nabla \loglihoodgaus
    (\theta))}{2} & \\
    & \hspace{- 6 em} \leq \unicon \; r \; \sqrt{\frac{d + \log
      \parenth{\log n/ \delta}}{n}} + \frac{\log (1 /
    \delta)}{n} \biggr) \geq 1 - \delta,
\end{align*}
which completes the proof of the claim~\eqref{eq:empi_process_gaus}.

\subsection{Proof of~\cref{cor:general-gaussian-mixture}}

\label{app:subsec-proof-general-gaussian-mixture}

We prove this claim by verifying the conditions
in~\cref{cor:final-nonconvex}
and~\cref{thm:non-asymp-credible-set}. In particular, we claim the
following bounds on the population log-likelihood $\loglihood (\theta)
= \Exs \left[ \log p_\theta(X) \right]$ and its empirical counterpart
$\loglihood_\numobs(\theta) = \frac{1}{\numobs} \sum_{i = 1}^\numobs
\log p_\theta (X_i)$: For each permutation function $\sigma$ and $R_0
> 0$, we have
\begin{subequations}
  \begin{align}
\opnorm{\nabla^2 \loglihood (\theta) - \nabla^2 \loglihood
  (\thetastar_\sigma)} &\leq c K \left( \vecnorm{\theta -
  \thetastar_\sigma}{2}^3 + \sigma_X^3 \right) \vecnorm{\theta -
  \thetastar_\sigma}{2}^, \quad \forall \theta \in \real^{d
  K}, \label{eq:k-mixture-hessian-lip} \\
\sup_{\theta \in \ball (\thetastar_\sigma, R_0)} \vecnorm{\nabla
  \loglihood_\numobs (\theta) - \nabla \loglihood (\theta)}{2} & \leq
c K (\sigma_X + R_0) \sqrt{\frac{K d \log (K d) + \log
    \delta^{-1}}{\numobs}}, \quad \mbox{w.p. } 1 -
\delta, \label{eq:k-mixture-grad-emp-proc} \\
\sup_{\theta \in \ball (\thetastar_\sigma, R_0)} \opnorm{\nabla^2
  \loglihood_\numobs (\theta) - \nabla^2 \loglihood (\theta)} &\leq c
K^2 (\sigma_X^2 + R_0^2) \sqrt{\frac{K d \log (K d) +
    \log(1/\delta)}{\numobs}}\quad \mbox{w.p. } 1 -
\delta, \label{eq:k-mixture-hessian-emp-proc} \\
\posterior \Big( \ball^{c} \big(0, 3 \sqrt{K} R_X + & \sqrt{K (d \log
  (R_X \numobs + \smoothprior) + \log (\prior_0 (1/\vartheta) )} \big)
\mid \DataX \Big) < \vartheta,
\label{eq:k-mixture-far-tail}     
    \end{align}
\end{subequations}
where $R_X \mydefn \max_{i} \vecnorm{X_i}{2}$.  The proofs of these
bounds are deferred to later subsections.  Taking the four bounds as
given, we now proceed with the proof of the corollary.

First, we define $\localradius \mydefn \frac{\mu}{4 c'K (\sigma_X^3 +
  1)} \wedge \big( \frac{\mu}{4 c' K} \big)^{1/4} \wedge 1$, for any
permutation function $\sigma$ and any $\theta \in \ball
(\thetastar_\sigma, \localradius )$. Then,
equation~\eqref{eq:k-mixture-hessian-lip} guarantees the following
local bound:
\begin{align*}
    - \nabla^2 \loglihood (\theta) \succeq - \nabla^2 \loglihood
    (\thetastar_{\sigma}) - \frac{\mu}{2} I_d \succeq \frac{\mu}{2}
    I_d,
\end{align*}
which implies the condition $- \inprod{\nabla \loglihood (\theta)}{\theta - \thetastar_\sigma} 
\geq \frac{\mu}{2} \vecnorm{\theta - \thetastar_\sigma}{2}^2$ inside the ball $\ball (\thetastar_\sigma, \localradius)$.

Combining this bound with equation~\eqref{eq:k-mixture-grad-emp-proc}
by taking $R_0 = \localradius$, we invoke~\cref{thm:local-weak-convex}
and obtain the following localized posterior contraction bound with
probability $1 - \delta$:
\begin{align}
    \posterior \big( \ball (\thetastar_\sigma, \localradius) \mid
    \DataX \big)^{-1} \posterior \left( \ball \left(
    \thetastar_\sigma, r_\numobs \right) \mid \DataX \right) \geq 1 -
    \vartheta,\label{eq:k-mixture-posterior-near-bound}
\end{align}
where the contraction radius $r_\numobs$ is given by:
\begin{align*}
    r_\numobs \mydefn \frac{c K \sigma_X}{\strongconvex} \sqrt{\frac{K
        d \log (K d) + \log \delta^{-1}}{\numobs}} + c
    \sqrt{\frac{\log \vartheta^{-1}}{\strongconvex \numobs}}.
\end{align*}
By the sub-Gaussian condition, we obtain that
\begin{align*}
    \max_i \vecnorm{X_i}{2} \leq 2\sigma_X \sqrt{d + \log \numobs +
      \log \delta^{-1}}, \quad \mbox{with probability } 1 - \delta.
\end{align*}
Therefore, the following tail bound holds true with probability $1 -
\delta$:
\begin{align}
     \posterior \Big( \ball^{c} \parenth{0, c \sigma_X \sqrt{K d \log
         \frac{\numobs}{\vartheta\delta \prior_0}} } \; \mid \; \DataX
     \Big) < \vartheta. \label{eq:k-mixture-posterior-far-bound}
\end{align}
Taking $R_0 = c \sigma_X \sqrt{K d \log \frac{\numobs}{\vartheta\delta
    \prior_0}}$ and applying
equation~\eqref{eq:k-mixture-grad-emp-proc}, we conclude that there
exists a quantity $a_0 > 0$ depending on the constants $K, d,
\sigma_X$, such that
\begin{align*}
    \sup_{\theta \in \ball (0, R_0)} \vecnorm{\nabla
      \loglihood_\numobs (\theta) - \nabla \loglihood_\numobs
      (\thetastar)}{2} \leq \frac{a_0\big(1 + \log \delta^{-1} + \log
      \vartheta^{-1}\big)}{\sqrt{\numobs}}.
\end{align*}
Note that, $\loglihood_\numobs (0) = - \frac{1}{\numobs} \sum_{i =
  1}^\numobs \vecnorm{X_i}{2}^2$.  An application of sub-exponential
concentration bounds~\cite{Wainwright_nonasymptotic} leads to
\begin{align*}
    \Prob \left( \abss{\loglihood_\numobs (0) - \loglihood (0)} >
    \sigma_X^2 \frac{c d \log \delta^{-1}}{\sqrt{\numobs}} \right) <
    \delta.
\end{align*}
Combining the previous two bounds, there exists $a_1 > 0$, such that
the following bound holds true with probability $1 - \delta$:
\begin{align*}
    \sup_{\theta \in \ball (0, R_0)} \abss{\loglihood_\numobs (\theta)
      - \loglihood_\numobs (\thetastar)} \leq \frac{a_1\big(1 + \log
      \delta^{-1} + \log \vartheta^{-1}\big)}{\sqrt{\numobs}}.
\end{align*}
On the other hand, since the Gaussian mixture model is identifiable up
to permutations, there exists $\Delta_0 > 0$ depending on
$\thetastar$, such that:
\begin{align*}
    \inf_{\theta \in \left( \bigcup_{\sigma: [K] \rightarrow [K]}
      \ball (\thetastar_\sigma, \localradius) \right)^{c}} \loglihood
    (\thetastar_{\mathrm{Id}}) - \loglihood (\theta) \geq \Delta_0.
\end{align*}
Consequently, for $\numobs \geq \left( \frac{3 a_0}{\Delta_0} \big(
\log \delta^{-1} + \log \vartheta^{-1} \big) \right)^2$, with
probability $1 - \delta$ we have that
\begin{align*}
    \loglihood_\numobs (\theta') \leq \loglihood_\numobs (\theta) -
    \frac{\Delta_0}{3}
\end{align*}
for all $\sigma: [K] \rightarrow [K], ~ \theta \in \ball \Big(\thetastar_\sigma, \sqrt{ \frac{\Delta_0}{d + \log \delta^{-1}} } \Big)$ 
and $\theta' \in \ball (0, R_0) \setminus \bigcup_{\sigma': [K] \rightarrow [K]} \ball (\thetastar_{\sigma'}, \localradius)$. Thus, we have the posterior probability bound:
\begin{align*}
    \posterior \left[\ball (0, R_0) \setminus \bigcup_{\sigma': [K] \rightarrow [K]} 
    \ball (\thetastar_{\sigma'}, \localradius) \mid \DataX \right] \leq \exp \left( - \frac{\Delta_0}{3} \numobs \right) \cdot \left( \frac{R_0 \sqrt{d + \log \delta^{-1}}}{ \sqrt{\Delta_0}} \right)^{d K}.
\end{align*}
It indicates that there exists $a_2 > 0$ depending on the problem instances $\thetastar$, $K$, $d$, 
such that for $\numobs \geq  \frac{a_2}{\Delta_0} \log \vartheta^{-1}$, we have the following bound:
\begin{align}
    \posterior \left[\ball (0, R_0) \setminus \bigcup_{\sigma': [K] \rightarrow [K]} 
    \ball (\thetastar_{\sigma'}, \localradius) \mid \DataX \right] \leq \vartheta.\label{eq:k-mixture-posterior-middle-bound}
\end{align}
Collecting the bounds~\eqref{eq:k-mixture-posterior-near-bound},~\eqref{eq:k-mixture-posterior-far-bound} and~\eqref{eq:k-mixture-posterior-middle-bound}, 
we conclude that for $\numobs \geq \numobs_{\min} \cdot \log^2 \frac{1}{\delta \vartheta}$, the following bound holds true with probability $1 - \delta$:
\begin{align*}
    \posterior \left[  \bigcup_{\sigma: [K] \rightarrow [K]} \ball (\thetastar_{\sigma}, r_\numobs) \mid \DataX \right] \geq 1 - \vartheta,
\end{align*}
for contraction radius $r_\numobs$ defined as:
\begin{align*}
    r_\numobs \mydefn \frac{c K \sigma_X}{\strongconvex} \sqrt{\frac{K d \log (K d) + \log \delta^{-1}}{\numobs}} + c \sqrt{\frac{\log \vartheta^{-1}}{\strongconvex \numobs}},
\end{align*}
which proves the bound~\eqref{eq:general-mixture-posterior-contraction}.

Furthermore, applying~\cref{thm:non-asymp-credible-set} to each local
neighborhood $\ball (\thetastar_\sigma, \localradius)$, for any
$\offpar \in (0, 1)$, we obtain the following bound with probability
$1 - \delta$:
\begin{multline*}
    \posterior \left[ \vecnorm{\theta -
        \thetamap_\sigma}{\HessianStar_\sigma}^2 \leq (1 + \offpar)
      \frac{d}{n} + c \frac{1 + \log \kappa
        (\HessianStar_\sigma)}{\offpar} \Big( \frac{\log
        \vartheta^{-1}}{\numobs} + \frac{a' (\log \vartheta^{-1} +
        \log \delta^{-1})^2}{\numobs^2} \Big) \mid \DataX \right]
    \\ \geq (1 - \vartheta) \posterior \left( \ball
    (\thetastar_\sigma, \localradius) \mid \DataX\right),
\end{multline*}
for a constant $a' > 0$ depending on $K, d$ and $\thetastar$.

Combining with the tail
bounds~\eqref{eq:k-mixture-posterior-far-bound}
and~\eqref{eq:k-mixture-posterior-middle-bound}, we obtain the
result~\eqref{eq:general-mixture-bvm-type-contraction}.

\subsubsection{Proof of the claim~\eqref{eq:k-mixture-hessian-lip}}

We first verify the local
conditions~\ref{item:without_global_geometry}
and~\ref{item:without_global_deviation}. Given the parameters
$(\gausscenter_1, \gausscenter_2, \cdots, \gausscenter_K)$, direct
calculation yields
\begin{align*}
    - \nabla_\theta \loglihood_\numobs (\theta) = \left[
      \frac{1}{\numobs} \sum_{i = 1}^n (\gausscenter_j - X_i)
      \frac{\exp \left( - \vecnorm{\gausscenter_j - X_i}{2}^2 / 2
        \right)}{\sum_{\ell = 1}^K\exp \left( -
        \vecnorm{\gausscenter_\ell - X_i}{2}^2 / 2 \right)}\right]_{j
      \in [K]}.
\end{align*}
Given distinct centers $(\gausscenter_j)_{j \in [K]}$ of each mixture
component, we have that $\HessianStar_\sigma \succ 0$ for any
permutation $\sigma$. To show the local growth
condition~\ref{item:without_global_geometry}, we study the local
conditions around $\thetastar_\sigma$. Denote $q_j (x; \theta) \mydefn
\frac{\exp \left( - \vecnorm{\gausscenter_j - x}{2}^2 / 2
  \right)}{\sum_{\ell = 1}^K\exp \left( - \vecnorm{\gausscenter_\ell -
    x}{2}^2 / 2 \right)} $ for any $x \in \real^{d}$ and $j \in
     [K]$. Direct calculation shows that
\begin{align*}
   - \nabla_\theta^2 \log p_\theta (X) = \mathrm{diag} \left(\big( I_d + (\gausscenter_j - X) (\gausscenter_j - X)^\top \big) q_j (X; \theta)  \right)_{j \in [K]} & \\
   & \hspace{-10 em} - \left[ (\gausscenter_j - X) (\gausscenter_\ell - X)^\top q_j (X; \theta) q_\ell (X; \theta) \right]_{j , \ell \in [K]}.
\end{align*}
For the third-order derivative, for any vector $v = \left[ v_1 ~ v_2 ~ \cdots v_K \right] \in \sphere^{K d - 1}$, direct calculation leads to
\begin{align*}
    &\opnorm{\nabla^3 \loglihood (\theta) [v]} \\
    &\leq \opnorm{ \begin{bmatrix} \opnorm{\Exs \left[ \nabla_{\gausscenter_\ell} q_j (X; \theta) \left( I_d + (\gausscenter_j - X) (\gausscenter_j - X)^\top \right)v_j \right]} \end{bmatrix}_{j, \ell \in [K]}} \\
    &\qquad +  \max_{j \in [K]} \abss{\Exs \left[ q_j (X; \theta) (\gausscenter_j - X)^\top v_j \right]} + \max_{j \in [K]} \abss{\Exs \left[ q_j (X; \theta) \sum_{\ell \in [K]} q_\ell (X; \theta) (\gausscenter_\ell - X)^\top v_\ell \right]}\\
    &\qquad + \opnorm{\left[ \opnorm{\Exs \left[ q_j (X; \theta) q_\ell (X; \theta) v_\ell (\gausscenter_\ell - X)^\top \right]} \right]_{j, \ell \in [K]}}\\
    &\qquad \opnorm{\left[ \opnorm{\sum_{\ell \in [K]} \Exs \left[  q_\ell (X; \theta) (\gausscenter_j - X) (\gausscenter_\ell - X)^\top v_\ell \nabla_{\gausscenter_k} q_j (X; \theta) \right]} \right]_{j, k \in [K]}}\\
    &\qquad + \opnorm{\left[ \opnorm{\sum_{\ell \in [K]} \Exs \left[  q_j (X; \theta) (\gausscenter_j - X) (\gausscenter_\ell - X)^\top v_\ell \nabla_{\gausscenter_k} q_\ell (X; \theta) \right]} \right]_{j, k \in [K]}}.
\end{align*}
Using H\"{o}lder inequality and the variational representation of the operator norm, we obtain that
\begin{align*}
    &\opnorm{\nabla^3 \loglihood (\theta) [v]}\\
    &\leq c K \cdot \sup_{\stackrel{y, z \in \sphere^{d - 1}}{ j, k, \ell \in [K]}} \Exs \left[\abss{ (X - u_j)^\top y \cdot (X - u_k)^\top z \cdot (X - u_\ell)^\top v_\ell } \right]  \\
    & \hspace{18 em} + c K \cdot \sup_{\stackrel{y, z \in \sphere^{d - 1}}{\ell \in [K]}} \Exs \left[\abss{ y^\top z (X - u_\ell)^\top v_\ell } \right]\\
    &\leq c K \cdot \sup_{\stackrel{y, z \in \sphere^{d - 1}}{ j, k, \ell \in [K]}} \Exs \left[\abss{ (X - u_j)^\top y }^3 \right]^{1/3} \cdot  \Exs \left[\abss{ (X - u_k)^\top z }^3 \right]^{1/3} \cdot  \Exs \left[\abss{ (X - u_\ell)^\top v_\ell }^3 \right]^{1/3} \\
    & \hspace{18 em} + c K \cdot \sup_{\stackrel{y, z \in \sphere^{d - 1}}{\ell \in [K]}} \Exs \left[ \big( (X - u_\ell)^\top v_\ell \big)^2 \right]^{1/2}\\
    &\leq c' K \left( \vecnorm{\theta - \thetastar_{\sigma}}{2}^3 + \sigma_X^3 + 1\right),
\end{align*}
for a universal constant $c' > 0$ and any permutation function $\sigma$. This proves the desired claim.

\subsubsection{Proof of the claims~\eqref{eq:k-mixture-grad-emp-proc} and~\eqref{eq:k-mixture-hessian-emp-proc}}
Now we turn to the empirical process bounds for the gradient and Hessian of $\loglihood_\numobs$. For $\theta \in \real^{d K}$ and $v, w \in \sphere^{d K - 1}$, we define the following quantities
\begin{align*}
  Y_{\theta, v}^{(1)} &\mydefn \inprod{\nabla F_\numobs   (\theta)}{v} 
  = \frac{1}{\numobs} \sum_{j \in [K]} \sum_{i = 1}^\numobs (\gausscenter_j - X_i)^\top v_j q_j (X_i, \theta) ,\qquad \mbox{and}\\
  Y_{\theta, v, w}^{(2)} &\mydefn v^\top \nabla^2 F_\numobs   (\theta)  w \\
  &= \frac{1}{\numobs} \sum_{j = 1}^K \sum_{i = 1}^\numobs  \big( v^\top w + v_j^\top (\gausscenter_j - X_i) w_j^\top (\gausscenter_j - X_i) \big) q_j (X_i; \theta)\\
  &\qquad -  \frac{1}{\numobs}  \sum_{j, \ell \in [K]}  \sum_{i = 1}^\numobs  v_j^\top (\gausscenter_j - X_i) \cdot w_\ell^\top (\gausscenter_\ell - X_i) \cdot q_j (X_i, \theta)  q_\ell (X_i, \theta).
\end{align*}
We further define $Z_{\theta, v}^{(1)} \mydefn Y_{\theta, v}^{(1)} - \Exs \left[ Y_{\theta, v}^{(1)} \right]$ and $Z_{\theta, v, w}^{(2)} \mydefn Y_{\theta, v}^{(2)} - \Exs \left[ Y_{\theta, v}^{(2)} \right]$.

In the following derivation, we first regard the vectors $v, w$ as
fixed, and then use standard discretization approach to take the
maximum with respect to both vectors.  Similar to the proof
of~\cref{cor:logit_regres}, we use~\cref{prop:adamczak} to control the
concentration behavior of the above quantities.  Note that, $q_j$ is a
bounded function for each $j \in [K]$.  Therefore, by
applying~\cref{prop:adamczak} to each term of $Z_{\theta, v}^{(1)}$
with envelop function $\bar{G}^{(1)} (X) = 1 + R_0 +
|(\gausscenterstar_j - X)^\top v_j|$ for each $j \in [K]$, we obtain
the following bound with probability $1 - \delta$:
\begin{align*}
    \sup_{\theta \in \ball (\thetastar_\sigma, R_0)} Z_{\theta,
      v}^{(1)} \leq 2 \Exs \left[ \sup_{\theta \in \ball
        (\thetastar_\sigma, R_0)} Z_{\theta, v}^{(1)} \right] + K (1 +
    R_0 + \sigma_X) \left( \sqrt{\frac{\log \delta^{-1}}{\numobs}} +
    \frac{\log \delta^{-1}}{\numobs} \sqrt{\log \numobs} \right).
\end{align*}
Similarly, by applying~\cref{prop:adamczak} to each term of
$Z_{\theta, v}^{(2)}$ with envelop function $\bar{G}^{(2)} (X) = 1 +
R_0 \cdot \left( \abss{v_j^\top (\gausscenterstar_j - X_i) } + \abss{
  w_\ell^\top (\gausscenterstar_\ell - X_i)} \right) + R_0^2 +
\abss{v_j^\top (\gausscenterstar_j - X_i) \cdot w_\ell^\top
  (\gausscenterstar_\ell - X_i)}$, for each $j, \ell \in [K]$, we
obtain the following bound with probability $1 - \delta$:
\begin{align*}
    \sup_{\theta \in \ball (\thetastar_\sigma, R_0)} Z_{\theta, v,
      w}^{(2)} \leq 2 \Exs \left[ \sup_{\theta \in \ball
        (\thetastar_\sigma, R_0)} Z_{\theta, v, w}^{(2)} \right] + K^2
    (1 + R_0^2 + \sigma_X^2 ) \left( \sqrt{\frac{\log
        \delta^{-1}}{\numobs}} + \frac{\log \delta^{-1}}{\numobs} \log
    \numobs \right).
\end{align*}

Now, we consider the function classes
\begin{align*}
    \mathcal{G}^{(1)}_{v, j} &\mydefn \left\{ x \mapsto \inprod{\nabla
      \log p_\theta (x)}{v_j}: \theta \in \ball (\thetastar_\sigma,
    R_0) \right\}, \quad \mbox{and}\\ \mathcal{G}^{(2)}_{v, w, j,
      \ell} &\mydefn \left\{ x \mapsto \inprod{\nabla^2 \log p_\theta
      (X_i) v_j}{w_\ell}: \theta \in \ball (\thetastar_\sigma, R_0)
    \right\}.
\end{align*}
Apparently, $\bar{G}^{(1)}$ and $\bar{G}^{(2)}$ are envelop functions
for the corresponding classes $\mathcal{G}^{(1)}_{v, j}$ and
$\mathcal{G}^{(2)}_{v, w, j, \ell}$.  In order to bound the expected
suprema, we define the following symmetrized random variables:
\begin{align*}
    V_{\theta, v}^{(1)} = \frac{1}{\numobs} \sum_{i = 1}^\numobs
    \rade_i \inprod{\nabla \log p_\theta (X_i)}{v}, \quad \mbox{and}
    \quad V_{\theta, v, w}^{(2)} = \frac{1}{\numobs} \sum_{i =
      1}^\numobs \rade_i \inprod{\nabla^2 \log p_\theta (X_i) v}{w},
\end{align*}
for $\mathrm{i.i.d.}$ Rademacher random variables $(\rade_i)_{i =
  1}^\numobs$.  Standard symmetrization arguments imply that $\Exs
\left[ \sup_{\theta \in \ball (\thetastar_\sigma, R_0)} Z^{(i)}
  \right] \leq 2 \Exs \left[ \sup_{\theta \in \ball
    (\thetastar_\sigma, R_0)} V^{(i)} \right]$ for $i \in \{1,2\}$.

Let $P_\numobs \mydefn \frac{1}{\numobs} \sum_{i = 1}^\numobs
\delta_{X_i}$, we claim the following covering number bounds,
conditionally on the data $X_1^\numobs$:
\begin{subequations}
\begin{align}
   \bar{N}^{(1)} (t) \mydefn \abss{ \mathcal{N} \left(
     \mathcal{G}^{(1)}_{v, j}, \vecnorm{\cdot}{L^2 (P_\numobs)}, t
     \vecnorm{\bar{G}^{(1)}}{L^2 (P_\numobs)} \right) } & \nonumber
   \\ & \hspace{-5 em} \leq \left( \frac{c \sum_{k' = 1}^K
     \vecnorm{\gausscenterstar_{k'} - X_i}{2}^2 + c K (R_0^2 + 1)}{t}
   \right)^{K
     d}, \label{eq:k-mixture-covering-gradient}\\ \bar{N}^{(2)} (t)
   \mydefn \abss{ \mathcal{N} \left( \mathcal{G}^{(2)}_{v, w, j,
       \ell}, \vecnorm{\cdot}{L^2 (P_\numobs)}, t
     \vecnorm{\bar{G}^{(2)}}{L^2 (P_\numobs)} \right) } & \nonumber
   \\ & \hspace{-5 em} \leq \left( \frac{c K \sum_{k' = 1}^K
     \vecnorm{\gausscenterstar_{k'} - X_i}{2}^3 + c K^2 (R_0^3 +
     1)}{t} \right)^{K d}. \label{eq:k-mixture-covering-hessian}
\end{align}
\end{subequations}
By Dudley's chaining integral bound, we obtain the following bounds:
\begin{align*}
    \Exs \left[ \sup_{\theta \in \ball (\thetastar_\sigma, R_0)}
      V^{(1)} \right] & \leq \sqrt{\frac{1}{\numobs }\Exs \left[
        \bar{G}^{(1)} (X)^2 \right]} \int_0^1 \sqrt{1 + \Exs \left[
        \log \bar{N}^{(1)} (t) \right]} dt \\ & \leq c K (1 + R_0 +
    \sigma_X) \sqrt{\frac{K d \log (K d)}{\numobs}},\\ \Exs \left[
      \sup_{\theta \in \ball (\thetastar_\sigma, R_0)} V^{(2)} \right]
    &\leq \sqrt{\frac{1}{\numobs }\Exs \left[ \bar{G}^{(2)} (X)^2
        \right]} \int_0^1 \sqrt{1 + \Exs \left[ \log \bar{N}^{(2)} (t)
        \right]} dt \\ & \leq c K^2 (1 + R_0^2 + \sigma_X^2)
    \sqrt{\frac{K d \log (K d)}{\numobs}}.
\end{align*}
Combining with the concentration inequalities, we obtain the following
bounds with probability $1 - \delta$:
\begin{align*}
     \sup_{\theta \in \ball (\thetastar_\sigma, R_0)} Z_{\theta,
       v}^{(1)} & \leq c K (1 + R_0 + \sigma_X) \left[ \sqrt{\frac{K d
           \log (K d) + \log \delta^{-1}}{\numobs}} + \frac{\log
         \delta^{-1}}{\numobs} \sqrt{\log \numobs}
       \right],\\ \sup_{\theta \in \ball (\thetastar_\sigma, R_0)}
     Z_{\theta, v}^{(2)} & \leq c K^2 (1 + R_0^2 + \sigma_X^2) \left[
       \sqrt{\frac{K d \log (K d) + \log \delta^{-1}}{\numobs}} +
       \frac{\log \delta^{-1}}{\numobs} \log \numobs \right].
\end{align*}
Finally, by taking union bound over a maximal $\frac{1}{8}$-packing of
the sphere $\sphere^{d - 1}$, which has cardinality bounded by $17^{K
  d}$, for $\sigma_X \geq 1$ and $\frac{\numobs}{\log \numobs} \geq K
d \log \frac{K d}{\delta}$, we conclude that
\begin{align*}
     \sup_{\theta \in \ball (\thetastar_\sigma, R_0)} \vecnorm{\nabla
       \loglihood_\numobs (\theta) - \nabla \loglihood (\theta)}{2} &
     \leq c K (\sigma_X + R_0) \sqrt{\frac{K d \log (K d) + \log
         \delta^{-1}}{\numobs}},\\ \sup_{\theta \in \ball
       (\thetastar_\sigma, R_0)} \opnorm{\nabla^2 \loglihood_\numobs
       (\theta) - \nabla^2 \loglihood (\theta)} & \leq c K^2
     (\sigma_X^2 + R_0^2) \sqrt{\frac{K d \log (K d) + \log
         \delta^{-1}}{\numobs}},
\end{align*}
which proves the desired bounds in
claims~\eqref{eq:k-mixture-grad-emp-proc}
and~\eqref{eq:k-mixture-hessian-emp-proc}.

\paragraph{Proof of equations~\eqref{eq:k-mixture-covering-gradient} and~\eqref{eq:k-mixture-covering-hessian}} 

Given a positive number $\varepsilon'$ to be determined later, let
$\{\theta_1, \theta_2, \cdots, \theta_M\}$ be a minimal
$\varepsilon'$-covering of the parameter space $\ball
(\thetastar_\sigma, R_0)$.  By standard volume arguments, we have that
$M \leq \big(\frac{c}{\varepsilon'} \big)^{K d}$.

We bound the $L^2 (P_n)$ covering number by studying the Lipschitz
constant for the functions in these classes. Note that for each
$\theta \in \ball (\thetastar_\sigma, R_0)$, simple derivation yields
that:
\begin{align*}
    \vecnorm{\nabla_\theta \inprod{\nabla \log p_\theta (X_i)}{v_j}}{2} 
    & \leq c \sum_{k' = 1}^K \vecnorm{\gausscenterstar_{k'} - X_i}{2}^2 + c K (R_0^2 + 1),\quad \mbox{and}\\
     \vecnorm{\nabla_\theta \inprod{\nabla^2 \log p_\theta (X_i) v_j}{w_\ell}}{2} 
     & \leq c K \sum_{k' = 1}^K \vecnorm{\gausscenterstar_{k'} - X_i}{2}^3 + c K^2 (R_0^3 + 1).
\end{align*}
Note furthermore that $\bar{G}^{(1)} (x) \geq 1$ and $\bar{G}^{(2)} (x) \geq 1$ by definition. 
By taking $\varepsilon' \mydefn \frac{t}{c \sum_{k' = 1}^K \vecnorm{\gausscenterstar_{k'} - X_i}{2}^2 + c K (R_0^2 + 1)}$, 
the set $\{p_{\theta_i}: i \in [M]\}$ constitutes a $t$-packing of the set $\mathcal{G}^{(1)}_{v, j}$. We, therefore, have the bound
\begin{align*}
    \bar{N}^{(1)} (t) \leq \abss{\mathcal{N}  \left( \mathcal{G}^{(1)}_{v, j}, \vecnorm{\cdot}{L^{\infty} (P_\numobs)}, t  \right)} 
    \leq \left( \frac{c \sum_{k' = 1}^K \vecnorm{\gausscenterstar_{k'} - X_i}{2}^2 + c K (R_0^2 + 1)}{t} \right)^{K d}.
\end{align*}
Similarly, we have that
\begin{align*}
    \bar{N}^{(2)} (t) \leq \abss{\mathcal{N}  \left( \mathcal{G}^{(2)}_{v, w, j, \ell}, \vecnorm{\cdot}{L^{\infty} (P_\numobs)}, t  \right)} 
    \leq \left( \frac{c K \sum_{k' = 1}^K \vecnorm{\gausscenterstar_{k'} - X_i}{2}^3 + c K^2 (\localradius^3 + 1)}{t} \right)^{K d},
\end{align*}
which proves the claim in equations~\eqref{eq:k-mixture-covering-gradient} and~\eqref{eq:k-mixture-covering-hessian}.

\subsubsection{Proof of the claim~\eqref{eq:k-mixture-far-tail}}

For the global condition, we use an argument slightly different from
the third condition in~\cref{cor:final-nonconvex}. Note that for
$\theta = [\gausscenter_j]_{j \in [K]}$, the log-likelihood function
takes the form
\begin{align*}
    \loglihood_\numobs \big( [\gausscenter_j]_{j \in [K]} \big) 
    = \frac{1}{\numobs} \sum_{i = 1}^\numobs \log \left( \sum_{j = 1}^K \exp \big( - \vecnorm{\gausscenter_j - X_i}{2}^2 / 2 \big) \right) - \frac{\log(2 \pi) d}{2} - \log (K).
\end{align*}

Given $\DataX$, we denote $R_X \mydefn \max_{i \in [\numobs]}
\vecnorm{X_i}{2}$, and define the compact set
\begin{subequations}
\begin{align}
 U(r) \mydefn \left\{ [\gausscenter_j]_{j \in [K]}: \vecnorm{u_j}{2}
 \leq r \quad \mbox{for all $j \in [K]$} \right\}.
\end{align}
Now, we claim that for all $t > 0$
\begin{align}
   \label{eq:tail-at-infinity-k-mixtures}  
  \posterior \left( U^{c} (3R_X + \sqrt{d + t}) \mid \DataX \right) \leq
  \left( c (R_X + 1) \sqrt{\numobs} + \sqrt{\smoothprior}
  \right)^{d/2} \prior_0^{-1} e^{- t / 2}.
\end{align}
\end{subequations}
Taking this claim as given, by choosing $t = c d \log (R_X \numobs +
\smoothprior) + c \log \frac{1}{\prior_0 \vartheta}$, we have the tail
bound $\posterior \left( \ball^{c} (0, 3 \sqrt{K} R_X + \sqrt{K (d +
  t)}) \mid \DataX \right) < \vartheta$. As a consequence, we obtain the conclusion of claim~\eqref{eq:k-mixture-far-tail}.


\subsubsection{Proof of the claim~\eqref{eq:tail-at-infinity-k-mixtures}}

Given $\theta = [\gausscenter_j]_{j \in [K]}$, if $\theta \notin U (3
R_X)$, there exists $j_0 \in [K]$ such that
$\vecnorm{\gausscenter_{j_0}}{2} > 3 R_X$. We note that for each $i
\in [\numobs]$, we have:
\begin{align*}
    \exp \left( - \frac{1}{2} \vecnorm{\gausscenter_{j_0} - X_i}{2}^2
    \right) & < \exp \left( - \frac{1}{2} (3 R_X - \vecnorm{ X_i}{2})^2
    \right) \\
    & \leq \exp \left( - \frac{1}{2} (2 R_X)^2 \right) < \exp
    \left( - \frac{1}{2} \vecnorm{X_1 - X_i}{2}^2 \right).
\end{align*}
Therefore, we obtain
\begin{align*}
    \loglihood_\numobs ([\gausscenter_j]_{j \in [K]}) < \loglihood_\numobs \left([\gausscenter_1, \cdots, \gausscenter_{j_0 - 1}, X_1, \gausscenter_{j_0 + 1}, \cdots, \gausscenter_K] \right).
\end{align*}
Consequently, we can replace any $\gausscenter_j$ whose norm is larger than $3 R_X$ with $X_1$, and increase the log-likelihood function. 
The global maximum of the $\loglihood_\numobs$ is therefore attained only in the set $U (3 R_X)$.

On the other hand, for any $\thetamap_\sigma \in \arg\max_{\theta \in U (3 R_X)} \loglihood_\numobs (\theta)$, we have that
\begin{align*}
    \opnorm{\nabla^2 \loglihood_\numobs (\theta)} \leq \frac{1}{\numobs} 
    \sum_{i = 1}^\numobs \left( \vecnorm{u_j - X_i}{2}^2 + 1 \right) \leq 16 R_X^2 + 1.
\end{align*}
Taking the local radius $r \mydefn \frac{1}{(4 R_X + 1) \sqrt{n} } \wedge \frac{1}{\sqrt{\smoothprior}}$, we have the lower bound
\begin{align*}
    \int_{\ball (\thetamap_\sigma, r)} e^{n \loglihood_\numobs (\theta)} \prior (d \theta) 
    & \geq \prior_0 \cdot \mathrm{Vol} \big( \ball (\thetamap_\sigma, r) \big) e^{\numobs \loglihood_\numobs (\thetamap_\sigma)} \cdot e^{- \smoothprior r^2 / 2} \cdot e^{- \numobs (16 R_X^2 + 1) r^2 / 2}\\
    &\geq 4 \prior_0 \big( c r \big)^d e^{\numobs \loglihood_\numobs \big( \thetamap_\sigma \big)},
\end{align*}
for a universal constant $c > 0$.

On the other hand, for any $t > 0$, we have that:
\begin{align*}
    \int_{\ball^{c} (\thetamap_\sigma, 3 R_X + \sqrt{d + t})} e^{n \loglihood_\numobs (\theta)} \prior (d \theta) 
    \leq e^{\numobs \loglihood_\numobs (\thetamap_\sigma)}  \int_{\ball^{c} (\thetamap_\sigma, t \sqrt{d})} 
    \prior (d \theta) \leq e^{\numobs \loglihood_\numobs (\thetamap_\sigma)} \cdot e^{- t / 2}.
\end{align*}
Consequently, we have the upper bound on the posterior tail probability:
\begin{align*}
    \posterior \left( \ball (0, 3R_X + \sqrt{d + t}) \mid \DataX\right) 
    & \leq \left(\int_{\ball (\thetamap_\sigma, r)} e^{n \loglihood_\numobs (\theta)} \prior (d \theta) \right)^{-1} \\
    & \hspace{6 em} \times \int_{\ball^{c} (\thetamap_\sigma, 3 R_X + \sqrt{d + t})} e^{n \loglihood_\numobs (\theta)} \prior (d \theta)\\
    &\leq  \left( c (R_X + 1) \sqrt{\numobs} + \sqrt{\smoothprior} \right)^{d/2} \prior_0^{-1} e^{- t / 2},
\end{align*}
for some universal constant $c > 0$. Therefore, we obtain the conclusion of claim~\eqref{eq:tail-at-infinity-k-mixtures}. 

\subsection{Proof of~\cref{cor:ibragimov}} \label{subsec:proof-ibragimov}

We first invoke~\cref{thm:local-weak-convex} in a small local
neighborhood of $\thetastar = 0$.  We claim that there exist constants
$q_1, q_2, q_3, R_0- > 0$ that depend on the density function $f$ but
independent of $\numobs$ and $a_\numobs$, such that:
\begin{subequations}
\begin{align}
    - \inprod{\theta}{\nabla \smoothloglihood (\theta)} 
    & \geq q_1 |\theta|^{1 + 2 \beta} - q_2 a_\numobs^{1 + 2 \beta}, 
    \quad \mbox{for } \theta \in (- \localradius/2, \localradius/2), \label{eq:singular-models-geometry-bound}\\
    \sup_{\theta \in [-1, 1]} \abss{\nabla \smoothloglihood (\theta) - \nabla \smoothloglihood_\numobs (\theta)} 
    & \leq q_3 \left( a_\numobs^{\beta - 1/2} \sqrt{\frac{\log \numobs / \delta}{\numobs}} 
    + a_\numobs^{\beta - 1} \frac{\log \numobs / \delta}{\numobs} \right), \label{eq:singular-models-deviation-bound}
\end{align}
\end{subequations}
with probability $1 - \delta$.  Assume that the above claims are given
at the moment (their proofs are given in
Appendices~\ref{subsec:local_singularity}
and~\ref{subsec:deviation_bound_singularity}).
Invoking~\cref{thm:local-weak-convex}, there exists a universal
constant $c > 0$, such that:
\begin{align*}
    \widetilde{\posterior} \left( \ball (\thetastar, \localradius / 4)
    \mid \DataX\right)^{-1} \widetilde{\posterior} \left( \ball
    (\thetastar,c \cdot r_\numobs) \mid \DataX \right) \geq 1 -
    \vartheta,
\end{align*}
where the scalar $r_\numobs$ is given by
\begin{align*}
r_\numobs = q_2^{\frac{1}{1 + 2 \beta}} a_\numobs + \left(\frac{1 +
  \log (1/\vartheta)}{\numobs} \right)^{\frac{1}{1 + 2 \beta}} +
\left( \frac{q_3}{q_1} \cdot \Big( a_\numobs^{\beta - 1/2}
\sqrt{\frac{\log (\numobs/\delta)}{\numobs}} + a_\numobs^{\beta - 1}
\frac{\log (\numobs/\delta)}{\numobs} \Big) \right)^{\frac{1}{2
    \beta}} & \\
    & \hspace{-5 em} + \left( \frac{1}{q_3 \numobs} \right)^{\frac{1}{2
    \beta}}.
\end{align*}
Taking $a_\numobs = \numobs^{- \frac{1}{1 + 2 \beta}}$, we conclude
that
\begin{align}
\label{eq:singular-density-models-contraction-rate}  
 r_\numobs \leq q' \cdot \numobs^{- \frac{1}{1 + 2 \beta}} \left(
 \log^{\frac{1}{2 \beta}} (\numobs/\delta) + \log^{\frac{1}{1 + 2
     \beta}} (1/\vartheta) \right),
\end{align}
where the constant $q' > 0$ depends on $q_1, q_2, q_3$ and $\beta$.

It remains to lower bound the posterior probability in a small ball
$\ball (0, \localradius / 4)$.  We claim that there exists a constant
$\gap > 0$ depending on $\localradius$ and $f$, such that there exists
a constant $q_0 > 0$ depending on the function $f$ and the quantities
$\localradius, \gap$, when $\numobs \geq q_0 \log^{ 1 + \frac{1}{2
    \beta}} \delta^{-1}$, the following bound holds true with
probability $1 - \delta$:
\begin{align}
    \label{eq:singular-models-global-gap-bound}  
    \sup_{\theta \in \ball (0, 1) \setminus \ball (0, \localradius /
      4)} \smoothloglihood_\numobs (\theta) < \inf_{|\theta'| <
      a_\numobs}\smoothloglihood (\theta') - \frac{1}{2} \gap.
\end{align}

Taking this bound as given, we proceed with the proof of this
corollary. In order to bound the smoothed posterior probability
outside the ball $\ball(0, \localradius/ 4)$, we note that:
\begin{align*}
& \hspace{-6 em} \widetilde{\posterior} \left( \ball (0, 1) \setminus
  \ball(0, \localradius/ 4) \mid \DataX \right) \\
& \leq \widetilde{\posterior} \left( \ball (0, a_\numobs)\mid \DataX
  \right)^{-1} \widetilde{\posterior} \left( \ball (0, 1) \setminus
  \ball(0, \localradius/ 4) \mid \DataX \right)\\
& \leq \frac{ \sup_{\theta \in \ball (0, 1) \setminus \ball (0,
      \localradius / 4)} \exp \left( \numobs \smoothloglihood_\numobs
    (\theta) \right)} {2 a_\numobs \cdot \inf_{|\theta| < a_\numobs}
    \exp \left( \numobs \smoothloglihood (\theta) \right) \cdot
    \inf_{|\theta| < a_\numobs} \prior (\theta)} \\
& \leq \frac{1}{2 a_\numobs \prior (0) e^{- \boundconsprior}} \cdot
  \exp \left( - \frac{\gap \numobs}{2} \right).
\end{align*}
Given $\numobs \geq \frac{2}{\gap} \left( \boundconsprior + c \log
\frac{\numobs}{\vartheta \prior (0)} \right)$, for a prior density
$\prior$ supported on the interval $[-1, 1]$, we have
that
\begin{align*}
\widetilde{\posterior} \left( \ball(0, \localradius/ 4)^C \mid \DataX
\right) = \widetilde{\posterior} \left( \ball (0, 1) \setminus
\ball(0, \localradius/ 4) \mid \DataX \right) \leq \vartheta.
\end{align*}
Therefore, for $\numobs \geq q_0 \log^{1 + \frac{1}{2 \beta}}
\delta^{-1}$, we conclude that
\begin{align*}
    \widetilde{\posterior} \left( \ball (0, c r_\numobs) \mid \DataX
    \right) \geq 1 - 2 \vartheta
\end{align*}
with probability at least $1 - \delta$, where $r_\numobs$ was defined
in equation~\eqref{eq:singular-density-models-contraction-rate}.

\subsubsection{Proof of claim~\eqref{eq:singular-models-global-gap-bound}}

We first prove the result for the original population-level
log-likelihood $\loglihoodsing$, and then show that the smoothing does
not affect the gap up to constant factors.  Finally we show the
sample-level version using the deviation
bound~\eqref{eq:singular-models-deviation-bound}.

Denote the density function $f_\theta(x) \mydefn f (x - \theta)$.  We
note that $$\loglihoodsing (0) - \loglihoodsing (\theta) =
\kull{f}{f_\theta} \geq 0$$ for any $\theta \in \real$.  Furthermore,
the function $\loglihoodsing$ is continuous in the interval $[-1, 1]$.
On the compact set $[-1, - \localradius/ 4] \cup [\localradius / 4,
  1]$, the maximum point of the continuous function $\loglihoodsing$
is attainable, i.e.,
\begin{align*}
    \exists \theta_0 \in [-1, - \localradius/ 4] \cup [\localradius /
      4, 1], \quad \mathrm{s.t.}~ \loglihoodsing (\theta_0) =
    \sup_{\theta \in [-1, - \localradius/ 4] \cup [\localradius / 4,
        1]} \loglihoodsing (\theta).
\end{align*}
Since $f \neq f_{\theta_0}$, we have that $\kull{f}{f_{\theta_0}} >
0$.  We define $\gap \mydefn \frac{1}{2} \kull{f}{f_{\theta_0}}$.

On the other hand, since the function $\loglihoodsing$ is continuous
on the compact set $[-2, 2]$, by Heine-Cantor theorem,
$\loglihoodsing$ is also uniformly continuous on $[-2, 2]$, i.e.,
\begin{align*}
    \lim_{\delta \rightarrow 0^+} \sup_{\theta, \theta' \in [-2, 2], ~
      \abss{\theta - \theta'} \leq \delta} \abss{F(\theta) -
      F(\theta')} = 0.
\end{align*}
So there exists $\delta_0 > 0$, such that when $\abss{\theta -
  \theta'} < \delta_0$ for some $\theta, \theta' \in [-2, 2]$, we have
that $$\abss{F(\theta) - F(\theta')} \leq \frac{1}{2} \gap.$$
Consequently, for $\numobs$ large enough such that $a_\numobs =
\numobs^{- \frac{1}{1 + 2 \beta}} < \delta_0 / 2$, for any $\theta \in
       [-1, - \localradius/ 4] \cup [\localradius / 4, 1]$, we have
       the following bound:
\begin{align*}
    \smoothloglihood (\theta) & \leq \loglihoodsing (\theta) +
    \sup_{\theta' \in [\theta - a_\numobs, \theta + a_\numobs]}
    \abss{F(\theta) - F(\theta')} \leq \loglihoodsing (\theta_0) +
    \frac{1}{2} \gap \\ & \leq \loglihoodsing (0) - 2 \gap +
    \frac{1}{2} \gap \\ & \leq \inf_{|\theta'| < a_\numobs}
    \smoothloglihood (\theta') - \frac{3}{2} \gap + \sup_{\theta' \in
      [- a_\numobs, a_\numobs]} \abss{F(0) - F(\theta')} \\ & \leq
    \inf_{|\theta'| < a_\numobs}\smoothloglihood (\theta') - \gap.
\end{align*}

For the sample version, we note that the
bound~\eqref{eq:singular-models-deviation-bound} implies the following
inequality with probability $1 - \delta$:
\begin{align*}
    & \hspace{- 6 em} \sup_{\theta, \theta' \in [-1, 1]}
  \abss{\big(\smoothloglihood_\numobs (\theta) -
    \smoothloglihood_\numobs (\theta')\big) - \big(\smoothloglihood
    (\theta) - \smoothloglihood (\theta')\big)}\\ &\leq \sup_{\theta
    \in [-1, 1]} \int_{\theta'}^\theta \abss{\nabla
    \smoothloglihood_\numobs (s) - \nabla \smoothloglihood (s)} ds\\ &
  \leq q_3 \numobs^{- \frac{2 \beta}{1 + 2 \beta}} \log
  \frac{\numobs}{\delta}.
\end{align*}
Given $\numobs \geq c \left( \frac{q_3}{\gap} \log(1/\delta)
\right)^{1 + \frac{1}{2 \beta}}$, we have that:
\begin{align*}
    \sup_{\theta, \theta' \in [-1, 1]}
    \abss{\big(\smoothloglihood_\numobs (\theta) -
      \smoothloglihood_\numobs (\theta')\big) - \big(\smoothloglihood
      (\theta) - \smoothloglihood (\theta')\big)} \leq \frac{1}{4}
    \gap.
\end{align*}
Combining with the population-level bound, we obtain the following
bound with probability $1 - \delta$:
\begin{align*}
    \sup_{\theta \in [-1, - \localradius/ 4] \cup [\localradius / 4,
        1]} \smoothloglihood_\numobs (\theta) \leq \inf_{|\theta'| <
      a_\numobs}\smoothloglihood (\theta') - \frac{1}{2} \gap,
\end{align*}
which proves the desired claim.

\subsubsection{Local structure of $\smoothloglihood$}
\label{subsec:local_singularity}

Now we prove claim~\eqref{eq:singular-models-geometry-bound}.  We
first analyze the local structure of $\loglihoodsing$, and then study
the effect of smoothing. For $\theta > 0$, direct calculation yields:
\begin{align}
    -\nabla_\theta \loglihoodsing (\theta) &= \int_{- \infty}^{+
      \infty} f (x) \nabla \log f (x - \theta) dx \nonumber \\ &=
    \int_{- \infty}^{+ \infty} (f(x + \theta) - f (x)) \nabla \log f
    (x) dx \nonumber \\ &= \int_{- \infty}^{+ \infty} \int_0^\theta f
    (x + z) (\nabla \log f (x + z) \cdot \nabla \log f (x)) dz dx =
    I_1 + I_2 + I_3 + I_4. \label{eq:singularity_key_equation}
\end{align}
where the terms $I_1$, $I_2$, $I_3$ and $I_4$ are defined as follows:
\begin{align*}
    I_1 (\theta)&\mydefn \beta^2 \int_{- \infty}^{\infty}
    \int_{0}^{\theta} f (x + z) \ell (x) \ell (x + z) |x|^{\beta - 1}
    |x + z|^{\beta - 1} \mathrm{sgn} (x (x + z)) dz dx\\ I_2
    (\theta)&\mydefn \beta^2 \int_{- \infty}^{\infty}
    \int_{0}^{\theta} f (x + z) \ell (x + z) |x + z|^{\beta - 1}
    \nabla \log h (x) \mathrm{sgn} (x + z) dz dx,\\ I_3
    (\theta)&\mydefn \beta^2 \int_{- \infty}^{\infty}
    \int_{0}^{\theta} f (x + z) \ell (x) |x|^{\beta - 1} \nabla \log h
    (x + z) \mathrm{sgn} (x) dz dx,\\ I_4 (\theta)& \mydefn \beta^2
    \int_{- \infty}^{\infty} \int_{0}^{\theta} f (x + z) \nabla \log h
    (x) \cdot \nabla \log h (x + z) dz dx.
\end{align*}
For the term $I_1$, we note that:
\begin{align*}
    I_1 & = \theta^{2 \beta} \ell (0^+)^2 \beta \int_1^{+ \infty}
    y^{\beta - 1} (y^\beta - (y - 1)^{\beta}) f (\theta y) dy \\ &
    \quad \quad + \theta^{2 \beta} \ell (0^-)^2 \beta \int_0^{+\infty}
    y^{\beta - 1} ((y + 1)^\beta - y^\beta) f (- \theta y) dy\\ &
    \quad \quad + \theta^{2 \beta} \beta \ell (0^+) \int_0^1 y^{\beta
      - 1}( \ell (0^+) y^\beta - \ell (0^-) (1 - y)^\beta ) f (\theta
    y) dy\\ & \geq \theta^{2 \beta} \beta \int_0^1 \ell (0^+)^2
    y^{\beta - 1} y^\beta f (\theta y) dy + \int_0^{1} \ell (0^-)^2
    y^{\beta - 1} ((y + 1)^\beta - y^\beta) f (- \theta y) dy\\ &
    \qquad- \int_0^1 \ell (0^+) \ell (0^-) y^{\beta - 1} (1 - y)^\beta
    f (\theta y) dy.
\end{align*}
Since the function $f$ is continuous at point $0$, and $f (0) > 0$, there exists $r' > 0$, such that:
\begin{align*}
    \forall s \in (- r', r'), \quad |f (s) - f (0)| \leq \frac{1}{10}
    f (0).
\end{align*}
For $\theta \in (0, r')$, we have:
\begin{align*}
    \int_0^1 \ell (0^+) \ell (0^-) y^{\beta - 1} (1 - y)^\beta f
    (\theta y) dy & \leq \frac{11}{10} \ell (0^+) \ell (0^-) f(0)
    \int_0^1 y^{\beta - 1} (1 - y)^\beta dy\\ &\leq \left( \frac{4}{5}
    \ell (0^+)^2 + \frac{121}{320} \ell (0^-)^2 \right) f(0) \int_0^1
    y^{\beta - 1} (1 - y)^{\beta} dy,\\ \int_0^{1} \ell (0^-)^2
    y^{\beta - 1} ((y + 1)^\beta - y^\beta) f (- \theta y) dy & \geq
    \frac{9}{10} \ell (0^-)^2 f(0) \int_0^1 y^{\beta - 1} ((1 +
    y)^\beta - y^\beta) dy,\\ \int_0^1 \ell (0^+)^2 y^{\beta - 1}
    y^\beta f (\theta y) dy &\geq \frac{9}{10} \ell(0^+)^2 f (0)
    \int_0^1 y^{2 \beta - 1} dy.
\end{align*}
Note that $y^{\beta - 1} (1 - y)^\beta \leq y^{\beta - 1} \left( (1 +
y)^\beta - y^\beta \right) + y^{2 \beta - 1}$.  Therefore, for $\theta
\in (0, r')$, we have the following lower bound on $I_1$:
\begin{align*}
    I_1 \geq \frac{\beta \left( \ell (0^+)^2 + \ell (0^-)^2 \right) f (0)}{10}\theta^{2 \beta}
\end{align*}
On the other hand, we can also deduce the following upper bound on
$I_1$ from above expression:
\begin{align*}
    |I_1| &\leq \theta^{2 \beta} \left( b^2 \int_1^{+ \infty} y^{\beta
      - 1} (y^\beta - (y - 1)^{\beta}) f (\theta y) dy + a^2
    \int_0^{+\infty} y^{\beta - 1} ((y + 1)^\beta - y^\beta) f (-
    \theta y) dy \right) \\ & \hspace{18 em} +\theta^{2 \beta} b
    \int_0^1 y^{\beta - 1}( b y^\beta + a (1 - y)^\beta ) f (\theta y)
    dy\\ &\leq \theta^{2 \beta} \sup_{z \in \real} f (z) \cdot \biggr(
    \int_1^{+ \infty} (b^2 y^{\beta - 1} + a^2 (y - 1)^{\beta - 1} )
    (y^\beta - (y - 1)^{\beta}) dy \\ & \hspace{18 em} + b \int_0^1
    y^{\beta - 1}( b y^\beta + a (1 - y)^\beta ) dy \biggr)\\ &\leq
    M_1 \theta^{2 \beta},
\end{align*}
for a constant $M_1 < + \infty$ depending on $a, b$ and $\beta$.

Note that above arguments holds true also on the side $\theta
\rightarrow 0^-$. We therefore have the following lower bound for
$\theta \in (-r' , r')$:
\begin{align}
    \frac{\beta (a^2 + b^2) f (0)}{10} |\theta|^{2 \beta} \leq I_1
    (\theta) \leq M_1 |\theta|^{2 \beta}. \label{eq:bound_I1}
\end{align}

Now we bound each of $I_2, I_3, I_4$ respectively, and show that they
are of order $O (\theta)$, as $\theta \rightarrow 0^+$.  For the term
$I_2 (\theta)$, it is easy to see by definition that $I_2 (0) = 0$,
and by the Lebesgue differentiation theorem, we have that:
\begin{align*}
    \abss{\frac{d I_2}{d \theta} (\theta)} &\leq \beta^2 \int_{-
      \infty}^{\infty} f (x + \theta) \abss{\ell (x + \theta)} \abss{x
      + \theta}^{\beta - 1} \cdot \abss{\nabla \log h (x)} dx\\ &\leq
    (|\ell^+ (0)| + |\ell^- (0)|) \cdot \biggr( \int_{-1}^1 f (y)
    \abss{y}^{\beta - 1} \cdot |\nabla \log h (y - \theta)| dy
    \\ & \hspace{16 em} + \int_{-\infty}^\infty f (y) |\nabla \log h
    (y - \theta)| dy \biggr).
\end{align*}
Invoking the assumption~\eqref{eq:assume-lip-log-h-in-singular-models}
on $h$, we have that:
\begin{align*}
    \abss{\frac{d I_2}{d \theta} (\theta)} &\leq 4 (|a| + |b|)
    \singconstone e^{|a| + |b| + \singconstone } f (0) \int_{-1}^1
    |y|^{\beta - 1} dy + \singconstone (|a| + |b|) .\\ &\leq \frac{8
      (|a| + |b|)}{\beta} \singconstone e^{|a| + |b| + \singconstone}
    f (0) + (|a| + |b|) \singconstone =: M_2
\end{align*}
So for $|\theta|\leq 1$, we have that:
\begin{align}
    |I_2 (\theta)| \leq M_2 |\theta|. \label{eq:bound_I2}
\end{align}
Similarly, for the term $I_3$, when $|\theta| \leq 1$, we have:
\begin{align*}
    \abss{\frac{d I_3}{d \theta} (\theta)} &\leq (|a| + |b|)\int_{-
      \infty}^{+ \infty} f (x + \theta) |x|^{\beta - 1} \abss{\nabla
      \log h (x + \theta)} dx\\ &\leq (|a| + |b|) \left( \int_{-1}^1 f
    (0) e^{\singconstone} |x|^{\beta - 1} \singconstone dx +
    \singconstone \right).\\ &\leq \frac{8 (|a| + |b|)}{\beta}
    \singconstone e^{|a| + |b| + \singconstone} f (0) + (|a| + |b|)
    \singconstone =: M_3,
\end{align*}
and consequently, for $\theta \in [-1,1]$, we have the bound
\begin{align}
    |I_3 (\theta)| \leq M_3 |\theta|. \label{eq:bound_I3}
\end{align}
For the last term $I_4$, simple calculation yields:
\begin{align}
    |I_4 (\theta)| &\leq |\theta| \cdot \Exs_{f} \left[ \sup_{z \in
    |[0, \theta]} \nabla \log h (X)| \cdot |\nabla \log h (X + z)|
    |\right] \nonumber \\ &\leq \singconstone^2 |\theta| =: M_4
    ||\theta|. \label{eq:bound_I4}
\end{align}
for $\theta \in [-1, 1]$.

We define $\localradius \mydefn \min \left(r', 1, \big(\tfrac{ \beta
  (a^2 + b^2) f (0)}{20 (M_2 + M_3 + M_4)}\big)^{\frac{1}{1 - 2
    \beta}} \right)$.  Plugging the
bounds~\eqref{eq:bound_I1}-\eqref{eq:bound_I4} to
equation~\eqref{eq:singularity_key_equation}, for $|\theta| <
\localradius$, we have:
\begin{align*}
    \inprod{\theta}{\nabla \loglihoodsing (\theta)} &\geq \frac{\beta
      (a^2 + b^2)}{10} |\theta|^{1 + 2 \beta} - (M_2 + M_3 + M_4)
    |\theta|^2\\ &\geq \frac{\beta (a^2 + b^2)}{20} |\theta|^{1 +
      2\beta}.
\end{align*}
Given the smoothing radius $a_n < r_0/2$, for any $\theta \in (a_n, r_0 - a_n)$, we have that:
\begin{align}
    \inprod{\theta}{\nabla \smoothloglihood (\theta)} & = \frac{1}{2
      a_n} \int_{- a_n}^{a_n} \frac{\theta}{\theta + z} \cdot (\theta
    + z) \cdot \nabla \loglihoodsing (\theta + z) dz \nonumber\\ &
    \geq \frac{1}{2 a_n} \int_{- a_n}^{a_n} \frac{\theta}{\theta + z}
    \cdot \frac{\beta (a^2 + b^2)}{20} (\theta + z)^{1 + 2\beta} d z
    \nonumber\\ & \geq \frac{\beta (a^2 + b^2)}{20} \theta \cdot
    (\theta - a_n)^{2
      \beta}.\label{eq:population-unsmoothed-bound-in-singular-models}
\end{align}
For $\theta \in (0, a_n)$, we note that:
\begin{align*}
    \inprod{\theta}{\nabla \smoothloglihood (\theta)} &\geq \frac{1}{2
      a_n} \left( \int_{- \theta}^{a_n} \frac{\theta}{\theta + z}
    \cdot (\theta + z) \cdot \nabla \loglihoodsing (\theta + z) dz -
    \int_{-a_n}^{- \theta} |\theta| \cdot \abss{\nabla \loglihoodsing
      (\theta + z)} dz \right)\\ &\geq - \frac{1}{2a_n} \cdot a_n
    \theta \cdot \sup_{|z| \leq a_n} \left( |I_1 (\theta)| + |I_2
    (\theta)| + |I_3 (\theta)| + |I_4 (\theta)| \right)\\ &\geq - 2
    M_1 a_n^{1 + 2 \beta}.
\end{align*}
We can observe that similar bounds also hold true in the intervals $(-
r_0 + a_n, - a_n)$ and $(- a_n, 0)$.  Therefore, we conclude that the
following bound holds true within the interval $(\localradius / 2,
\localradius / 2)$:
\begin{align*}
    \inprod{\theta}{\nabla \smoothloglihood (\theta)} = \begin{cases}
      - 2 M_1 a_n^{1 + 2 \beta}, & 0 \leq r \leq a_n,\\ \frac{\beta
        (a^2 + b^2)}{20} \cdot (r - a_n)^{1 + 2 \beta} - 2 M_1 a_n^{1
        + 2 \beta}, & a_n \leq r \leq r_0 / 2,
    \end{cases}
\end{align*}


\subsubsection{Bounding the difference $\nabla \smoothloglihood - \nabla \smoothloglihood_n$}
\label{subsec:deviation_bound_singularity}

Now, we proceed to prove
claim~\eqref{eq:singular-models-deviation-bound}. By definition, we
note that
\begin{align*}
\frac{d}{d \theta} \smoothloglihood_n (\theta) & = \frac{1}{n} \sum_{i
  = 1}^n \frac{1}{2 a_n} \left( \log f (\theta + a_n - X_i) - \log f
(\theta - a_n - X_i) \right) \\
& = \frac{1}{n} \sum_{i = 1}^n \frac{1}{2 a_n} \left( |\theta + a_n -
X_i|^\beta \ell (\theta + a_n - X_i) - |\theta - a_n - X_i|^\beta \ell
(\theta + a_n - X_i) \right)\\ &\quad \quad+ \frac{1}{n} \sum_{i =
  1}^n \frac{\log h (\theta + a_n - X_i) - \log h (\theta - a_n -
  X_i)}{2 a_n}.
\end{align*}
We define the following function:
\begin{align*}
 \eta_\theta (x) &\mydefn \frac{\log h (\theta + a_n - x) - \log h
   (\theta - a_n - x)}{2 a_n}, \quad \mbox{and}\\ \nu_\theta (x)
 &\mydefn \frac{1}{2 a_n} \left( |\theta + a_n - x|^\beta \ell (\theta
 + a_n - x) - |\theta - a_n - x|^\beta \ell (\theta + a_n - x) \right)
\end{align*}
We also define the following random variable:
\begin{align*}
    Z^{(1)}_\numobs (\theta) \mydefn \frac{1}{\numobs} \sum_{i =
      1}^\numobs \eta_\theta (X_i) - \Exs \left[ \eta_\theta (X)
      \right], \quad \mbox{and} \quad Z^{(2)}_\numobs (\theta) \mydefn
    \frac{1}{\numobs} \sum_{i = 1}^\numobs \nu_\theta (X_i) - \Exs
    \left[ \nu_\theta (X) \right].
\end{align*}


\subsubsection{Upper bounds for the term $Z^{(1)}_\numobs$}

By the Lipschitz
assumption~\eqref{eq:assume-lip-log-h-in-singular-models}, we have
\begin{align*}
    \abss{\eta_\theta (x)} \leq \frac{1}{2 a_n}\int_{\theta - a_n -
      x}^{\theta + a_n - x} \abss{\nabla \log h (t)} dt \leq
    \singconstone.
\end{align*}
By the Hoeffding bound, for any given $\theta \in [-1, 1]$ and $t >
0$, we obtain
\begin{align*}
    \Prob \left( |Z_\numobs^{(1)} (\theta)| > t \right) \leq 2 \exp
    \left( - \frac{2 \numobs t^2}{\singconstone^2} \right).
\end{align*}
On the other hand, for $\theta_1, \theta_2 \in [-1, 1]$, we note that:
\begin{align*}
    \abss{\eta_{\theta_1} (x) - \eta_{\theta_2} (x)} \leq
    \singconstone \frac{|\theta_1 - \theta_2|}{2 a_n},
\end{align*}
which implies that $\abss{Z^{(1)}_\numobs (\theta_1) - Z^{(1)}_\numobs
  (\theta_1)} \leq \singconstone\frac{|\theta_1 - \theta_2|}{2 a_n}$
almost surely.

Let $\mathcal{M}_\numobs \mydefn \{\theta_1, \theta_2, \cdots,
\theta_{K}\}$ be a maximal $\frac{a_\numobs}{\singconstone
  \numobs}$-packing of the interval $[-1, 1]$. By union bound, we find
that
\begin{align*}
    \Prob \left( \exists \theta \in \mathcal{M}_\numobs ,
    \abss{Z_\numobs^{(1)} (\theta)} \geq t \right) \leq
    2|\mathcal{M}_\numobs| e^{- \frac{2 \numobs
        t^2}{\singconstone^2}}.
\end{align*}
Consequently, for any $\delta > 0$, we have the following uniform
upper bound with probability $1 - \delta$:
\begin{align*}
    \sup_{\theta \in [-1, 1]} \abss{Z_\numobs^{(1)} (\theta)} \leq
    \frac{2}{\numobs} + \singconstone \sqrt{\frac{1}{\numobs} \log
      \frac{|\mathcal{M}_\numobs|}{\delta}} \leq \frac{2}{\numobs} + 3
    \singconstone \sqrt{\frac{1}{\numobs} \log
      \frac{\numobs}{\delta}}.
\end{align*}

\paragraph{Upper bounds for the term $Z^{(2)}_\numobs$}

We first study moment bounds of the random variable $\nu_\theta
(X_i)$. For $p \geq 2$, we have
\begin{align*}
    \Exs \left[ \abss{\nu_\theta (X_i)}^p \right] = \frac{1}{(2
      a_n)^p} \int_{- \infty}^{+ \infty} \abss{ | z + a_n|^\beta \ell
      (z + a_n) - |z - a_n|^\beta \ell (z - a_n) }^p f (\theta - z)
    dz.
\end{align*}
Define $S \mydefn \ell(0^-) + \ell (0^+)$, which is positive. To upper
bound the integral, we split it into three terms:
\begin{align*}
    \bar{I}_1 (\theta) &\mydefn \int_{- 3 a_n}^{3 a_n} |\nu_\theta
    (\theta - z)|^p f (\theta - z) dz,\\ \bar{I}_2 (\theta) &\mydefn
    \int_{-1 - 3 a_n}^{- 3 a_n} |\nu_\theta (\theta - z)|^p f (\theta
    - z) dz + \int_{3 a_n}^{1 + 3 a_n} |\nu_\theta (\theta - z)|^p f
    (\theta - z) dz,\\ \bar{I}_3 (\theta) &\mydefn \int_{- \infty}^{-1
      - 3 a_n} |\nu_\theta (\theta - z)|^p f (\theta - z) dz + \int_{1
      + 3 a_n}^{+\infty} |\nu_\theta (\theta - z)|^p f (\theta - z)
    dz.
\end{align*}
For the term $\bar{I}_1(\theta)$, we can simply take upper bounds on
each term of $\nu_\theta (\theta - z)$, and obtain
\begin{align*}
    \bar{I}_1 (\theta) \leq (6 a_n)^{(\beta - 1)p + 1} S^p \cdot
    \sup_{z \in [-a_n, a_n]} f (\theta - z).
\end{align*}
For the term $\bar{I}_2 (\theta)$, note that
\begin{align*}
    & \hspace{- 3 em} \int_{3 a_n}^{1 + 3 a_n} |\nu_\theta (\theta -
  z)|^p f (\theta - z) dz \\ & = (2 a_n)^{- p} \ell (0^+)^p \int_{2
    a_n}^{1 + 2 a_n} z^{p \beta} \left( \left(1 + \frac{2 a_n}{z}
  \right)^\beta - 1 \right)^p f (\theta - z - a_n) dz\\ &\leq (2
  a_n)^{- p} \ell (0^+)^p \int_{2 a_n}^{1 + 2 a_n} (2\beta a_n
  z^{\beta - 1})^p f (\theta - z - a_n) dz\\ &\leq \frac{\beta^p}{p (1
    - \beta) - 1}\ell (0^+)^p (2 a_n)^{1 + p (\beta - 1)} \sup_{z \in
    [- 1 - 3 a_n, 1 + 3 a_n]} f (\theta - z).
\end{align*}
For the integral within the interval $[- 1 - 3 a_n, - 3 a_n]$, we have
a similar upper bound. Putting them together, we obtain
\begin{align*}
    \bar{I}_2 (\theta) \leq \frac{S^p}{p (1 - \beta) - 1} (2 a_n)^{1 +
      p (\beta- 1)} \sup_{z \in [- 1 - 3 a_n, 1 + 3 a_n]} f (\theta -
    z).
\end{align*}
For the last term $\bar{I}_3 (\theta)$, we note that for $|z| > 1 +
a_n$, there is
\begin{align*}
    \abss{|z + a_n|^\beta - |z - a_n|^\beta} \leq \int_{z - a_n}^{z +
      a_n} \beta |z - s|^{\beta - 1} ds \leq 2 \beta a_n \leq a_n.
\end{align*}
Consequently, we have
\begin{align*}
    \bar{I}_3 (\theta) \leq (2a_n)^{- p} \left(\int_{- \infty}^{-1 - 3
      a_n} + \int_{- \infty}^{-1 - 3 a_n} \right) a_n^p f (\theta - z)
    dz \leq 1.
\end{align*}
Combining the above upper bounds of $\bar{I}_{1}(\theta)$,
$\bar{I}_{2}(\theta)$, and $\bar{I}_{3}(\theta)$, for any $a_n < 1$
and $p \geq 2$, we obtain that
\begin{align*}
    \left( \Exs |\nu_\theta (X_i)|^p \right)^{\frac{1}{p}} \leq
    C\left( \frac{ \sup_{z \in [-4, 4]} f (\theta - z)}{p(1 - \beta) -
      1} \right)^{\frac{1}{p}} a_n^{\beta - 1 + \frac{1}{p}},
\end{align*}
for a universal constant $C > 0$.

Invoking Bernstein inequality, for any fixed $\theta \in [-1, 1]$ and
$t > 0$, we have
\begin{align*}
    \Prob \left( \abss{\frac{1}{n} \sum_{i = 1}^n \nu_\theta (X_i) -
      \Exs \brackets{\nu_\theta (X)}} > t \right) \leq \exp \left(
    \frac{- n t^2 / 2}{Q a_n^{2 \beta - 1} + a_n^{\beta - 1} t / 3 }
    \right),
\end{align*}
where $Q \mydefn C \frac{ \sup_{z \in [-4, 4]} f (\theta - z)}{1 - 2
  \beta} $ for universal constant $C > 0$.

Now we extend the concentration inequality for a fixed $\theta$ to the
uniform bound for any $\theta \in [-1, 1]$.  Recall that $Z_n^{(2)}
(\theta) \mydefn \frac{1}{n} \sum_{i = 1}^n \nu_\theta (X_i) - \Exs
[\nu_\theta (X)]$.  For $\theta_1, \theta_2 \in [-1, 1]$, we have
\begin{align*}
    & \hspace{- 2 em} |Z_n^{(2)} (\theta_1) - Z_n^{(2)} (\theta_2)|
  \\ & \leq \frac{1}{n a_n} \sum_{i = 1}^n \abss{|\theta_1 + a_n -
    X_i|^\beta \ell (\theta_1 + a_n - X_i) - |\theta_2 + a_n -
    X_i|^\beta \ell (\theta_2 + a_n - X_i)}\\ & + \frac{1}{n a_n}
  \sum_{i = 1}^n \abss{|\theta_1 - a_n - X_i|^\beta \ell (\theta_1 -
    a_n - X_i) - |\theta_2 - a_n - X_i|^\beta \ell (\theta_2 - a_n -
    X_i)}\\ & \leq \frac{2 S}{a_n} |\theta_1 - \theta_2|^\beta,\quad
  \mathrm{a.s.}
\end{align*}
where we use the inequality $\abss{|x|^\beta - |y|^\beta} \leq |x -
y|^\beta$ for any $x, y$.

Let $b_n \mydefn \left( \frac{a_n}{2 S n} \right)^{\frac{1}{\beta}}$
and $\mathcal{M}_{b_n}$ be a maximal $b_n$-packing of the interval
$[-1, 1]$. For any $\theta \in [-1, 1]$, there exists $\theta' \in
\mathcal{M}_{b_n}$, such that $|\theta - \theta'| < b_n$, which
implies that $|Z_n^{(2)} (\theta) - Z_n^{(2)} (\theta')| <
\frac{1}{n}$. Consequently, for any $t > 0$, we find that
\begin{align*}
    \Prob \left(\sup_{\theta \in [-1, 1]} |Z_n^{(2)} (\theta)| > t +
    \frac{1}{n} \right) & \leq \Prob \left( \sup_{\theta \in
      \mathcal{M}_{b_n}} |Z_n^{(2)} (\theta)| > t \right) \\ & \leq
    |\mathcal{M}_{b_n}|\exp \left( \frac{- n t^2 / 2}{Q a_n^{2 \beta -
        1} + a_n^{\beta - 1} t / 3 } \right).
\end{align*}
Given $a_n > \tfrac{S}{n^2}$, we have $\log |\mathcal{M}_{b_n}| \leq
\frac{3}{\beta} \log n$.  Choosing appropriate value of $t$, we have
\begin{align*}
    \sup_{\theta \in [-1, 1]} |Z_n^{(2)} (\theta)| \leq C \left(
    \sqrt{Q} a_n^{\beta - 1/2} \sqrt{\frac{\log n / \delta}{n}} +
    a_n^{\beta - 1} \frac{\log n / \delta}{n} + \frac{1}{n} \right),
\end{align*}
with probability $1 - \delta$.

Collecting the bounds for the terms $Z^{(1)}_\numobs$ and
$Z^{(2)}_\numobs$, for $a_\numobs \in \big( \tfrac{S}{\numobs^2}, 1
\big)$, we conclude the following bound that holds true with
probability $1 - \delta$:
\begin{align*}
    \sup_{\theta \in [-1, 1]}\abss{\nabla_\theta
      \smoothloglihood_\numobs (\theta) -
      \nabla_\theta\smoothloglihood (\theta) } \leq c \cdot \left(
    \sqrt{Q} a_n^{\beta - 1/2} \sqrt{\frac{\log n / \delta}{n}} +
    a_n^{\beta - 1} \frac{\log n / \delta}{n} \right),
\end{align*}
for a universal constant $c > 0$.

\section{Proofs of the remaining auxiliary results}
\label{subsec:auxiliary_results}

In this appendix, we provide proofs of the remaining auxiliary results
in the paper.

\subsection{Proof of~\cref{prop:keybound}}
\label{subsec:proof-prop-keybound}

For any $p \geq 2$, we define the quantity:
\begin{align*}
    R_p \mydefn \sup_{p \geq 0} \left( \Exs_{\pi_t} \left[
      \vecnorm{X}{2}^p \right] \right)^{1/p} \vee \left( \Exs_{\pi^*}
    \left[ \vecnorm{X}{2}^p \right] \right)^{1/p}
\end{align*}

For any given value $\bar{R} > 0$, we note the following
decomposition:
\begin{align*}
    &\abss{\Exs_{\pi_t} \left[ \vecnorm{X}{2}^p \right] - \Exs_{\pi^*} \left[ \vecnorm{X}{2}^p \right]}\\
    &\leq \int_{\ball(0, \bar{R})} |\pi_t - \pi^*| \cdot \vecnorm{x}{2}^p dx + \int_{\ball(0, \bar{R})^C} \pi_t (x) \vecnorm{x}{2}^p dx + \int_{\ball(0, \bar{R})^C} \pi^* (x) \vecnorm{x}{2}^p dx\\
    &\leq \bar{R}^p \cdot \totalvariation(\pi_t, \pi^*) + \Exs_{\pi_t} \left[  \vecnorm{X}{2}^p \bm{1}_{\vecnorm{X}{2} > \bar{R}} \right] + \Exs_{\pi^*} \left[  \vecnorm{X}{2}^p \bm{1}_{\vecnorm{X}{2} > \bar{R}} \right]\\
    &\leq  \bar{R}^p \cdot \totalvariation(\pi_t, \pi^*) + \sqrt{\Exs_{\pi_t} \left[  \vecnorm{X}{2}^{2p} \right]}  \sqrt{\pi_t \left( \vecnorm{X}{2} > \bar{R} \right)} + \sqrt{\Exs_{\pi^*} \left[  \vecnorm{X}{2}^{2p} \right]}  \sqrt{\pi^* \left( \vecnorm{X}{2} > \bar{R} \right)} \\
    &\leq  \bar{R}^p \cdot \totalvariation(\pi_t, \pi^*) + 2 R_{2p}^{p} \cdot R_2 / \bar{R}.
\end{align*}
For any $\varepsilon > 0$, take $\bar{R} \mydefn \frac{\varepsilon}{2 R_{2p}^p R_2}$, we have that:
\begin{align*}
    \lim_{t \rightarrow + \infty} \abss{\Exs_{\pi_t} \left[ \vecnorm{X}{2}^p \right] - \Exs_{\pi^*} \left[ \vecnorm{X}{2}^p \right]} \leq \varepsilon,
\end{align*}
which proves the claim.

\subsection{Proof of~\cref{cor:final-nonconvex}} 
\label{subsec:proof:cor:final-nonconvex}
  
It follows from~\cref{thm:local-weak-convex} that
\begin{align}
   \posterior \left(\ball (\thetastar_j, r_n^{(j)}) \mid \DataX
   \right) \leq \posterior \left(\ball (\thetastar_j, \localradius)
   \mid \DataX \right) \cdot (1 - \vartheta), \quad \forall \ j \in
        [M].\label{eq:final-nonconvex-part1}
\end{align}
It remains to prove a lower bound on the sum $\sum_{j = 1}^M
\posterior (\ball (\thetastar_j, \localradius) \mid \DataX)$.  We
utilize the following lemma, which controls the tail behavior of
posterior distribution in an unbounded space.
\begin{lemma}\label{lemma:tail-far-away}
    Under the condition C.3 for~\cref{cor:final-nonconvex}, for any
    $\vartheta > 0$, we have that:
    \begin{align}
        \posterior \left( \ball \big( 0, R ( \vartheta) \big) \mid
        \DataX \right) \geq 1 - \vartheta, \quad \mbox{where} \quad R
        ( \vartheta) \mydefn 2 R_\delta \log \vartheta^{-1} +
        \sqrt{\frac{6 (d + \log
            \vartheta^{-1})}{c_\prior}}.\label{eq:final-nonconvex-part2}
    \end{align}
\end{lemma}

Taking this lemma as given, we proceed with the proof of the
corollary.  First, by the empirical process assumption within the ball
$\ball \big(0, R ( \vartheta) \big)$, for the sample size satisfying
$\widebar{\noise}_{\numobs, \delta} \big(R (\vartheta)\big) <
\frac{1}{4} \Delta_0$, with probability $1 - \delta$, we have the
bound
\begin{align*}
    \sup_{\theta \in \ball (0, R (\vartheta))} \abss{\loglihood
      (\theta) - \loglihood_\numobs (\theta)} \leq \frac{1}{4}
    \Delta_0.
\end{align*}

Denote $F^* \mydefn F (\thetastar_1)$. By the smoothness condition (A)
of the population-level log-likelihood, we denote $\widetilde{r}_0
\mydefn \sqrt{\Delta_0 / 8 \smooth} \wedge \localradius$. Then, we
have that
\begin{align*}
    \min_{j \in [M]} \inf_{\theta \in \ball (\thetastar_j,
      \widetilde{r}_0)} F (\theta) \geq F^* - \frac{1}{4} \smooth
    (\widetilde{r}_0)^2 = F^* - \frac{1}{4} \Delta_0.
\end{align*}
Denote the set $\mathcal{Z} \mydefn \ball \big( 0, R (\vartheta) \big)
\setminus \bigcup_{j \in [M]} \ball (\thetastar_j, \widetilde{r}_0)$.
Applying above inequalities in conjunction with the gap condition on
the log-likelihood, we find that
\begin{align*}
    \loglihood_\numobs (\theta) \geq \loglihood_\numobs (\theta') +
    \frac{\Delta_0}{4}, \quad \forall \theta \in \bigcup_{j \in [M]}
    \ball (\thetastar_j, \widetilde{r}_0), ~ \theta' \in \mathcal{Z}.
\end{align*}
Consequently, we obtain
\begin{align*}
   \frac{\posterior \left( \theta \in \bigcup_{j \in [M]} \ball
     (\thetastar_j, \widetilde{r}_0) \mid \DataX \right)}{\posterior (
     \theta \in \mathcal{Z} \mid \DataX)} \geq \prior \left(
   \bigcup_{j \in [M]} \ball (\thetastar_j, \widetilde{r}_0) \right)
   \cdot e^{\frac{n \Delta_0}{4}}.
\end{align*}
The prior mass can be lower bounded by the prior density condition and
smoothness condition (B): let $j_0 \in [M]$ be an index such that
$\prior(\thetastar_j) \geq \prior_0$, we have
\begin{align*}
 \prior \left( \bigcup_{j \in [M]} \ball (\thetastar_j,
 \widetilde{r}_0) \right) \geq \prior \left( \ball (\thetastar_{j_0},
 \widetilde{r}_0) \right) \geq \prior_0 \cdot e^{-
   \frac{\smoothprior}{2} \widetilde{r}_0^2} \cdot \mathrm{Vol} \big(
 \ball (\thetastar_{j_0}, \widetilde{r}_0) \big) \geq \prior_0 \cdot
 e^{- \frac{\smoothprior}{2} \widetilde{r}_0^2} \cdot \big(
 \widetilde{r}_0 / \sqrt{d} \big)^d.
\end{align*}
Therefore, given the sample size satisfying the condition:
\begin{align*}
 \numobs \geq \frac{4}{\Delta_0} \left( \log \vartheta^{-1} + \log
 \prior_0^{-1} + \smoothprior \widetilde{r}_0^2 + d \log
 \frac{d}{\widetilde{r}_0} \right),
\end{align*}
we have the lower bound:
\begin{multline}
\posterior \left( \ball \big(0, R (\vartheta)\big) \right)^{-1}
\sum_{j = 1}^M \posterior \left( \ball (\thetastar_j, \localradius)
\mid \DataX \right) \geq 1 - \posterior \left( \ball \big(0, R
(\vartheta)\big) \right)^{-1} \posterior \left( \mathcal{Z} \mid
\DataX \right) \\
\label{eq:final-nonconvex-part3}
\geq 1 - \posterior \Big( \bigcup_{j \in [M]} \ball (\thetastar_j,
\widetilde{r}_0) \mid \DataX \Big)^{-1} \posterior \left( \mathcal{Z}
\mid \DataX \right) \geq 1 - \vartheta.
\end{multline}
Collecting the
bounds~\eqref{eq:final-nonconvex-part1},~\eqref{eq:final-nonconvex-part2},
and~\eqref{eq:final-nonconvex-part3}, we arrive at the lower bound
\begin{align*}
    \posterior \left( \bigcup_{j \in [M]} \ball (\thetastar_j,
    r_n^{(j)}) \; \bigg| \; \DataX \right) \geq (1 - \vartheta)^3,
\end{align*}
which completes the proof of this corollary.


\subsubsection{Proof of~\cref{lemma:tail-far-away}}

For the simplicity of presentation, we make the following argument
conditionally on $X_1^n$.  The diffusion
process~\eqref{eq-diffusion-main} has $\posterior (\cdot | X_1^n)$ as
its stationary distribution. Given $p \geq 4$, we take the potential
function as:
\begin{align*}
    \Phi (\theta) \mydefn \max \left( \vecnorm{\theta}{2} - R_\delta,
    0 \right)^{p}.
\end{align*}
By It\^{o}'s formula, for $T \geq 0$, we have the expansion
\begin{align*}
    \Exs \left[ \Phi (\theta_T) \right] &= \frac{p}{2}
    \underbrace{\int_0^T \Exs \left[ \inprod{\nabla F_n
          (\theta_t)}{\theta_t} \cdot \max \left( \vecnorm{\theta}{2}
        - R_\delta, 0 \right)^{p - 2} \right] dt}_{\mydefn I_1}\\ & +
    \frac{p}{2n} \underbrace{\int_0^T\Exs \left[ \inprod{\nabla \log
          \prior (\theta_t)}{\theta_t} \cdot \max \left(
        \vecnorm{\theta}{2} - R_\delta, 0 \right)^{p - 2} \right]
      dt}_{\mydefn I_2} \\ & + \frac{p}{2n} \underbrace{\int_0^T \Exs
      \left[ \Big( \max(\vecnorm{\theta_t}{2} - R_\delta, 0)^2 d +
        \vecnorm{\theta_t}{2}^2 (p - 1) \Big) \max
        (\vecnorm{\theta_t}{2} - R_\delta, 0)^{p - 4} \right]
      dt}_{I_3}.
\end{align*}
By condition~\eqref{eq:condition-for-nonconvex-tail-loglihood} on the
log-likelihood function, we have that $I_1 \leq 0$.

For the term $I_{2}$,
condition~\eqref{eq:condition-for-nonconvex-tail-prior} implies the
following upper bound:
\begin{align*}
    I_2 \leq - c_\prior \int_0^T \Exs \left[\vecnorm{\theta_t}{2}^2
      \cdot \max \left( \vecnorm{\theta_t}{2} - R_\delta, 0 \right)^{p
        - 2} \right] dt \leq - c_\prior \int_0^T \Exs \left[ \Phi
      (\theta_t) \right] dt.
\end{align*}
The term $I_3$ can be decomposed into two parts:
\begin{align*}
    I_3 &\leq (p + d) \int_0^T \Exs \left[ \max \left(
      \vecnorm{\theta_t}{2} - R_\delta, 0 \right)^{p - 2} \right] dt +
    p R_\delta^2 \int_0^T \Exs \left[ \max \left(
      \vecnorm{\theta_t}{2} - R_\delta, 0 \right)^{p - 4} \right] dt
    \\ &= (p + d) \int_0^T\Exs \left[ \Phi (\theta_t)^{\frac{p -
          2}{p}} \right] dt + p R_\delta^2 \int_0^T \Exs \left[ \Phi
      (\theta_t)^{\frac{p - 4}{p}} \right] dt\\ &\leq (p + d) \int_0^T
    \Big( \Exs \left[ \Phi (\theta_t)\right] \Big)^{\frac{p - 2}{p}}
    dt + p R_\delta^2 \int_0^T \Big( \Exs \left[ \Phi (\theta_t)
      \right] \Big)^{\frac{p - 4}{p}} dt.
\end{align*}
Denote $\psi_t \mydefn  \Exs \left[ \Phi (\theta_t) \right] $, we have the integral inequality
\begin{align*}
    \psi_T &\leq \frac{p}{2n} \int_0^T \left( - c_\prior \psi_t + (p +
    d) \psi_t^{\frac{p - 2}{p}} + p R_\delta^2 \psi_t^{\frac{p -
        4}{p}} \right) dt \\ &\leq \frac{p}{2n} \int_0^T \left( -
    c_\prior \psi_t + \frac{c_\prior}{3} \psi_t + \frac{(p +
      d)^{\frac{p}{2} }}{(c_\prior / 3)^{\frac{p - 2}{2}}} +
    \frac{c_\prior}{3} \psi_t + \frac{(p R_\delta)^{\frac{p}{4}
    }}{(c_\prior / 3)^{\frac{p - 4}{4}}} \right) dt\\ &\leq \frac{p
      c_\prior}{6n} \int_0^T \left( - \psi_t dt + \left( \frac{6 (p +
      d)}{c_\prior} \right)^{\frac{p}{2}} + (p R_\delta)^p \right) dt.
\end{align*}
Note that $\psi_0 = 0$ by definition. Applying Gr\"{o}nwall
inequality, we arrive at the bound
\begin{align*}
    \psi_t \leq \left( \frac{6 (p + d)}{c_\prior}
    \right)^{\frac{p}{2}} + (p R_\delta)^p, \quad \forall t \geq 0.
\end{align*}
Consequently, we have the bound
\begin{align*}
    \left( \Exs \left[ \vecnorm{\theta}{2}^p \mid \DataX \right]
    \right)^{1/p} &\leq R_\delta + \lim\sup_{t \rightarrow + \infty}
    \left( \Exs \left[ \Phi (\theta_t) \right] \right)^{1/p}\\ &\leq
    (p + 1) R_\delta + \sqrt{\frac{6 (p + d)}{c_\prior}},
\end{align*}
which completes the proof of~\cref{lemma:tail-far-away}.


\subsection{A limit result}

We begin with a lemma on the limiting behavior of a certain type of
function. The lemma is used in the proof
of~\cref{theorem-main-weakly-convex}
in~\Cref{subsection:weakly_convex_proof}.
\begin{lemma}
  \label{lemma-integral-ineq-limit-control}
Let $\phi$ be a non-increasing continuous function on the real line
with $\phi(c) = 0$, and such that $\phi(t) \geq 0$ for all $t \in (c,
\infty)$.  Suppose that there exist two continuous functions $f, g:
      [0, +\infty) \rightarrow \real$ such that $\lim_{t \rightarrow +
          \infty} g(t)$ exists and $f (t) \leq \int_0^t \phi(g(s)) ds$
        for all $t \geq 0$.  Under these conditions, we have $\lim_{t
          \rightarrow + \infty} g(t) \leq c$.
\end{lemma}
\begin{proof}
Define the limit $A \mydefn \lim_{t \rightarrow +\infty} g(t)$, which
exists according to the assumptions.  We proceed via proof by
contradiction.  In particular, suppose that $A > c$.  Based on the
definition of $A$, for the positive constant $\varepsilon = (A - c)/2
>0$, we can find a sufficiently large positive constant $T$ such that
$g(t) > A - \varepsilon$ for any $t \geq T$.  According to the
assumptions on $\phi$, we obtain that
\begin{align*}
  \delta \mydefn - \sup_{s \geq c + \varepsilon} \phi(s) < 0.
\end{align*}
Therefore, for all $t > T$, we arrive at the following inequalities
   \begin{align*}
       0 \leq f(t) \leq \int_0^T \phi( g(s) ) ds + \int_T^{t} \phi(
       g(s) ) ds \leq \int_0^T \phi( g(s) ) ds - \delta (t - T).
   \end{align*}
By choosing $t = 1 + T + \delta^{-1} \int_0^T \phi( g(s) ) ds$, the
above inequality cannot hold.  This yields the desired contradiction,
which completes the proof.
\end{proof}


\subsection{A tail bound based on truncation}

We now state an upper deviation inequality based on a truncation
argument.  This lemma is used in~\cref{subsec:proof:cor:single_index}
to prove the uniform concentration
bound~\eqref{eq:empi_process_index}.  Consider a sequence of random
variables $\{Y_i\}_{i=1}^\numobs$ satisfying the moment bounds
\begin{align}
  \label{EqnMomentBounds}
  \Exs \brackets{|Y_i|^q} \leq (a q)^{b q} \quad \mbox{for all $q = 1,
    2, \ldots$}
\end{align}
where $a, b$ are universal constants.

\begin{lemma}
\label{lemma-truncated-bernstein}
Given an i.i.d. sequence of zero-mean random variables
$\{Y_i\}_{i=1}^n$ satisfying the moment
bounds~\eqref{EqnMomentBounds}, we have
\begin{align*}
  \Prob \left( \frac{1}{n}\sum_{i = 1}^n Y_i \geq (4a)^b
  \sqrt{\frac{\log 4 / \delta}{n}} + \parenth{a \log
    \frac{n}{\delta}}^b \frac{\log 4 / \delta}{n} \right) \leq \delta.
\end{align*}
\end{lemma}
\begin{proof}
The proof of the lemma is a direct combination of truncation argument
and Bernstein's inequality.  In particular, for each $i \in [n]$,
define the truncated random variable \mbox{$\widetilde{Y}_i \mydefn Y_i
  \Ind \left[ |Y_i| \leq 3 (a \log \frac{n}{\delta})^b \right]$.} With
this definition, we have
\begin{align*}
  \Prob \left( (Y_i)_{i = 1}^n \neq (\widetilde{Y}_i)_{i = 1}^n \right) &
  = \Prob \left( \max_{1 \leq i \leq n} |Y_i| > 3 \parenth{a \log
    \frac{n}{\delta} }^b \right) \\ & \leq n \Prob \left( |Y_i| > 3
  \parenth{a \log \frac{n}{\delta}}^b \right) \leq \frac{\delta}{2}.
\end{align*}
Therefore, it is sufficient to study a concentration behavior of the
quantity $\sum_{i = 1}^n \widetilde{Y}_i$. Invoking Bernstein's
inequality~\cite{Lugosi2016}, we obtain that
\begin{align*}
  \Prob \left( \frac{1}{n}\sum_{i = 1}^n \widetilde{Y}_i \geq \varepsilon
  \right) \leq 2 \exp \left( - \frac{n \varepsilon^2}{2 (2a)^{2b} +
    \frac{2}{3} \varepsilon \cdot 3 (a \log \frac{n}{\delta})^b
  }\right).
\end{align*}
In order to make the RHS of the above inequality less than
$\frac{\delta}{2}$, it suffices to set
\begin{align*}
  \varepsilon = (4a)^b \sqrt{\frac{\log (4 / \delta)}{n}} + \parenth{a
    \log \frac{n}{\delta}}^b \frac{\log (4 / \delta)}{n}.
\end{align*}
Collecting all of the above inequalities yields the claim.
\end{proof}

\bibliographystyle{plain}
\bibliography{reference}

\end{document}